\documentclass[11pt]{article}
\usepackage{a4, amsthm, amsmath, amssymb,amsbsy, epic,graphicx,mathrsfs,enumerate}
\usepackage[all]{xy}
\usepackage{multicol,longtable}
\usepackage{ stmaryrd }
\usepackage{subfigure}
\usepackage{array}

\usepackage{multirow}

\usepackage[margin=2.5cm]{geometry}
\def\width{15.4cm}
\def\wid1{15.7cm}

\numberwithin{equation}{section}
\numberwithin{table}{section}

\newcolumntype{I}{!{\vrule width 1.2pt}}

\newlength\savedwidth
\newcommand\whline{\noalign{\global\savedwidth\arrayrulewidth
                            \global\arrayrulewidth 1.2pt}%
                   \hline
                   \noalign{\global\arrayrulewidth\savedwidth}}

\usepackage{caption}
\DeclareCaptionLabelSeparator{period}{. }
\usepackage[section]{placeins}

\newtheorem{thm}{Theorem}[section]
\newtheorem{lem}[thm]{Lemma}
\newtheorem{cor}[thm]{Corollary}
\newtheorem{prop}[thm]{Proposition}
\newtheorem{conj}[thm]{Conjecture}

\newtheorem{hyp}[thm]{Hypothesis}
\newtheorem{defi}[thm]{Definition}

\theoremstyle{definition}
\newtheorem{remark}[thm]{Remark}
\newtheorem{exam}{Example}[section]

\newcommand{\ls}[2]{{\vphantom{#2}}^{#1}{#2}} 	
\newcommand{\lsa}[2]{{\vphantom{#2}}^{#1\!}{#2}}

\newcommand{\ncase}{\medskip\noindent}
\newcommand{\di}{\,| \,}
\newcommand{\divi}{\; | \;}
\newcommand{\dmod}{\mod \;}
\newcommand{\coleq}{\! :=}
             
\renewcommand{\ker}{\operatorname{ker}}

\newcommand{\Aut}{\operatorname{Aut}}

\newcommand{\ra}{\rightarrow}
\newcommand{\mZ}{\mathbb{Z}}
\newcommand{\mF}{\mathbb{F}}

\newcommand{\mC}{\mathbb{C}}

\newcommand{\mN}{\mathbb{N}}

\newcommand{\Hom}{\operatorname{Hom}}

\newcommand{\Stab}{\operatorname{Stab}}

\renewcommand{\rm}{\mathrm}

\newcommand{\mbf}{\mathbf}
\newcommand{\Cal}{\mathcal}
\newcommand{\End}{\operatorname{End}}
\newcommand{\hra}{\hookrightarrow}

\newcommand{\lda}{\lambda}

\newcommand{\fr}{\mathfrak}

\newcommand{\lra}{\longrightarrow}

\newcommand{\ol}{\overline}

\newcommand{\Irr}{\operatorname{Irr}}
\renewcommand{\a}{\alpha}
\renewcommand{\b}{\beta}
\newcommand{\e}{\epsilon}
\newcommand{\g}{\gamma}
\renewcommand{\d}{\delta}

\renewcommand{\th}{\theta}
\newcommand{\om}{\omega}
\newcommand{\Om}{\Omega}
\newcommand{\Ind}{\operatorname{Ind}}
\newcommand{\Res}{\operatorname{Res}}

\newcommand{\PSL}{\operatorname{PSL}}
\newcommand{\PSU}{\operatorname{PSU}}

\renewcommand{\fr}{\mathfrak}

\newcommand{\Def}{\operatorname{Def}}
\newcommand{\Inf}{\operatorname{Inf}}
\newcommand{\Bl}{\operatorname{Bl}}
\newcommand{\El}{\operatorname{El}}
\newcommand{\Proj}{\operatorname{Proj}}
\newcommand{\D}{\Delta}
\renewcommand{\O}{\mathcal O}
\newcommand{\lan}{\langle}
\newcommand{\ran}{\rangle}
\newcommand{\SU}{\operatorname{SU}}
\newcommand{\GU}{\operatorname{GU}}
\newcommand{\reg}{\mathrm{reg}}
\newcommand{\Char}{\operatorname{char}}
\newcommand{\capm}{\operatorname{cap}}
\newcommand{\SL}{\operatorname{SL}}
\newcommand{\PSp}{\operatorname{PSp}}

\newcommand{\Llra}{\Longleftrightarrow}
\newcommand{\hght}{\operatorname{ht}}

\title{The McKay conjecture and Brauer's induction theorem}
\author{Anton Evseev\thanks{The author is supported by the EPSRC Postdoctoral Fellowship EP/G050244/1.}  \\
\\
School of Mathematics \\
Queen Mary, University of London \\
Mile End Road, London E1 4NS, U.K. \\ 
\texttt{A.Evseev@qmul.ac.uk}
}

\date{} 

\begin{document}
\maketitle
\begin{abstract} Let $G$ be an arbitrary finite group. The McKay conjecture asserts that $G$ and the normaliser $N_G (P)$ of a Sylow $p$-subgroup $P$ in $G$ have the same number of characters of degree not divisible by $p$ (that is, of $p'$-degree). We propose a new refinement of the McKay conjecture, which suggests that one may choose a correspondence between the characters of $p'$-degree of $G$ and $N_G (P)$ to be compatible with induction and restriction in a certain sense. This refinement implies, in particular, a conjecture of Isaacs and Navarro.
We also state a corresponding refinement of the Brou{\'e} abelian defect group conjecture. We verify the proposed conjectures in several special cases.
\end{abstract}

\section{Introduction}\label{intro}

\subsection{Refinements of the McKay conjecture}\label{refinements}

Let $G$ be a finite group and $p$ a prime. As usual, $\Irr(G)$ will denote the set of complex irreducible characters of $G$. We write $\Irr_{p'}(G)$ for the set of characters $\chi\in \Irr(G)$ such that $\chi(1)$ is not divisible by $p$.

The McKay conjecture is one of the most intriguing open problems in representation theory of finite groups. The most straightforward version of the conjecture asserts that if $P$ is a Sylow $p$-subgroup of $G$ then $|\Irr_{p'}(G)|=|\Irr_{p'}(N_G (P))|$. The conjecture was proposed by McKay~\cite{McKay1971},~\cite{McKay1972} in a special case and was investigated more generally by Isaacs~\cite{Isaacs1973}.
Alperin~\cite{Alperin1976} stated the conjecture in the form given above and, in fact, proposed a stronger version in terms of blocks 
(cf.\ Section~\ref{blarb} below).
The conjecture has been proved in many cases, in particular, for $p$-solvable groups (for the prime $p$), symmetric and alternating groups, sporadic simple groups and many classes of finite groups of Lie type. In fact, Isaacs, Malle and Navarro~\cite{IsaacsMalleNavarro2007} have reduced the conjecture to verifying a (fairly complicated) set of conditions for each finite simple group, and there has been considerable progress towards proving that those conditions are satisfied via the classification of finite simple groups 
(see e.g.~\cite{Malle2008}, \cite{Spaeth2009}, \cite{Spaeth2010}, \cite{BrunatHimstedt2010} and references therein). 

However, there seems to be no known general approach that might lead to a proof of 
the conjecture for all finite groups and all primes simultaneously.
It appears that, in general, no bijection between $\Irr_{p'} (G)$ and $\Irr_{p'} (N_G (P))$ can be described as ``canonical'' (in an informal sense). However, there is a number of natural decompositions of $\Irr_{p'}(G)$ into a union of disjoint subsets $X_1,\ldots, X_n$ such that there is a corresponding decomposition $\Irr_{p'} (N_G (P))=\sqcup_{i=1}^n Y_i$. One can then ask whether $|X_i|=|Y_i|$ for each $i$.

Several such refinements of the McKay conjecture have been proposed. The most relevant of these to the present paper is probably the Isaacs--Navarro conjecture~\cite{IsaacsNavarro2002}, which may be stated as follows. Fix a prime $p$ and, for an integer $l$, write
$$
M_l (G) = |\{\chi \in \Irr(G) \mid \chi(1) \equiv \pm l \bmod p \}|.
$$

\begin{conj}[\cite{IsaacsNavarro2002}, Conjecture A]\label{conjIN}
If $P$ is a Sylow $p$-subgroup of $G$ then $M_l (G)=M_l (N_G(P))$ for every integer $l$ coprime to $p$.
\end{conj}

Clearly, Conjecture~\ref{conjIN} is equivalent to the McKay conjecture when $p=2$ or $3$.
In this paper we propose a stronger refinement, which, if true, would give more information than the McKay conjecture for every prime. Most of the paper is devoted to proving that refinement in certain special cases.

 We denote by $\Cal C (G)$ the abelian group $\mZ\Irr(G)$ of \emph{virtual characters} of $G$. If $H$ is a subgroup of $G$, we have the induction and restriction maps 
$\Ind_H^G\colon \Cal C(H) \ra \Cal C(G)$ and $\Res^G_H\colon \Cal C(G) \ra \Cal C(H)$.

\begin{defi}\label{defX} Let $P$ be a $p$-subgroup of $G$ and $H$ be a subgroup of $G$. We define the~\emph{intersection set} 
$\Cal S(G,P,H)$ by 
$$
\Cal S (G,P,H) = \{ Q\le P \mid  Q \subseteq \lsa{t}{P} \text{ for some } t\notin H \}.
$$
\end{defi}

We will only consider cases where $H$ contains the normaliser $N_G (P)$. Note that then $\Cal S(G,P,H)$ is one of the sets appearing in the definition of the Green correspondence (cf.\ e.g.~\cite{CRI}, Definition 20.4).

\begin{defi}\label{defI} Let $P$ be a $p$-subgroup of $G$.
Suppose $\Cal S$ is a family of subgroups of $P$ such that $\Cal S$ is downward closed (i.e.\ closed under taking subgroups). 
 Then $\Cal I(G,P,\Cal S)$ is defined to be the subgroup of $\Cal C (G)$
spanned by the class functions of the form $\Ind_{L}^G \phi$ where $\phi\in \Cal C(L)$ and $L$ is a subgroup of $G$ such that
\begin{enumerate}[(i)]
\item $L\cap P$ is a Sylow $p$-subgroup of $L$; and 
\item $L\cap P\in \Cal S$.
\end{enumerate} 
\end{defi}

Our starting point is the following observation.

\begin{thm}\label{corrthm}
Let $P$ be a Sylow $p$-subgroup of $G$, and let $H\le G$ contain $N_G (P)$. Set $\Cal S = \Cal S(G,P,H)$. Then
\begin{enumerate}[(i)]
 \item $\Ind_H^G (\Cal I(H,P,\Cal S))\subseteq \Cal I(G,P,\Cal S)$;
\item $\Res_H^G (\Cal I(G,P,\Cal S)) \subseteq \Cal I(H,P,\Cal S)$;
\item the maps $\Res_H^G$ and $\Ind_H^G$ yield mutually inverse isomorphisms between the abelian groups $\Cal C(G)/\Cal I(G,P,\Cal S)$ and 
$\Cal C(H)/\Cal I(H,P,\Cal S)$.
\end{enumerate}
\end{thm}

In Section~\ref{gensetup} we state and prove a more general version of this result (Theorem~\ref{corrthmgen}).
Brauer's classical induction theorem is a key ingredient of the proof. 

Let $H=N_G (P)$ where $P$ is a Sylow $p$-subgroup of $G$. The McKay conjecture predicts the existence of a one-to-one correspondence between $\Irr_{p'} (G)$ and $\Irr_{p'} (H)$. It is natural to ask whether this correspondence can be chosen in such a way that it ``agrees'' with the isomorphism  of Theorem~\ref{corrthm} between 
$\Cal C(G)/\Cal I(G,P,\Cal S)$ and $\Cal C(H)/\Cal I(H,P,\Cal S)$.

If $S$ is a set of virtual characters, write $\pm S=\{\chi, -\chi \mid \chi\in S \}$. If $S_1$ and $S_2$ are two sets of virtual characters, 
we say that a bijection $F\colon\pm S_1 \ra \pm S_2$ is \emph{signed} if $F(-\chi)= -F(\chi)$ for all $\chi\in \pm S_1$. Whenever $P$ is a Sylow $p$-subgroup of $G$ and $H$ is a subgroup of $G$ containing $N_G (P)$, we will consider whether the pair $(G,H)$ satisfies the following property:

\bigskip
\noindent
\begin{tabular}{lp{\width}}
\!\!\!\!(IRC-Syl)\!\!\! &
\emph{There is a signed bijection $F\colon \pm\Irr_{p'}(G) \ra \pm\Irr_{p'}(H)$ such that
$$
F(\chi) \equiv \Res_H^G \chi \dmod \Cal I(H,P,\Cal S(G,P,H)) \qquad\qquad
$$
for all $\chi\in \pm \Irr_{p'}(G)$.
} \\
\end{tabular}
\medskip

If this holds, we will say that $G$ satisfies (IRC-Syl) with respect to $H$ (and the prime $p$).
In the present paper we largely restrict ourselves to the case $H=N_G (P)$.  However, the author is not aware of counterexamples to the assertions of conjectures stated below if $H$ is instead taken to be an arbitrary subgroup of $G$ containing $N_G (P)$ (note that the set $\Cal S(G,P,H)$ may be smaller than $\Cal S(G,P,N_G (P))$ if $H\supsetneq N_G (P)$). 

\begin{conj}\label{conjA} Let $G$ be a finite group with an abelian Sylow $p$-subgroup $P$.  Then $G$ satisfies property \emph{(IRC-Syl)} with respect to $N_G (P)$.  
\end{conj}

%If $\Cal S(G,P,H)=\{ \mbf 1 \}$ (where $\mbf 1$ is the trivial subgroup) and $P$ is abelian then the equation of property (IRC-Syl) is equivalent to the condition
%$$
%F(\chi)(xy)= \chi(xy) \quad \text{for all } x\in P-\mbf 1, \; y\in C_G (x) \cap N_G (P)
%$$

%Nevertheless, we will see that there is a number of cases when the Sylow subgroup $P$ is not abelian and (IRC-Syl) still holds for $(G,N_G (P))$. 

As we will see in Sections~\ref{TI} and~\ref{calc}, property (IRC-Syl) sometimes fails for the normaliser $H$ of a non-abelian Sylow subgroup $P$. 
When seeking a refinement that might be true in all cases, it seems appropriate to replace $\Cal I(H,P,\Cal S)$ with a larger subgroup of $\Cal C(H)$ in the statement of (IRC-Syl). Let
$$
\Irr^p (G) = \{ \chi\in \Irr(G) \mid \chi(1) \equiv 0 \bmod p \},
$$
and write $\Cal C^p (G) = \mZ \Irr^p (G)$. We will consider three more properties that may or may not hold for a pair $(G,H)$ where, as before,
 $P$ is a Sylow $p$-subgroup of $G$ and $N_G (P)\le H \le G$:

\bigskip
\noindent
\begin{tabular}{lp{\width}}
\text{\!\!\!(pRes-Syl)} & 
\centering{
$
\Res^G_H (\Cal C^p (G)) \subseteq \Cal C^p (H) + \Cal I(H,P,\Cal S(G,P,H)). \qquad\qquad
$}
\end{tabular}

\bigskip
\noindent
\begin{tabular}{lp{\width}}
\text{\!\!\!(pInd-Syl)} & 
\centering{
$
\Ind^G_H (\Cal C^p (H)) \subseteq \Cal C^p (G) + \Cal I(G,P,\Cal S(G,P,H)). \qquad\qquad
$}
\end{tabular}

%\noindent
%\begin{tabular}{lc} A & B \\ \end{tabular}

\bigskip
\noindent
\begin{tabular}{lp{\width}}
\!\!\!(WIRC-Syl)\!\!\!
&
\emph{There is a signed bijection $F\colon \pm\Irr_{p'}(G) \ra \pm\Irr_{p'}(H)$ such that
$$
F(\chi) \equiv \Res_H^G \chi \dmod  \Cal C^p (H) + \Cal I(H,P,\Cal S(G,P,H)) \qquad\qquad
$$
for all $\chi\in \pm \Irr_{p'}(G)$.
} \\
\end{tabular}

\medskip
The acronyms IRC and WIRC stand for \emph{(weak) induction-restriction correspondence}.

Note that if (pRes-Syl) and (pInd-Syl) hold then we have the following analogue of Theorem~\ref{corrthm}: the maps $\Ind^G_H$ and $\Res^G_H$ yield mutually inverse abelian group isomorphisms
$$
\frac{\Cal C(G)}{\Cal C^p (G) + \Cal I(G,P,\Cal S)} \longleftrightarrow \frac{\Cal C(H)}{\Cal C^p (H) + \Cal I(H,P,\Cal S)}
$$
where $\Cal S=\Cal S(G,P,H)$.
In this case (WIRC-Syl) asserts that there exists a signed bijection between $\pm \Irr_{p'}(G)$ and $\pm \Irr_{p'}(H)$ which is compatible with these isomorphisms. 

\begin{conj}\label{conjB} Let $P$ be a Sylow $p$-subgroup of an arbitrary finite group $G$. Then properties \emph{(pRes-Syl)}, \emph{(pInd-Syl)} and \emph{(WIRC-Syl)} hold for $G$ with respect to the normaliser $N_G (P)$.
\end{conj}

Let $P$ be a Sylow $p$-subgroup of $G$ and $H=N_G (P)$. Then $P$ is not an element of the set $\Cal S=\Cal S(G,P,H)$, whence each element $\th$ of 
$\Cal C^p (H) + \Cal I(H,P,\Cal S)$ has a degree divisible by $p$. Suppose (WIRC-Syl) holds for $(G,N_G(P))$ and the map $F\colon\pm \Irr_{p'}(G) \ra \pm \Irr_{p'}(H)$ is a witness to that. Then we have 
$
F(\chi)(1) \equiv \chi(1) \pmod p
$
for each $\chi\in\Irr_{p'}(G)$. The following proposition is now clear.

\begin{prop}\label{wircsylIN}
Let $P$ be a Sylow $p$-subgroup of a finite group $G$. If \emph{(WIRC-Syl)} holds for the pair $(G, N_G (P))$ then Conjecture~\ref{conjIN} holds for the group $G$ and the prime $p$. 
\end{prop}

In Section~\ref{Broue} we state what appears to be a natural refinement of the Brou{\'e} abelian defect group conjecture and will show that Conjecture~\ref{conjA} follows from that refinement. Even if we do not assume $P$ to be abelian, property (IRC-Syl) holds in a number of cases, in particular, for split groups of Lie type if $p$ is the defining characteristic with certain exceptions (see Theorem~\ref{liesplit}) and for the groups $S_n$ and $A_n$ whenever $n\le 11$ (see Section~\ref{calc}). In fact, the only examples that we have been able to find so far
 of simple groups $G$ for which (IRC-Syl) fails  are those where $G$ contains a twisted group of Lie type defined over a field of characteristic $p$. Nevertheless, property (IRC-Syl) fails for some solvable groups, at least for $p=2$ (see Section~\ref{exampleQ8}). To find a statement plausibly approaching a necessary and sufficient condition for (IRC-Syl) seems to be an interesting problem.

Conjecture~\ref{conjB}, particularly property (WIRC-Syl), appears to stand on somewhat less firm ground than Conjecture~\ref{conjA}. As there are comparatively few cases where (IRC-Syl) fails, it is not easy to judge what the ``right'' way to weaken (IRC-Syl) is. 
However, at this time the ``right'' property seems more likely to be stronger rather than weaker than (WIRC-Syl). 
Indeed, Eaton~\cite{Eaton2008}, motivated partially by a result of Navarro~\cite{Navarro2003}, has proposed a somewhat stronger property in the special case $\Cal S(G,P,H)=\{ \mbf 1 \}$ (see Section~\ref{Eaton}). Proposition~\ref{wircsylIN}, the results of~\cite{Eaton2008}  and the results of Sections~\ref{lietype} and~\ref{calc} all seem to indicate that (WIRC-Syl) is a reasonable starting point, at the very least. 

%In particular, this occurs for split finite groups of Lie type, under some conditions (see Section~\ref{lietype}) and for symmetric and alternating %groups on at most $11$ letters (see Section~\ref{calc}). 

The paper is organised as follows. In Section~\ref{gensetup} we generalise Theorem~\ref{corrthm}, the four properties stated above and Proposition~\ref{wircsylIN} to the case where $P$ is an arbitrary $p$-subgroup of $G$. In Section~\ref{splendid} we state the refinement of the Brou{\'e} abelian defect group conjecture mentioned above and verify it in two special cases. In Section~\ref{TI} we consider the case where
$\Cal S=\Cal S(G,P,N_G (P))=\{ \mbf 1 \}$ and observe that property (WIRC-Syl) follows from a property stated by Eaton~\cite{Eaton2008} (called (P+)) in this case. We formulate a strengthening of (WIRC-Syl) based on (P+). Also, we consider an example of a solvable group for which (IRC-Syl) fails for the prime $2$. In Section~\ref{lietype} we prove results concerning the four properties above in the case where $G$ is a finite group of Lie type defined over a field of characteristic $p$. Finally, in Section~\ref{calc} we give results of computer calculations testing our properties for ``small'' groups. In particular, the tables of Section~\ref{calc} give the structure of the abelian groups 
$\Cal C(N_G(P))/\Cal I(N_G(P),P,\Cal S)$ and $\Cal C(N_G(P))/(\Cal C^p(G)+\Cal I(N_G (P),P,\Cal S))$ in a number of cases (where $\Cal S=\Cal S(G,P,H)$), allowing one to get a 
feeling as to how much information the properties (IRC-Syl) and (WIRC-Syl) encode when they hold. For the most part, Sections~\ref{lietype} and~\ref{calc} may be read independently of each other and of Sections~\ref{splendid} and~\ref{TI}.

\subsection{Some notation and conventions}

Most of our notation is standard. If $k<l$ are integers, we sometimes denote by $[k,l]$ the set $\{k, k+1,\ldots,l\}$.

\ncase
\textbf{Groups.} 
We denote by $Z(G)$ the centre of a group $G$. If $S$ is a $G$-set then $S^G$ is the set of points in $S$ fixed by all elements of $G$. 

Now suppose $G$ is a finite group. The \emph{$p$-part} $g_p$ and the \emph{$p'$-part} $g_{p'}$ of an element $g\in G$ are defined uniquely by the conditions that $g=g_p g_{p'}=g_{p'} g_p$, that $g_p$ is a $p$-element and that $g_{p'}$ is a $p'$-element. If $g,h\in G$, we write $[g,h]=g^{-1}h^{-1}gh$. If $L$ is a subgroup of $G$, we write $L_p$ and $L_{p'}$ for the sets $p$-elements and $p'$-elements of $L$ respectively. We denote by $\mbf 1$ the trivial subgroup of $G$. If $X$ and $Y$ are subsets of $G$, we write $X\subseteq_G Y$ if $\lsa{g}{X}=Y$ for some $g\in G$, and we write $X=_G Y$ if $\lsa{g}{X}=Y$ for some $g\in G$. 

\ncase
\textbf{Characters.}
Suppose $\chi\in \Cal C(G)$. We say that $\th\in \Irr(G)$ is an \emph{irreducible constituent} of $\chi$ if the scalar product $\lan \chi,\th \ran$ is non-zero, and we say that this scalar product is the \emph{multiplicity} of $\phi$ in $\chi$. We say that $\chi$ is \emph{multiplicity-free} if all its irreducible constituents occur with multiplicity $1$ or $-1$. We denote by $1_G$ the trivial character of $G$. If $\th\in \Cal C(H)$ for some finite group $H$, we write $\chi\times \th\in \Cal C(G\times H)$ for the ``outer'' product of $\chi$ and $\th$, defined by $(\chi\times \th) (g,h)=\chi(g)\th(h)$.

Let $L$ is a normal subgroup of $G$ and $\phi\in\Irr(L)$. We write $\Irr(G \di \phi)$ for the set of characters $\chi\in \Irr(G)$ such that $\phi$ is a constituent of $\Res^G_L \chi$. We define $\Irr_{p'} (G\di \phi)=\Irr_{p'}(G) \cap \Irr(G\di \phi)$.
If $\chi\in\Cal C(G)$, we define
$$
\pi_{\phi} \chi = \sum_{\xi\in \Irr(G | \phi)} \lan \chi,\xi \ran \xi.
$$
This is the projection of $\chi$ onto the $\mZ$-span of $\Irr(G\di \phi)$.

\ncase
\textbf{Rings.}
Throughout we will denote by $\Cal O$ a complete discrete valuation ring of characteristic $0$ with a maximal ideal $\fr p$ such that the field $k=\Cal O/\fr p$ has characteristic $p$ and is algebraically closed. We will denote by $K$ be the field of fractions of $\Cal O$. We assume that $\Cal O$ is ``large enough'', i.e.\ contains all $|G|$-th roots of unity for all groups $G$ in question. 
Thus we may and do consider elements of $\Cal C (G)$ as class functions with values in $K$ (via an identification of a splitting field for $G$ in $\mC$ with one in $K$.) The dual $\bar\chi$ of a virtual character $\chi\in \Cal C(G)$ is defined by $\bar\chi(g)=g^{-1}$, $g\in G$. If 
$e=\sum_{g\in G} a_g g$ is an idempotent of $Z(KG)$, we write $\bar e =\sum_{g\in G} a_g g^{-1}$. The centre of an algebra $A$ is denoted by $Z(A)$.

\ncase
\textbf{Modules.}
 All modules will be assumed to be finitely generated. Modules are assumed to be left ones unless we specify otherwise. Let $R$ be a commutative ring. An \emph{$RG$-lattice} is an $RG$-module 
which is free as an $R$-module. If $M$ and $N$ are $R$-modules, we 
write $M\otimes N=M\otimes_R N$. We will sometimes write $M^{\oplus n}$ to denote the direct sum of $n$ copies of an $RG$-module $M$. 
If $L$ is a normal subgroup of $G$ and $M$ is 
an $R(G/L)$-module then $\Inf_{G/L}^G M$ is the $RG$-module obtained from $M$ by inflation.
If $M$ is an $\Cal O G$-lattice then $K\otimes_{\Cal O} M$ is a $KG$-module; we say that $M$ \emph{affords} 
the character afforded by $K\otimes_{\Cal O} M$. For finite groups $G$ and $H$, 
we will identify $R(G\times H)$-modules with $RG$-$RH$-bimodules: if $M$ is an $R(G\times H)$-module then its bimodule structure is given by
$gmh=(g,h^{-1})m$, $g\in G$, $h\in H$, $m\in M$.

\bigskip
\noindent
\textbf{Acknowlegdements.} The author is grateful to Robert Boltje for pointing out an error in a previous version of the paper and to C{\'e}dric Bonnaf{\'e} for providing a correct proof of Proposition 14.32 in~\cite{DigneMichelBook} (see Remark~\ref{Bonrem} below).

\section{The general setup}\label{gensetup}

\subsection{Brauer's induction theorem and relatively projective modules}\label{willems}

As indicated above, Theorem~\ref{corrthm} ultimately relies on Brauer's induction theorem. We state the latter result, following~\cite{IsaacsBook}, Chapter 8.

\begin{defi} If $l$ is a prime, a group $E$ is said to be \emph{$l$-elementary} if it is a direct product of an $l$-group and a cyclic $l'$-group. A group is \emph{elementary} if it is $l$-elementary for some prime $l$. If $G$ is a finite group, we denote by $\El(G)$ the set of all elementary subgroups of $G$.
\end{defi}

Note that if $E$ is an elementary group and $p$ is any prime then $E$ is a direct product of a $p$-subgroup and a $p'$-subgroup.

\begin{thm}[Brauer]\label{BrInd} Let $G$ be a finite group, and suppose $\chi\in \Cal C(G)$. Then we can choose a virtual character $\th_E\in \Cal C(E)$ for each 
$E\in \El(G)$ in such a way that 
$$
\chi = \sum_{E\in \El(G)} \Ind_E^G \th_E.
$$
\end{thm}

Let $Q$ be a $p$-subgroup of $G$.  
Recall that, if $R$ is a commutative ring, an $RG$-module is said to be (relatively) \emph{$Q$-projective} if it is a direct summand of $\Ind_{Q}^G M$ for some $RQ$-module $M$.
We will use certain standard facts related to the theory of vertices and sources, which may be found in~\cite{CRI}, Sections 19 and 20, for example.
Note that every irreducible character of $G$ is afforded by some $\Cal OG$-lattice (see, for example,~\cite{ThevenazBook}, Proposition 42.6). We denote by $\Cal A(Q)$ the set of all subgroups of $Q$.

\begin{thm}[Willems~\cite{Willems1979}]\label{relpr} Let $Q$ be a $p$-subgroup of a finite group $G$. Suppose $M$ is a 
 $Q$-projective $\Cal O G$-lattice. 
If $\chi$ is the character afforded by $M$ then $\chi\in \Cal I(G,Q,\Cal A (Q))$.
\end{thm}

In order to highlight the use of Theorem~\ref{BrInd}, we give a more direct and slightly shorter proof than the one in~\cite{Willems1979}.

\begin{proof}  
First we prove the result in the case when $G$ is a direct product of a $p$-group and a $p'$-group, so $G= G_p \times G_{p'}$ and $Q\le G_p$. 
We may assume that $M$ is indecomposable. 
By the hypothesis, $M$ is a direct 
summand of $\Ind_Q^G S$ for some indecomposable $\Cal O Q$-lattice $S$ (which is a source of $M$). It follows from the Krull--Schmidt theorem that $M$ is a summand of $\Ind_{QG_{p'}}^G N$ 
for an indecomposable $\Cal O QG_{p'}$-lattice $N$. However, by Green's indecomposability theorem (see~\cite{CRI}, Corollary 19.23), $\Ind_{QG_{p'}}^G N$ is indecomposable, whence $M=\Ind_{QG_{p'}}^G N$. It follows that $\chi\in \Cal I(G,Q,\Cal A (Q))$ in this case.

Now let $G$ be arbitrary. By Theorem~\ref{BrInd}, we have
$$
1_G = \sum_{E\in \El(G)} \sum_{\phi\in \Irr(E)} n_{\phi} \Ind_{E}^G \phi
$$
for some integers $n_{\phi}$.
Hence
$$
\chi = \sum_{E\in \El(G)} \sum_{\phi\in\Irr(E)} n_{\phi} \Ind_E^G ((\Res^G_E \chi) \phi),
$$
so it suffices to show that $\Ind_E^G ((\Res^G_E \chi) \phi)\in \Cal I(G,Q,\Cal A(Q))$ for all $E$ and $\phi$. 
Let $U$ be an $\Cal O E$-lattice affording $\phi$, and let $V$ be an indecomposable summand of $\Res^G_E M$. 
As $M$ is a summand of $\Ind^G_Q S$ for an $\Cal O Q$-lattice $S$, 
it follows from the Mackey formula that $V$ is a summand of $\Ind_{\ls{g}{Q} \cap E}^E W$ for some $g\in G$ and
some $\Cal O(\ls{g}{Q}\cap E)$-lattice $W$. Thus $V$ is $\ls{g}{Q}\cap E$-projective, and hence so is $V\otimes_{\Cal O} U$. Therefore, 
by the first part of the proof, 
$(\Res^G_E \chi) \phi\in \Cal I(E, \ls{g}{Q}\cap E, \Cal A(\ls{g}{Q}\cap E))$, and it follows that 
$$
\Ind_E^G (\Res_E^G(\chi) \phi) \in \Cal I(G,\ls{g}{Q}, \Cal A(\ls{g}{Q})) =  \Cal I(G,Q,\Cal A(Q)),
$$
as required.
\end{proof}

\subsection{Blocks of arbitrary defect}\label{blarb}

We denote by $\Bl(G)$ the set of $p$-blocks of a finite group $G$, that is, of primitive idempotents of $Z(\Cal OG)$. 
We will repeatedly use standard facts of block theory, which may be found, for example, in~\cite{CRII}, Chapter 7.
In particular, we have a decomposition of $\Cal OG$ as a direct sum of algebras:
$$
\Cal OG\simeq\bigoplus_{b\in \Bl (G)} \Cal OGb.
$$
On tensoring with $K$, this gives rise to
$$
KG \simeq \bigoplus_{b\in \Bl (G)} KGb.
$$
Each class function on $G$ with values in $K$ can be extended to a map $KG \ra K$ by linearity, and we have
$$
\Irr(G) = \bigsqcup_{b\in \Bl (G)} \Irr (G,b),
$$
where $\Irr(G,b)=\{\chi\in\Irr(G) \mid \chi(b) \ne 0 \}$.
We write $\Cal C(G,b)=\mZ \Irr(G,b)$. The projection map $\Proj_b\colon \Cal C(G) \ra \Cal C(G,b)$ is defined in the obvious way: if $\chi=\sum_{b\in \Bl(G)} \chi_b$
where $\chi_b\in \Cal C(G,b)$ for each $b$, then $\Proj_b (\chi)=\chi_b$.

\begin{prop}\label{splitblocks}
Let $G$ be a finite group and $p$ a prime. Suppose $P$ is a $p$-subgroup of $G$ and $\Cal S$ is a downward closed set of subgroups of $P$. Then
$$
\Cal I(G,P,\Cal S) = \bigoplus_{b\in \Bl(G)} (\Cal I(G,P,\Cal S) \cap \Cal C(G,b)).
$$
\end{prop}

\begin{proof}
Let $\xi\in \Cal I(G,P,\Cal S)$. We are to show that $\Proj_b \xi\in \Cal I(G,P,\Cal S)$ for each $b$. Without loss of generality, $\xi=\Ind_L^G \th$ where $L$ is a subgroup of $G$ such that there is a Sylow $p$-subgroup $Q$ of $L$ satisfying 
$Q\le P$ and $Q\in \Cal S$.

Let $W$ be an $\Cal O L$-lattice affording $\th$. Then $V=\Ind_L^G W$ is a $Q$-projective $\Cal O G$-lattice which affords $\xi$. We have
$$
V=\bigoplus_{b\in \Bl(G)} bV,
$$
so each $bV$ is $Q$-projective. Since $bV$ affords $\Proj_b \xi$ for each $b$, we have
$\Proj_b \xi\in \Cal I(G,Q,\Cal A(Q))\subseteq \Cal I(G,P,\Cal S)$ by Theorem~\ref{relpr}.  
\end{proof}

\begin{defi} Suppose $P$ is a $p$-subgroup of a finite group $G$. We denote by $\Irr(G,P)$ the union of the sets $\Irr(G,b)$ taken over the blocks $b$ of $G$ such that some defect group of $b$ is contained in $P$. We also set $\Cal C(G,P) = \mZ \Irr(G,P)$. The projection homomorphism
 $\Proj_P\colon \Cal C(G) \ra \Cal C(G,P)$ is defined by setting
$$
 \Proj_P (\chi)=
\begin{cases}
\chi & \text{if } \chi \in \Irr(G,P), \\
0 & \text{otherwise}
\end{cases}
$$
for $\chi\in \Irr(G)$ and extending this map by $\mZ$-linearity.
\end{defi}

% Now we state and prove a generalised version of Theorem~\ref{corrthm}. 

If $M$ and $N$ are $\Cal OG$-lattices, we write $M \di N$ if $M$ is a direct summand of $N$. Recall that if $\Cal S$ is a set of $p$-subgroups of $G$ then an $\Cal OG$-lattice $M$ is said to be \emph{$\Cal S$-projective} if for every indecomposable summand $N$ of $M$ there exists $S\in \Cal S$ such that $N$ is $S$-projective. 

We can now state and prove a generalisation of Theorem~\ref{corrthm}.

\begin{thm}\label{corrthmgen}
 Let $P$ be a $p$-subgroup of a finite group $G$. Suppose $H$ is a subgroup of $G$ containing $N_G(P)$ and set $\Cal S=\Cal S(G,P,H)$. Then 
\begin{enumerate}[(i)]
 \item\label{i1} $\Ind_H^G (\Cal C(H,P)) \subseteq \Cal C(G,P) + \Cal I(G,P,\Cal S)$;
 \item\label{i2} $\Ind_H^G (\Cal I(H,P,\Cal S)) \subseteq \Cal I(G,P,\Cal S)$;
 \item\label{i3} $\Proj_P \Res_H^G (\Cal I(G,P,\Cal S)) \subseteq \Cal I(H,P,\Cal S)$;
\item\label{i4} the maps $\Proj_P \Res^G_H$ and $\Ind^G_H$ yield mutually inverse isomorphisms between the abelian groups
$$
\frac{\Cal C(G,P)+\Cal I(G,P,\Cal S)}{\Cal I(G,P,\Cal S)} \quad \text{and} \quad \frac{\Cal C(H,P)}{\Cal I(H,P,\Cal S)\cap \Cal C(H,P)}.
$$
\end{enumerate}
\end{thm}

If $H$ is a subgroup of $G$ and $e\in \Bl(H)$, we write $e^G$ for the Brauer correspondent of $e$ in $G$ whenever it is defined (see~\cite{CRII}, Definition 58.8). We will rely on the following standard result.

\begin{thm}[Nagao; see \cite{CRII}, Theorem 58.22]\label{Alperin}
Let $P$ be a $p$-subgroup of a finite group $G$. Let $H$ be a subgroup of $G$ containing $PC_G(P)$. Suppose $N$ is an indecomposable $\Cal OG$-module and $M$ is an indecomposable $\Cal OH$-module such that $M \di \Res^G_H N$ and $P$ is a vertex of $M$. Let $e$ be the $p$-block of $H$ such that $eM=M$. Then $e^G$ is defined and $(e^G)N=N$.
\end{thm}

\begin{proof}[Proof of Theorem~\ref{corrthmgen}]
\eqref{i1} Let $\phi\in \Irr(H,P)$. It is enough to show that $\Ind_H^G \phi \in \Cal C(G,P)+ \Cal I(G,P,\Cal S)$ for all such $\phi$. 
Let $e$ be the $p$-block of $H$ such that $\phi\in \Irr(H,e)$, so that $P$ contains a defect group $D$ of $e$.
Let $M$ be an $\Cal OH$-lattice affording $\phi$ and $Q$ be a vertex of $M$. We may assume that $Q\subseteq D$. 
If $\ls{g}{Q}\in \Cal S$ for some $g\in G$, then $\Ind_H^G M$ is 
an $\Cal S$-projective $\Cal OG$-lattice, whence $\Ind_H^G \phi\in \Cal I(G,P,\Cal S)$ by Theorem~\ref{relpr}. 

So we may assume that no $G$-conjugate of $Q$ belongs to $\Cal S$. Then, by the Green correspondence associated with the triple $(G,H,P)$ (see~\cite{CRI}, Theorem 20.6), we have $\Ind_H^G M \simeq N \oplus V$ where $N$ is an indecomposable $\Cal OG$-lattice with vertex $Q$ and $V$ is relatively $\Cal S$-projective. We claim that $C_G (Q)\subseteq H$ and $N_G (D) \subseteq H$. Indeed, if $g\in C_G (Q)-H$ or if $g\in N_G (D)-H$ then 
$Q\subseteq P\cap \lsa{g}{P}\in \Cal S$ (as $Q\subseteq D$), contradicting our assumption on $Q$. Since $N_G (D) \subseteq H$, the block $e^G$ has $D$ as a defect group by Brauer's first main theorem. Now, since $QC_G(Q)\subseteq H$, we have $(e^G)N=N$ by Theorem~\ref{Alperin}, so the character afforded by $N$ 
belongs to $\Cal C(G,P)$. On the other hand, the character afforded by $V$ is in $\Cal I(G,P,\Cal S)$ by Theorem~\ref{relpr}. 
Hence $\Ind_H^G \phi \in \Cal C(G,P)+ \Cal I(G,P,\Cal S)$.

\eqref{i2} follows immediately from Definition~\ref{defI}. 

\eqref{i3} Let $L$ be a subgroup of $G$ such that $Q \coleq L\cap P \in \Cal S$ and $Q$ is a Sylow $p$-subgroup of $L$. Suppose $\phi\in \Irr(L)$, and let $M$ be an $\Cal OL$-lattice affording $\phi$. Let $N \di \Res^G_H \Ind_L^G M$ and let $\eta\in \Cal C(H)$ be the character afforded by $N$. It suffices to show that $\Proj_P (\eta) \in \Cal I(H,P,\Cal S)$ for all such $\eta$.

Since $Q$ is a Sylow $p$-subgroup of $L$, there exists an $\Cal OQ$-lattice $U$ such that $M \di \Ind_Q^L U$. Then 
\[
 N \di \Res^G_H \Ind^G_L M \di \Res^G_H \Ind_Q^G U,
\]
and it follows from the Mackey formula that $N \di \Ind_{\ls{g}{Q} \cap H}^H V$ for some $g\in G$ and some $\Cal O(\ls{g}{Q}\cap H)$-lattice $V$. 
Thus $N$ is $\ls{g}{Q}$-projective.
Let $e$ be the block of $H$ such that $eN=N$. If no defect group of $e$ is contained in $P$ then $\Proj_P (\eta)=0$. 

So we may assume that a vertex $S$ of $N$ is contained in $P$. Since $N$ is relatively $\ls{g}Q$-projective, 
there is $h\in H$ such that $S\subseteq \ls{hg}{Q}$, whence $N$
 is $(\ls{hg}{Q} \cap P)$-projective. We claim that $\ls{hg}{Q} \cap P \in \Cal S$. Indeed, since $Q\in \Cal S$, we have $Q\subseteq P \cap \lsa{x}{P}$ for some $x\in G-H$, whence
$$
\ls{hg}{Q} \cap P \subseteq P \cap \lsa{hg}{P} \cap \lsa{hgx}P.
$$
Since $x\notin H$, at least one of $hg$ and $hgx$ is not in $H$, and therefore $\ls{g}Q\cap P\in \Cal S$, as claimed. Hence, by Theorem~\ref{relpr}, we have
$\Proj_P (\eta)=\eta\in \Cal I(H,P,\Cal S)$, and part~\eqref{i3} follows.

\eqref{i4} Let $\phi\in \Cal C(H,P)$, and let $T$ be a set of representatives of double $H$-$H$-cosets in $G-H$. 
By the Mackey formula,
\begin{equation}\label{mackey}
\Res^G_H \Ind^G_H \phi  =  \phi + \sum_{t\in T} \Ind_{\lsa{t}{H}\cap H}^H \Res^{\lsa{t}H}_{\lsa{t}H\cap H} \lsa{t}\phi \\
\end{equation}
We claim that $\Proj_P \Ind_{\lsa{t}H\cap H}^H \Res^{\lsa{t}H}_{\lsa{t}{H}\cap H} \lsa{t}\phi \in \Cal I(H,P,\Cal S)$ for every $t\in T$. Let $U$ be an $\Cal OH$-lattice affording $\phi$,  and let $V$ be an indecomposable summand of $\Ind_{\lsa{t}H\cap H}^H \Res_{\lsa{t}H\cap H}^{\lsa{t}H} \ls{t}U$.  Let $e$ be a block of $H$ with a defect group contained in $P$. Then $eV$ is $P$-projective and therefore has a vertex $Q$ contained in $P$. On the other hand,
since $P$ contains a defect group of the block of $H$ containing $\phi$, the module $U$ is $P$-projective, whence $V$ is $\lsa{t}P\cap H$-projective.
Thus, $Q\subseteq \lsa{h}(\lsa{t}P\cap H)$ for some $h\in H$, so $Q\subseteq \lsa{ht}P \cap P \in \Cal S$. Hence, by Theorem~\ref{relpr}, the character afforded by $eV$ lies in $\Cal I(H,P,\Cal S)$, and our claim follows. Therefore, Equation \eqref{mackey} and Proposition~\ref{splitblocks} yield
\[
 \Proj_P \Res^G_H \Ind^G_H \phi \equiv \phi \dmod  \Cal I(H,P,\Cal S) \cap \Cal C(H,P).
\]

It only remains to show that the map between the quotients $\Cal C(H,P)/(\Cal I(H,P,\Cal S)\cap \Cal C(H,P))$ and 
$(\Cal C(G,P)+\Cal I(G,P,\Cal S))/\Cal I(G,P,\Cal S)$ induced by $\Ind_H^G$ is surjective, i.e.\ that 
$$
\Cal C(G,P) \subseteq \Ind_H^G (\Cal C(H,P)) + \Cal I(G,P,\Cal S).
$$
Let $\chi\in \Irr (G,P)$, and let $N$ be an $\Cal OG$-lattice affording $\chi$. Then $N$ has a vertex $Q$ contained in $P$. If $\ls{g}{Q}\in \Cal S$ for some $g\in G$ then $\chi \in \Cal I(G,P,\Cal S)$ by Theorem~\ref{relpr}. So we may assume that no $G$-conjugate of $Q$ belongs to $\Cal S$. 
Therefore, $N$ has a well-defined Green correspondent $M$ with respect to the triple $(G,H,P)$ (so that $M$ is an $\Cal OH$-lattice).
Let $b\in \Bl(G)$ and $e\in \Bl(H)$ be the blocks such that $bN=N$ and $eM=M$.
By the properties of Green correspondence, $Q$ is a vertex of $M$ and $M \di \Res^G_H N$. Also, $C_G (Q) \subseteq H$ (see the proof of~\eqref{i1}).
Using Theorem~\ref{Alperin} we deduce that $b=e^G$.

Since $Q$ is a vertex of an $\Cal OH$-module belonging to $e$, there is a defect group $S$ of $e$ that contains $Q$ (see~\cite{CRII}, Corollary 57.27). Since $b$ has a defect group contained in $P$ and $e^G=b$, by~\cite{CRII}, Corollary 58.18, we have $S\subseteq_G P$. Let $g$ be an element of $G$ such that $\lsa{g}{S}\subseteq P$. Then $\ls{g}{Q}\subseteq P$ and, since $Q\notin \Cal S$, we have $g\in H$. So $\lsa{g}{S}$ is a defect group of $e$ contained in $P$, and we infer that the character $\phi$ afforded by $M$ belongs to $\Cal C(H,P)$.

By a property of Green correspondence, $\Ind_H^G M \simeq N \oplus V$ where $V$ is a relatively $\Cal S$-projective $\Cal OG$-lattice. Denoting by $\th$ the character afforded by $V$, we see that $\th\in \Cal I(G,P,\Cal S)$ (by Theorem~\ref{relpr}), and hence 
\[
\chi=\Ind_H^G \phi - \th \in \Ind_H^G (\Cal C(H,P)) + \Cal I(G,P,\Cal S). \qedhere
\]
\end{proof}

\begin{remark}
If we assume that $P$ is a Sylow $p$-subgroup of $G$ (cf.\ Theorem~\ref{corrthm}), the proof becomes considerably shorter and no longer requires any use of modular representation theory. Indeed, parts~\eqref{i1} and~\eqref{i2} are clear, and~\eqref{i3} follows from the Mackey formula and the fact that $\ls{g}{Q}\cap P \in \Cal S$ whenever $Q\in \Cal S$ and $g\in G$ (see the proof of~\eqref{i3} above). The fact that $\Res^G_H \Ind^G_H$ yields 
the identity map on $\Cal C(H)/\Cal I(G,P, \Cal S)$ also follows from the Mackey formula.
We sketch the proof of the statement that $\Ind^G_H$ induces a surjective map between the quotients in~\eqref{i4}. By Theorem~\ref{BrInd}, it suffices to show that if $E$ is an elementary subgroup of $G$ then either $E\subseteq_G H$ or $E_p$ is $G$-conjugate to a subgroup of $P$ lying in $\Cal S$. Replacing $E$ with a $G$-conjugate, we may assume that $E_p\subseteq P$. Now suppose $E\nsubseteq H$ and consider $g\in E-H$. Then $E\subseteq \lsa{g}{P}\cap P$, whence $E\in \Cal S$. 
\end{remark}

\begin{remark} If $H$ does not contain $N_G (P)$, the statement of Theorem~\ref{corrthmgen} is still true, but is not interesting.
\end{remark}

Now we can state more general versions of the properties of Section~\ref{refinements}. 
As usual,  if $a$ is an integer coprime to $p$, we write $v_p (p^m a)=m$ and $(p^m a)_{p'}=a$.
Let $P$ be an arbitrary $p$-subgroup of $G$. Define
\begin{eqnarray*}
\Irr_{0}(G,P) & = &  \{ \chi\in \Irr(G,P) \mid v_p (\chi(1)) = v_p(|G:P|) \},  \\
\Irr^p (G,P) & = & \Irr(G,P) - \Irr_0 (G,P)
\end{eqnarray*}
and $\Cal C^p (G,P)=\mZ\Irr^p (G,P)$.
Thus, $\Irr_{0}(G,P)$ consists of irreducible characters of height $0$ in blocks with defect group $P$ in $G$. (Recall that the \emph{height} of an irreducible character $\chi$ of $G$ belonging to a block with defect group $P$ is defined as $v_p (\chi(1))-v_p (|G:P|)$.)
If $H$ is a subgroup of $G$ containing $N_G (P)$, we set $\Cal S=\Cal S(G,P,H)$ and consider the following five properties that may or may not hold for the triple $(G,P,H)$:

\bigskip
\noindent
\begin{tabular}{lp{\width}}
\!\!\!(IRC)\!\!\!
&
\emph{\quad There is a signed bijection $F\colon \pm\Irr_{0}(G,P) \ra \pm\Irr_{0}(H,P)$ such that
$$
F(\chi) \equiv \Proj_P \Res_H^G \chi \dmod  \Cal I(H,P,\Cal S) \qquad
$$
\quad for all $\chi\in \Irr_{0}(G,P)$.
} \\
\end{tabular}

\bigskip
\noindent
\begin{tabular}{lp{\width}}
\text{\!\!\!(pRes)} & 
\centering{
$
\Proj_P \Res^G_H (\Cal C^p (G,P)) \subseteq \Cal C^p (H,P) + \Cal I(H,P,\Cal S). \qquad
$}
\end{tabular}

\bigskip
\noindent
\begin{tabular}{lp{\width}}
\text{\!\!\!(pInd)} & 
\centering{
$
\Ind^G_H (\Cal C^p (H,P)) \subseteq \Cal C^p (G,P) + \Cal I(G,P,\Cal S). \qquad
$}
\end{tabular}

\bigskip
\noindent
\begin{tabular}{lp{\width}}
\!\!\!(WIRC)\!\!\!
&
\emph{\;There is a signed bijection $F\colon \pm\Irr_{0}(G,P) \ra \pm\Irr_{0}(H,P)$ such that
$$
F(\chi) \equiv \Proj_P \Res_H^G \chi \dmod  \Cal C^p (H,P) + \Cal I(H,P,\Cal S) \qquad
$$
\;for all $\chi\in \Irr_{0}(G,P)$.
} \\
\end{tabular}

\bigskip
\noindent
\begin{tabular}{lp{\width}}
\!\!\!(WIRC*)\!\!\!
&
\emph{There is a signed bijection $F\colon \pm\Irr_{0}(G,P) \ra \pm\Irr_{0}(H,P)$ such that
$$
\chi \equiv \Ind_H^G F(\chi) \dmod  \Cal C^p (G,P) + \Cal I(G,P,\Cal S) \qquad
$$
for all $\chi\in \Irr_{0}(G,P)$.
} \\
\end{tabular}
\medskip

It follows from Theorem~\ref{corrthmgen} that if (pRes) holds then (WIRC*) implies (WIRC); and if (pInd) holds then (WIRC) implies (WIRC*).
Conjectures~\ref{conjA} and~\ref{conjB} may be generalised as follows.

\begin{conj}\label{conjgen} Let $P$ be a $p$-subgroup of a finite group $G$. Then properties \emph{(pRes)}, \emph{(pInd)} and \emph{(WIRC)} hold for the triple $(G,P,N_G(P))$. Moreover, if $P$ is abelian then \emph{(IRC)} holds for that triple. 
\end{conj}

Now we consider the properties stated above at the level of an individual block.

\begin{prop}\label{otherblocks} 
Let $P$ be a $p$-subgroup of a finite group $G$, and let $b$ be a block of $G$ such that some defect group of $b$ is contained in $P$. Suppose $H$ is a subgroup of $G$ containing $N_G (P)$ and set $\Cal S=\Cal S(G,P,H)$. Let $e$ be a block of $H$ such that some defect group of $e$ is contained in $P$. 
\begin{enumerate}[(i)]
\item\label{ob1} If $\chi\in \Irr(G,b)$ and either $e^G$ is not defined or $e^G\ne b$ then $\Proj_e \Res^G_H \chi\in \Cal I(H,P,\Cal S)$.
\item\label{ob2} If $\phi\in\Irr(H,e)$ and $c$ is a $p$-block of $G$ such that either $e^G$ is not defined or $e^G\ne c$ then 
$\Proj_c \Ind_H^G \phi \in \Cal I(G,P,\Cal S)$.
\end{enumerate}
\end{prop}

\begin{proof} \eqref{ob1}
 Let $N$ be an $\Cal OG$-lattice affording $\chi$. If $N$ is $\Cal S$-projective then $\Proj_P \Res^G_H \chi \in \Cal I(H,P,\Cal S)$ by Theorem~\ref{corrthmgen}\eqref{i3} and the result follows from Proposition~\ref{splitblocks}. Thus we may assume that no vertex of $N$ 
is an element of $\Cal S$, so $N$ has a 
Green correspondent, $M$ say, with respect to the triple $(G,H,P)$. Let
$$
\Cal Y = \{ S \le P \mid S\le \lsa{g}{P} \cap H \text{ for some } g\in G-H \}.
$$
 Then 
$\Res^G_H N \simeq M\oplus U$ where $U$ is $\Cal Y$-projective. Let $V$ be an indecomposable summand of $eU$. Since $P$ contains a defect group of $e$, the lattice $V$ has a vertex $S$ contained in $P$. On the other hand, some $H$-conjugate of $S$ lies in $\Cal Y$ (by the properties of Green correspondence), so $\lsa{h}{S}\le \lsa{g}{P}$ for some $h\in H$ and $g\in G-H$. Thus, $S\subseteq P\cap \lsa{h^{-1}g}P$, whence $S\in \Cal S$. Using Theorem~\ref{relpr} we deduce that the character afforded by $eU$ lies in $\Cal I(H,P,\Cal S)$.

Since $\Proj_e \Res^G_H \chi$ is afforded by the lattice $eM\oplus eU$, it remains only to show that the character $\th$ afforded by $eM$ belongs to $\Cal I(H,P,\Cal S)$. Let $Q$ be a vertex of $M$ contained in $P$. Note that $Q\notin \Cal S$ by the properties of the Green correspondence.
If $C_G(Q) \subseteq H$ then, since $e^G\ne b$, Theorem~\ref{Alperin} shows that $eM=0$. On the other hand, if $g\in C_G (Q)-H$, then $Q\subseteq P\cap \ls{g}{P}$, whence $Q\in \Cal S$, which is a contradiction.

%By~\cite{CRII}, Theorem 58.22, $eM=M$, so $fM=0$. 

\eqref{ob2} Let $M$ be an $\Cal OH$-lattice affording $\chi$. If $M$ is $\Cal S$-projective then the result follows from Theorem~\ref{relpr} and Theorem~\ref{corrthmgen}\eqref{i2}. So we may assume that no vertex of $M$ lies in $\Cal S$. By the Green correspondence we have $\Ind_H^G M \simeq N\oplus V$ where $bN=N$ and $V$ is $\Cal S$-projective.
Let $Q$ be a vertex of $M$ contained in $P$. Then 
$C_G (Q) \subseteq H$ (as in the proof of~\eqref{ob1}), and by Theorem~\ref{Alperin} the Brauer correspondent $e^G$ is defined and satisfies 
$(e^G)N=N$, whence $cN=0$. On the other hand, the character afforded by $cV$ lies in $\Cal I(G,P,\Cal S)$ by Theorem~\ref{relpr}. 
The result follows.
\end{proof}

\begin{remark}
 Let the notation be as in Proposition~\ref{otherblocks}.
 If $P$ is abelian and we assume Brauer's height-zero conjecture to be true (see e.g.~\cite{NavarroBook}, Chapter 9, p.\ 212) then all elements of $\Irr^p (G,P)$ and $\Irr^p (H,P)$ belong to blocks with defect groups strictly contained in $P$. In this case properties (pRes) and (pInd) both hold by Proposition~\ref{otherblocks}. 
\end{remark}

If $b$ is a $p$-block of $G$, we denote by $\Irr_0 (G,b)$ the set of characters of height $0$ in $\Irr(G,b)$. We write 
$\Irr^p (G,b)=\Irr(G,b)-\Irr_0 (G,b)$ and $\Cal C^p(G,b)=\mZ \Irr^p (G,b)$. We will often use the following notation.

\begin{hyp}\label{hypbl} Let $G$ be a finite $p$-group. Let $P$ be a $p$-subgroup of $G$ and suppose $H$ is a subgroup of $G$ containing $N_G (P)$.
 Suppose that $b$ is a $p$-block of $G$ such that $P$ is a defect group of $b$, and let $e\in \Bl(H)$ be the Brauer correspondent of $b$. 
\end{hyp}

With these assumptions, we define properties (IRC-Bl), (pRes-Bl), (pInd-Bl), (WIRC-Bl) and (WIRC*-Bl) for the quadruple $(G,b,P,H)$ as follows. For each of the five properties, we consider the statement above of the corresponding property of the triple $(G,P,H)$ and replace $\Irr_0 (G,P)$ with $\Irr_0 (G,b)$, $\Irr_0 (H,P)$ with $\Irr_0 (H,e)$, $\Cal C^p (G,P)$ with $\Cal C^p (G,b)$, and $\Cal C^p(H,P)$ with $\Cal C^p(H,e)$. For instance,
(IRC-Bl) may be stated as follows.

\bigskip
\noindent
\begin{tabular}{lp{\width}}
\!\!\!(IRC-Bl)\!\!\!
&
\emph{\quad There is a signed bijection $F\colon \pm\Irr_{0}(G,b) \ra \pm\Irr_{0}(H,e)$ such that
$$
F(\chi) \equiv \Proj_P \Res_H^G \chi \dmod  \Cal I(H,P,\Cal S(G,P,H)) \qquad 
$$
\quad for all $\chi\in \Irr_{0}(G,b)$.
} \\
\end{tabular}
\medskip

Let $P$ be a $p$-subgroup of $G$ and suppose $H\le G$ contains $N_G (P)$.
It follows from Proposition~\ref{otherblocks} 
that if (IRC-Bl) holds for all blocks $b\in \Bl(G)$ for which $P$ is a defect group then (IRC) holds for the triple $(G,P,H)$. Analogous statements hold for the other four properties.

Also, it is easy to see, using Proposition~\ref{otherblocks}, that if (pRes) holds for the triple $(G,P,H)$ then (pRes-Bl) is true, with respect to $P$ and $H$, for each block $b$ of $G$ with defect group $P$; and the analogous statement is true for property (pInd). The analogues for properties 
(IRC), (WIRC) and (WIRC*) follow from Proposition~\ref{degprop} below and Proposition~\ref{otherblocks}.

Isaacs and Navarro~\cite{IsaacsNavarro2002} proposed a generalisation of their Conjecture~\ref{conjIN} to the case of blocks of arbitrary defect, which is a refinement of Alperin's strengthening~\cite{Alperin1976} of the McKay conjecture. 
If $b$ is a block of $G$, let $M_l (b)$ be the number of irreducible characters  $\chi\in \Irr_0 (G,b)$
such that $\chi(1)_{p'} \equiv \pm l \pmod p$.

\begin{conj}[\cite{IsaacsNavarro2002}, Conjecture B]\label{conjINgen}
Let $b$ be a $p$-block of a finite group $G$. Suppose that $P$ is a defect group of $b$ and $H=N_G (P)$. Let $e\in \Bl(H)$ be the Brauer correspondent of $b$.
Then for each integer $l$ not divisible by $p$ we have $M_{ml} (G,b)=M_l (H,e)$ where $m=|G:H|_{p'}$.
\end{conj}

This conjecture is implied by Conjecture~\ref{conjgen}.

\begin{prop}\label{degprop} Assume Hypothesis~\emph{\ref{hypbl}}. Suppose $F\colon \pm \Irr_{0} (G,P) \ra \pm \Irr_{0} (H,P)$ is a signed bijection witnessing either \emph{(WIRC)} or \emph{(WIRC*)}. If $\chi\in \Irr_0 (G,b)$ then 
$F(\chi)\in \pm\Irr_0 (H,e)$ and 
\begin{equation}\label{degpropeq}
\chi(1)_{p'} \equiv |G:H|_{p'} (F(\chi)(1))_{p'} \pmod p.
\end{equation} 
In particular, if $H=N_G (P)$ then Conjecture~\emph{\ref{conjINgen}} holds for all blocks of $G$ with defect group $P$.
\end{prop}

\begin{proof} We will only consider the case when $F$ witnesses property (WIRC) because a proof for  (WIRC*) is similar. Let $\chi\in \Irr_0 (G,b)$. Applying $\Ind_H^G$ to the equation of (WIRC), we obtain
\[
 \Ind_H^G F(\chi) \equiv \Ind_H^G \Proj_P \Res^G_H \chi \dmod \Ind_H^G (\Cal C^p (H,P)) + \Cal I(G,P,\Cal S),
\]
where $\Cal S=\Cal S(G,P,H)$. By Theorem~\ref{corrthmgen}\eqref{i4}, this gives
\[
 \Ind_H^G F(\chi) \equiv \chi \dmod \Ind_H^G (\Cal C^p (H,P))+ \Cal I(G,P,\Cal S).
\]
Now all virtual characters belonging to $\Ind_H^G (\Cal C^p(H,P))$ have degrees divisible by $p^{v_p(|G:P|)+1}$. The same is true for virtual characters lying in $\Cal I(G,P,\Cal S)$ because each element of $\Cal S$ is a proper subgroup of $P$. Therefore, we have
\begin{equation}\label{degpropeq1}
 |G:H| F(\chi)(1) \equiv \chi(1) \dmod p^{v_p(|G:P|)+1},
\end{equation}
and~\eqref{degpropeq} follows. 

Now suppose for contradiction that $F(\chi)\notin \pm\Irr_0 (H,e)$. Let $f$ be the block of $H$ containing $\pm F(\chi)$. Then, by the identity of (WIRC) and Proposition~\ref{splitblocks}, we have $F(\chi)-\Proj_f \Res^G_H \chi \in \Cal I(H,P,\Cal S)$. On the other hand, by Proposition~\ref{otherblocks}, we have $\Proj_f \Res^G_H \chi \in \Cal I(H,P,\Cal S)$, so $F(\chi)\in \Cal I(H,P,\Cal S)$. However, by~\eqref{degpropeq1}, 
$v_p (F(\chi)(1))=v_p (|H:P|)$, whereas we have already observed that $v_p (\xi)>v_p (|H:P|)$ for each $\xi\in \Cal I(H,P,\Cal S)$. This contradiction completes the proof.
\end{proof}

We finish the section with another consequence of Conjecture~\ref{conjgen} (cf.~\cite{IsaacsMalleNavarro2007}, Conjecture C). If $\nu$ is an irreducible character of a normal subgroup of $G$, we write
$\Irr(G,P \di \nu)=\Irr(G,P) \cap \Irr(G \di \nu)$, $\Cal C(G \di \nu)=\mZ \Irr(G\di \nu)$, and so on.

\begin{prop}\label{normsub}
Let $L$ be a normal subgroup of a finite group $G$. Let $\tilde{P}/L$ be a $p$-subgroup of $G/L$ and $P$ be a Sylow $p$-subgroup of $\tilde{P}$. Let $H$ be a subgroup of $G$ containing $N_G (P)$.
\begin{enumerate}[(i)]
 \item\label{ns1} If $\nu\in\Irr(L)$ and $F\colon \pm \Irr_0 (G,P) \ra \pm \Irr_0 (HL,P)$ is a signed bijection witnessing either \emph{(WIRC)} or  \emph{(WIRC*)} then 
for every $\chi\in \pm \Irr_0 (G,P \di \nu)$ there is a $G$-conjugate $\nu'$ of $\nu$ such that $F(\chi)\in \pm \Irr_0 (HL,P \di \nu')$.
 \item\label{ns2} If any one of the properties \emph{(IRC)}, \emph{(pRes)}, \emph{(pInd)}, \emph{(WIRC)} and \emph{(WIRC*)} holds for the triple $(G,P,HL)$ then the same property is true for $(G/L,\tilde{P}/L, HL/L)$. 
\end{enumerate}
\end{prop}

\begin{proof}
\eqref{ns1} Let $\Cal S=\Cal S(G,P,H)$. Let $\nu_1,\ldots,\nu_l$ be a complete set of representatives of $G$-orbits on $\Irr(L)$ and $\nu'_1,\ldots,\nu'_m$ be a complete set of representatives of the $H$-orbits on $\Irr(L)$. The abelian group 
$\Cal I(G,P,\Cal S)$ is spanned by the virtual characters of the form $\Ind_A^G \phi$ where $A$ is a subgroup of $G$ such that $A\cap P$ is a Sylow $p$-subgroup of $A$ and $A\cap P\in \Cal S$. Using the facts that $L\cap P$ is a Sylow $p$-subgroup of $L$ and that $L\cap P$ is contained in every maximal element of $\Cal S$, it is easy to see that $AL\cap P$ is a Sylow $p$-subgroup of $AL$ and $AL\cap P\in \Cal S$ for each such $A$. Thus we have
\[
 \Ind_A^G \phi = \sum_{i=1}^n \Ind_{AL}^G \pi_{\mu_i} \Ind_A^{AL} \phi 
\]
where $\mu_1,\ldots,\mu_n$ is a set of representatives of the $A$-orbits on $\Irr(L)$.
It follows that
\begin{equation}\label{nseq1}
 \Cal I(G,P,\Cal S) = \bigoplus_{i=1}^l \left( \Cal I(G,P,\Cal S) \cap \Cal C(G \di \nu_i) \right).
\end{equation}
Similarly,
\begin{equation}\label{nseq2}
 \Cal I(HL,P,\Cal S)=\bigoplus_{i=1}^m \left( \Cal I(HL,P,\Cal S) \cap \Cal C(HL \di \nu'_i) \right).
\end{equation}
Now suppose $F\colon \pm\Irr_0 (G,P) \ra \pm \Irr_0 (H,P)$ is a signed bijection satisfying (WIRC). Let $\nu\in \Irr(L)$ and suppose $\chi\in \Irr_0 (G,P \di \nu)$. Let $\nu'\in \Irr(L)$ be a character such that $F(\chi) \in \pm\Irr_0 (HL,P\di \nu')$. Suppose for contradiction that $\nu'$ is not $G$-conjugate to $\nu$. By the identity of (WIRC), we have 
\[
 F(\chi)-\Proj_P \Res^G_H \chi \in \Cal C^p (HL,P)+\Cal I(HL,P,\Cal S).
\]
However, each irreducible constituent of $\Proj_P \Res^G_H \chi$ lies in $\Irr(HL \di \lda)$ for some $\lda\in \Irr(L)$ that is $G$-conjugate to $\nu$,
and by~\eqref{nseq2} we deduce that
\[
 \Proj_P \Res^G_H \chi \in \Cal C^p (HL,P)+\Cal I(HL,P,\Cal S).
\]
Using Theorem~\ref{corrthmgen} we infer that 
\[
 \chi \in \Ind_{HL}^G (\Cal C^p (HL,P))+\Cal I(G,P,\Cal S).
\]
This is a contradiction because $v_p (\chi(1))=v_p (|G:P|)$, whereas every element of the set on the right-hand side has a degree divisible by 
$p^{v_p (|G:P|)+1}$. The proof in the case when $F$ witnesses (WIRC*) is similar (in fact, slightly easier).

\eqref{ns2} The statements for properties (pRes) and (pInd) follow immediately from~\eqref{nseq1} and~\eqref{nseq2}. We obtain the statements for properties (IRC), (WIRC) and (WIRC*) by putting $\nu=1_L$ in~\eqref{ns1} and using the same two identities. 
\end{proof}

\section{Splendid Rickard equivalences}\label{splendid}

\subsection{A refinement of the conjectures of Brou{\'e} and Rickard}\label{Broue}

It is believed that if $P$ is abelian then a one-to-one correspondence between the sets $\Irr_{0} (G,P)$ and $\Irr_{0}(N_G (P), P)$ is just one consequence of a much deeper relationship, namely of derived equivalences between relevant block algebras. The existence of such equivalences was conjectured by Brou{\'e}~\cite{Broue1990}. In this section we formulate a refinement of the Brou{\'e} conjecture which implies Conjecture~\ref{conjgen} when $P$ is abelian. 

First we recall two important definitions due to Rickard (see~\cite{Rickard1996}), using slightly different terminology. Throughout Section~\ref{splendid} we will denote by $R$ a ring which is either $\Cal O$ or $k$. Following~\cite{Rickard1996}, by a \emph{complex} we  understand a cochain complex 
$$
X = \qquad \cdots \ra X^i \ra X^{i+1} \ra \cdots.
$$
Recall that if $X$ is a complex of $R$-modules then the \emph{dual complex} $X^*$ is defined as follows:
 the $n$th term of $X^*$ is $(X^{-n})^* = \Hom_R (X^{-n}, R)$, 
and the differential $(X^*)^n \ra (X^*)^{n+1}$ of $X^*$ is the dual of the differential $X^{-n-1}\ra X^{-n}$ of $X$. 

\begin{defi}[see \cite{Rickard1996}, Section 2] Let $A$ and $B$ be symmetric $R$-algebras. A bounded complex $X$ of finitely generated $A$-$B$-bimodules is said to be a \emph{Rickard tilting complex} if all the terms of $X$ are projective as left and right modules and we have isomorphisms $X\otimes_B X^* \simeq A$ and $X^* \otimes_A X\simeq B$ in the homotopy categories of $A$-$A$-bimodules and $B$-$B$-bimodules respectively.
\end{defi}

If $P$ is embedded as a subgroup into two groups $G$ and $H$, we write 
\[
\D P=\{(x,x) \mid x\in P \} \le G \times H.
\]
If $\Cal S$ is a set of subgroups of $P$, we set $\D \Cal S = \{\D Q \mid Q\in \Cal S\}$. Recall that an $RG$-module $N$ is said to be \emph{$p$-permutation} if each indecomposable summand $U$ of $N$ has trivial source, that is, if each such $U$ is a summand of $\Ind_Q^G R$ for some $p$-subgroup $Q$ of $G$. 

\begin{defi}[see \cite{Rickard1996}, Section 2]\label{splendiddef} Let $G$ and $H$ be finite groups with a common $p$-subgroup $P$. Suppose $b$ and $c$ are central idempotents of $RG$ and $RH$ respectively. If $X$ is a Rickard tilting complex of $RGb$-$RHc$-bimodules and all the terms of $X$, considered as $R(G\times H)$-modules, are relatively $\D P$-projective $p$-permutation modules, then $X$ is said to be a \emph{splendid tilting complex}. 
\end{defi}

Assume Hypothesis~\ref{hypbl} and suppose $H=N_G (P)$. 
Rickard's refinement~\cite{Rickard1996} of the Brou{\'e} abelian defect group conjecture (\cite{Broue1990}, Question 6.2) 
asserts that if $P$ is abelian then there exists a splendid tilting complex of $\Cal OGb$-$\Cal OHe$-bimodules. 
In view of Definition~\ref{defX}, it seems reasonable to impose a further requirement upon Rickard tilting complexes. 

\begin{defi}\label{tempered} Let $P$ be a fixed $p$-subgroup of an arbitrary finite group $G$. 
Let $\Cal S$ be a set of subgroups of $P$. Let $X$ be a  complex of $RG$-modules. 
We say that $X$ is \emph{$\Cal S$-tempered} if at most one term $X^i$ is not $\Cal S$-projective. 
If such a term $X^i$ exists, let $X^i \simeq M\oplus N$ where $N$ is the largest $\Cal S$-projective summand of $X^i$. We say that $M$ is the \emph{pivot} (or \emph{$\Cal S$-pivot}) of $X$. 
If all terms of $X$ are $\Cal S$-projective, the pivot of $X$ is defined to be $0$.
\end{defi}

Assume Hypothesis~\ref{hypbl} (so, in particular, $\Cal S=\Cal S(G,P,H)$). 
It is well known that 
the $\Cal O(G\times H)$-module $b\Cal OGe$ has a unique indecomposable summand $U$ with vertex $\D P$: 
this is the Green correspondent of the $\Cal O(H\times H)$-module $\Cal OHe$ and of the $\Cal O(G\times G)$-module $\Cal OGb$ 
with respect to $\D P$.
 We will denote this module $U$ by $\fr{Gr} (G,b,H)$ (or simply by $\fr{Gr}(G,b)$ if $H=N_G (P)$). In particular, by the properties of the Green correspondence, 
we have isomorphisms of $\Cal O(G\times H)$-modules
\begin{enumerate}[(i)]
\item $b\Cal OG \simeq \fr{Gr}(G,b,H) \oplus V$ where $V$ is $\D \Cal S$-projective; and 
\item $\Cal OGe \simeq \fr{Gr}(G,b,H) \oplus W$ where $W$ is $\D \Cal S$-projective. 
\end{enumerate}
We will use these observations repeatedly. 

The following seems to be an appropriate refinement of the Brou{\'e} conjecture when one considers property (IRC).

\begin{conj}\label{rBr}
Assume Hypothesis~\emph{\ref{hypbl}}. Suppose $P$ is abelian and $H=N_G(P)$. Then there exists a 
splendid tilting complex of $\Cal OGb$-$\Cal OHe$-bimodules that is $\D\Cal S$-tempered with pivot $\fr{Gr} (G,b)$. 
\end{conj}

It may be of interest to consider a similar statement when $H$ is an arbitrary subgroup of $G$ containing $N_G (P)$. However, this more general situation is beyond the scope of the present paper.

\begin{remark} In the case when $\Cal S=\{ \mbf 1 \}$, Conjecture~\ref{rBr} corresponds precisely to the well-established approach of constructing a derived equivalence from the stable equivalence given by the Green correspondence (see e.g.\ \cite{Linckelmann1998}).
\end{remark}

It seems plausible that in many of the cases for which the Brou{\'e} conjecture and Rickard's refinement have been proved so far (a list is available e.g.\ in~\cite{ChuangRickard2007}), the double complexes constructed in the course of the proofs satisfy the more precise condition of Conjecture~\ref{rBr}. In Sections~\ref{cyclic} and~\ref{pnilpab} we verify this in two special cases.  

Now we establish character-theoretic consequences of Conjecture~\ref{rBr}. 
Suppose $G$ and $H$ are finite groups, and let $b\in \Bl(G)$ and $e\in \Bl(H)$. 
Let $\mu\in \Cal C(G \times H, b\otimes \bar e)$. 
Then $\mu$ gives rise to abelian group homomorphisms $I_{\mu}\colon\Cal C(H,e) \ra \Cal C(G,b)$ and $R_{\mu}\colon\Cal C(G,\bar b) \ra \Cal C(H,\bar e)$, defined as follows:
\begin{eqnarray*}
 I_{\mu} (\phi) (g) & = & \frac{1}{|H|} \sum_{h\in H} \mu(g,h) \phi(h) \quad \text{and} \\
R_{\mu} (\chi) (h) & = & \frac{1}{|G|} \sum_{g\in G} \mu(g,h) \chi(g)
\end{eqnarray*}
(see~\cite{Broue1990}, Section 1).

Now assume Hypothesis~\ref{hypbl}. Define 
\begin{equation}\label{omega}
\om(G,b,H)=\sum_{\chi\in \Irr(G,b)} \sum_{\phi\in\Irr(H,e)} \lan \Res^G_H \chi, \phi \ran (\chi \times \bar{\phi}),
\end{equation}
so that $I_{\om(G,b,H)} (\phi)=\Proj_b \Ind_H^G \phi$ and $R_{\ol{\om(G,b,H)}} (\chi)=\Proj_e \Res_H^G \chi$ for all $\phi\in\Irr(H,e)$ and $\chi\in\Irr(G,b)$. Moreover, $\om(G,b,H)$ is the character afforded by the $\Cal O(G\times H)$-lattice $b\Cal OGe$.

If $X$ is a chain complex of $\Cal OG$-lattices and $\chi_i$ is the character afforded by $X^i$, the \emph{virtual character afforded by $X$} is defined as 
$\sum_{i\in \mZ} (-1)^i \chi_i$. The following lemma follows immediately from Definition~\ref{tempered}, Theorem~\ref{relpr} and the properties of $\fr{Gr}(G,b)$ stated above.

\begin{lem}\label{mu1}
Assume Hypothesis~\emph{\ref{hypbl}}. Suppose $X$ is an $\Cal S$-tempered splendid 
tilting complex of $\O G b$-$\O He$-bimodules with pivot $\fr{Gr}(G,b)$, and let $X^i$ be the term of $X$ which $\fr{Gr}(G,b,H)$ is a summand of. 
Let $\mu$ be the character afforded by $X$. Then
$$
\mu \equiv (-1)^i \om(G,b,H) \dmod \Cal I(G\times H,\D P, \D \Cal S). 
$$
\end{lem}

\begin{lem}\label{mu2}
 Let $G$ and $H$ be finite groups with a common $p$-subgroup $P$ and suppose $\Cal S$ is a downward closed set of subgroups of $P$. 
If $\mu\in \Cal I(G\times H,\D P,\D\Cal S)$ then $I_{\mu} (\phi) \in \Cal I(G,P,\Cal S)$ for all $\phi\in \Cal C (H)$ and $R_{\mu}(\chi)\in \Cal I(H,P,\Cal S)$ for all $\chi\in \Cal C(G)$.
\end{lem}

\begin{proof} It suffices to prove the assertion for $I_{\mu}$ as the second statement follows after we swap $G$ and $H$. 
 Without loss of generality, $\mu= \Ind_{L}^{G\times H} \nu$ for some $\nu\in\Irr(L)$ where $L$ is a subgroup of $G\times H$ such that 
$\D Q \coleq L\cap \D P \in \Cal S$ and $\D Q$ 
is a Sylow $p$-subgroup of $L$. Let $\chi\in\Irr(G)$, and let $U$ be an $\O L$-lattice affording $\nu$. Then $V=\Ind_{L}^{G\times H} U$ affords $\mu$ and is relatively 
$\D Q$-projective, so $V \di \Ind_{\D Q}^{G\times H} W$ for some $\O Q$-lattice $W$. We may assume that $\phi\in\Irr(H)$. Let $N$ be an $\O H$-lattice affording $\phi$. By Theorem~\ref{relpr}, it suffices to show that 
$V \otimes_{\O H} N$ is $Q$-projective. In fact, it is enough to prove that 
$(\Ind_{\D Q}^{G\times H} W) \otimes_{\O H} N$ is $Q$-projective. However, it is not difficult to see that
$$
(\Ind_{\D Q}^{G\times H} W) \otimes_{\O H} N \simeq \Ind_{Q}^G (W\otimes_{\O} \Res^H_Q N),  \qquad
$$
where an isomorphism (of $\Cal OG$-modules) is given by
\[
 \qquad\qquad\qquad ((g,1) \otimes w) \otimes n \longmapsfrom g \otimes (w\otimes n), \quad g\in G,\, w\in W,\, n\in N,
\]
and the result follows.
\end{proof}

\begin{prop}\label{BrA} 
Assume Hypothesis~\emph{\ref{hypbl}}. Suppose $X$ is a splendid 
tilting complex of $\O G b$-$\O He$-bimodules that is $\Cal S$-tempered with pivot $\fr{Gr}(G,b)$, and let $X^i$ be the term of $X$ which $\fr{Gr}(G,b,H)$ is a summand of. 
Let $\mu$ be the character afforded by $X$. 
Then
\begin{equation}\label{BrA1}
R_{\bar{\mu}}(\chi) \equiv (-1)^i \Proj_e \Res^G_H \chi \dmod  \Cal I(H,P,\Cal S) \qquad \text{for all } \chi\in \Cal C(G,b).
\end{equation}
Moreover, the map $F=R_{\bar\mu}|_{\pm \Irr_0 (G,b)}$ is a bijection between $\pm \Irr_0 (G,b)$ and $\pm \Irr_0 (H,e)$.
In particular, Conjecture~\emph{\ref{rBr}} implies the last statement of Conjecture~\emph{\ref{conjgen}}.
\end{prop}

\begin{proof}
 The first statement is immediate from Lemmas~\ref{mu1} and~\ref{mu2}. It is well-known (see~\cite{Broue1990}, Th{\'e}or{\`e}me 3.1) that $R_{\bar\mu}$ and $I_{\mu}$ are mutually 
inverse isometries between $\Cal C(G,b)$ and $\Cal C(H,e)$, and hence restrict to signed bijections between $\pm\Irr (G,b)$ and $\pm\Irr (H,e)$. 
Let $\phi\in \pm \Irr(H,e)$. 
By Theorem~\ref{corrthmgen},
$$
I_{\mu} (\phi) \equiv \pm \Ind_H^G \phi \dmod \Cal I(G,P,\Cal S).
$$
Therefore, $v_p (I_{\mu}(\phi)(1))=v_p(|G:P|)$ if and only if $v_p(\phi(1))=v_p(|H:P|)$. The result follows.
\end{proof}

\subsection{Blocks with cyclic defect groups}\label{cyclic}

In this section we show that the complex constructed by Rouquier~\cite{Rouquier1998} in the case of a cyclic defect group $P$ satisfies the requirement
of Conjecture~\ref{rBr}.

First we need a result on composition of tempered complexes of $p$-permutation modules.
Recall that if $G$ and $H$ are groups with a common $p$-subgroup $S$, we say that \emph{$G$ controls $H$-fusion of subgroups of $S$} if, whenever $Q\le S$ and $h\in G_2$ satisfy
$\ls{h}{Q}\subseteq S$, there exists $g\in H$ such that $\ls{g}{x}=\ls{h}{x}$ for all $x\in Q$. The following lemma is a combination of the statement and the proof of~\cite{Chuang1999}, Lemma 8.2.
 
\begin{lem}[Chuang]\label{comp} 
Let $G$, $H$ and $L$ be finite groups with a common $p$-subgroup $P$. Let $M$ be a $p$-permutation 
$\Cal O(G\times H)$-module with vertex $\D S$ and let $N$ be a $p$-permutation $\Cal O(H\times L)$-module with vertex $\D Q$, where $S$ and $Q$ are subgroups of $P$. Then $U=M\otimes_{\Cal OH} N$ is a $p$-permutation module which is relatively projective with respect to the set
$$
\{ \ls{(1,h)} \D (S\cap \ls{h^{-1}} Q) \mid h\in H \}
$$
of subgroups of $P\times P$. 
Moreover, if $G$ controls the $H$-fusion of subgroups of $P$ then $U$ is $\D Q$-projective; and if $L$ controls the $H$-fusion of subgroups of $P$ then $U$ is $\D S$-projective.
\end{lem}

\begin{thm}\label{cyclicthm}
Conjecture~\ref{rBr} holds when $P$ is cyclic. 
\end{thm}

\begin{proof} We use the notation given by Hypothesis~\ref{hypbl}, with $H=N_G (P)$. 
 We argue by induction on $|P|$, observing that there is nothing to prove if $P$ is normal in $G$. Let $S=O_p (G)$, and let $Q$ be the subgroup of $P$ such that 
$|Q:S|=p$. Let $L=N_G (Q)$ and let $f\in \Bl(L)$ be the Brauer correspondent of $b$ and $e$. By~\cite{Rouquier1998}, Theorem 10.3, there is a splendid tilting complex $C$ of $\Cal OGb$-$\Cal OLf$-bimodules of the form
$$
0 \lra N \lra \fr{Gr}(G,b,L) \lra 0,
$$
where $N$ is a direct summand of $b\Cal OG \otimes_{\Cal OS} \Cal OLf$ (with, say, $C^0=\fr{Gr}(G,b,L)$).
By the inductive hypothesis, there is a splendid tilting complex $X$ of $\Cal OLf$-$\Cal OHe$-bimodules which is $\Cal S$-tempered with pivot 
$\fr{Gr}(L,f)$. It is clear that $C\otimes_{\O L} X$ is a Rickard tilting complex of $\Cal OGb$-$\Cal OHe$-bimodules.
Moreover, if $U$ is an indecomposable summand of a term of $X$ which is not isomorphic to $\fr{Gr}(L,f)$ then, by Lemma~\ref{comp}, 
$C^j \otimes_{\O L} U$ is a $\D \Cal S$-projective $p$-permutation module for each $j$ 
because $G$ certainly controls $L$-fusion of subgroups of $P$. Also,
$C^{-1} \mid \Ind_{\D S}^{G\times L} \O$ and, since $S$ is normal in $G$, it follows from Lemma~\ref{comp} that each tensor product 
$C^{-1}\otimes_{\Cal O L} X^j$ is a $p$-permutation module that is relatively projective with respect to the set of subgroups of the form
$\ls{(1,g)}{\D S} = \ls{(g^{-1},1)} \D S =_{G\times H} \D S$, $g\in L$. Since $H\ne G$, we have $S\in \Cal S$.

Finally, we claim that $\fr{Gr} (G,b)$ is a summand of $\fr{Gr}(G,b,L) \otimes_{\O L} \fr{Gr}(L,f)$ and that all other indecomposable summands of the latter 
$\O Gb$-$\O He$-bimodule are $\D \Cal S$-projective. Certainly $\fr{Gr}(G,b)$ is a summand of $b\O Ge \simeq b\O G \otimes_{\O L} \O Le$. As $\fr{Gr}(L,f)$ is the only summand of $\O Le$ that is not $\D\Cal S$-projective, it follows from Lemma~\ref{comp} that $\fr{Gr}(G,b)$ must be a summand of 
\[
b\O G \otimes_{\O L} \fr{Gr}(L,f)=b\O G \otimes_{\O L} f \fr{Gr}(L,f)\simeq b\O Gf \otimes_{\O L} \fr{Gr}(L,f).
\]
Now $b\Cal OGf\simeq \fr{Gr}(G,f,L)\oplus U$ for some $\D S$-projective bimodule $U$. By Lemma~\ref{comp}, $U\otimes_{\Cal OL} \fr{Gr}(L,f)$ is relatively projective with respect to the set of subgroups of the form $\ls{(1,g)}{\D S}$, $g\in L$ (as $S$ is normal in $G$). Since $\ls{(1,g)}{\D S}=_{G\times H} \D S$ for all $g\in L$, we see that $\fr{Gr} (G,b)$ cannot be a summand of $U\otimes_{\Cal OL} \fr{Gr}(L,f)$ and therefore must be a summand of 
$\fr{Gr}(G,f,L)\otimes_{\Cal OL}\fr{Gr}(L,f)$. 
  
 Since all indecomposable summands of $b\O Ge$ other than $\fr{Gr}(G,b)$ are $\D \Cal S$-projective, the same must be true for all indecomposable summands of $\fr{Gr}(G,f,L) \otimes_{\O L} \fr{Gr}(L,f)$ other than $\fr{Gr}(G,b)$. 
Therefore, $C\otimes_{\O L} X$ is $\D\Cal S$-tempered with pivot $\fr{Gr}(G,b)$.
\end{proof}

\subsection{Blocks of $p$-nilpotent groups}\label{pnilpab}

Rickard constructed a splendid tilting complex for blocks with defect group $P$
in the case when $G=L \rtimes P$ where $L$ is a $p'$-group and $P$ is abelian (see \cite{Rickard1996}, Section 7). In this section we show that essentially the same complex satisfies the requirement of Conjecture~\ref{rBr}. Rickard's construction relies on Dade's classification of endo-permutation modules for abelian $p$-groups (see~\cite{Dade1978a},~\cite{Dade1978b}), and we will need to make certain results concerning those modules more precise. An excellent survey of the theory of endo-permutation modules is given in~\cite{Thevenaz2007}.

Recall that $R$ is either $\Cal O$ or $k$. Let $P$ be a finite $p$-group. 
An $RP$-lattice $M$ is said to be an \emph{endo-permutation module} if the algebra $\End_{R} (M)$ is a permutation $RP$-module, that is, has a $P$-invariant 
$R$-basis. An endo-permutation $RP$-module $M$ is said to be \emph{capped} if it has an indecomposable summand with vertex $P$. Such a summand is then necessarily unique up to isomorphism (\cite{Dade1978a}, Theorem 3.8) and is denoted by $\capm(M)$. Two endo-permutation $RP$-modules are said to be equivalent if their caps are isomorphic. We shall write $[M]$ for the equivalence class of $M$. The ``extended'' \emph{Dade group} $D(P)=D_R (P)$ is the abelian group that consists of the equivalence classes of endo-permutation $RP$-modules with the operation $[M]+[N] = [M\otimes_R N]$. (Note that another definition is often used, leading to a slightly different notion of ``Dade group''; see~\cite{ThevenazBook}, Section 29.)

Let $\Cal S$ be a downward closed set of subgroups of $P$. Let $M$ be an endo-permutation $RP$-module. 
We will call $M$ an \emph{$\Cal S$-endo-permutation module} if each indecomposable summand of the $RP$-module $\End_R (M)$ is isomorphic either to $R$ or to a permutation module of the form $\Ind_Q^P R=R(P/Q)$ with $Q\in \Cal S$. Then, by~\cite{Dade1978a}, Lemma 6.4, capped indecomposable 
$\Cal S$-endo-permutation modules are precisely the indecomposable modules $M$ such that 
\begin{equation}\label{endo1}
\End_R (M) \simeq R \oplus \bigoplus_{i} \Ind_{Q_i}^P R \quad \text{as } RP\text{-modules}
\end{equation}
with $Q_i\in \Cal S$ for each $i$. 
If $Q$ and $S$ are subgroups of $P$ then, by the Mackey theorem, we have
$$
\Ind_Q^P R \otimes \Ind_S^P R \simeq \bigoplus_{x\in [Q\backslash P/S]} \Ind_{Q\cap\lsa{x}{S}}^P R,
$$
where $[Q\backslash P/S]$ denotes a set of representatives of $Q$-$S$-double cosets in $P$.
If $M$ and $N$ are $\Cal S$-endo-permutation modules then $\End_R (M\otimes N)\simeq \End_R (M) \otimes \End_R (N)$, and it follows that $M\otimes N$ is an $\Cal S$-endo-permutation module. 
Hence the classes $[M]$ such that $\capm(M)$ is an $\Cal S$-endo-permutation module form a subgroup of $D_R (P)$, which will be denoted by $D_{R} (P, \Cal S)$. 
We remark that $D_R (P,\{ \mbf 1 \})$ is precisely the group of endo-trivial $RG$-modules, up to equivalence (see~\cite{Dade1978a}, Section 7).

% Bouc2000 = ``Tensor induction of relative syzygies''

If $X$ is a non-empty finite (left) $P$-set and $RX$ denotes the corresponding permutation module then the \emph{relative syzygy} $\Om_X$ is defined as the kernel of the augmentation map $RX \ra R$, which is given by $x\mapsto 1$ for $x\in X$ (see~\cite{Alperin2001}). By~\cite{Bouc2000}, Lemma 2.3.3, $\End_R (\Om_X)$ is isomorphic, as an $RP$-lattice, to 
a direct summand of $R\oplus (RX \otimes_R RX)$. It follows that, if $X=P/Q$ with $Q\in \Cal S$, we have $[\Om_X] \in D_R (P, \Cal S)$. We denote by 
$D_R^{\Om | \Cal S} (P)$ the subgroup of $D_R (P, \Cal S)$ generated by the classes $[\Om_{P/Q}]$ with $Q\in\Cal S$. 
 
Let $Q$ be a normal subgroup of $P$. 
In this case, $\Om_{P/Q}=\Inf_{P/Q}^P \Om^1_{P/Q} (R)$ and, more generally, $n[\Om_{P/Q}]=[\Inf_{P/Q} \Om^n_{P/Q} (R)]$ for all $n\in \mZ$, where $\Om^n_{P/Q} (R)$ is the $n$th Heller translate of the trivial $R(P/Q)$-module (see~\cite{Thevenaz2007}, Section 4). We will write
$\Om_{P/Q}^{-1} = \Inf_{P/Q}^P \Om_{P/Q}^{-1}(R)$.

Denote by $\Def^P_{P/Q}$ the deflation (or ``slash") operation defined in~\cite{Dade1978a}, Section 4, which sends an endo-permutation $kP$-module to an endo-permutation $k(P/Q)$-module. It induces a group homomorphism $\Def^P_{P/Q}\colon D_k (P) \ra D_k (P/Q)$.
In particular, suppose $M$ is a capped indecomposable endo-permutation $kP$-module. Then $N=\Def^P_{P/Q} M$ is also indecomposable (see~\cite{Dade1978a}, Statement (5.3)) and, moreover, if $\End_k (M)\simeq kX$ as $kP$-modules then $\End_k (N) \simeq k(X^Q)$ as $k(P/Q)$-modules. It follows that $\Def^P_{P/Q} [M]=0$ if and only if $|X^Q|=1$.  

\begin{lem}\label{endolem1} If $P$ is abelian and $\Cal S$ is a downward closed set of subgroups of $P$ then 
$$
D_k (P,\Cal S) = \bigcap_{\substack{Q\le P \\ Q\notin \Cal S}} \ker \Def^P_{P/Q}.
$$
\end{lem}

\begin{proof} Suppose $M$ is an indecomposable $\Cal S$-endo-permutation $kP$-module, and let $X$ be a $P$-set such that $\End_k (M)$ is isomorphic to the permutation module $kX$. Then $[M]\in D_k (P,\Cal S)$ if and only if $kX$ has no indecomposable summand of the form $k(P/Q)$ with $Q<P$ and $Q\notin \Cal S$. On the other hand, for $S<P$, $\Def^P_{P/S} [M]=0$ if and only if $kX$ has no indecomposable summand of the form $k(P/Q)$ with $S\le Q<P$. The result follows immediately.
\end{proof}

\begin{lem}\label{endolem2} If $P$ is abelian and $\Cal S$ is a downward closed set of subgroups of $P$ then $D_k^{\Om|\Cal S} (P) = D_k (P,\Cal S)$.
\end{lem}

\begin{proof}
 We use induction of $|\Cal S|$. If $\Cal S=\varnothing$, there is nothing to prove.
Let $S$ be a maximal element of $\Cal S$, and let $\Cal Y=\Cal S - \{ S\}$. By Equation (5.29) of~\cite{Dade1978a},
$$
\bigcap_{Q\notin \Cal S} \ker \Def^{P}_{P/Q} = \Inf_{P/S}^P \left(\bigcap_{S<Q\le P} \ker \Def^{P/S}_{P/Q} \right) \times \bigcap_{Q\notin \Cal Y} \ker\Def^P_{P/Q}.
$$
Due to Lemma~\ref{endolem1}, we may rewrite this as
\begin{equation}\label{endo2}
D_k (P,\Cal S) = (\Inf_{P/S}^P D_k (P/S, \{\mbf 1\})) \times D_k (P, \Cal Y).
\end{equation}
By the inductive hypothesis, $D_k (P, \Cal Y)$ is contained in $D^{\Om|\Cal S}_k (P)$. By~\cite{Dade1978b}, Theorem 10.1, $D_k (Q,\{\mbf 1\})$ is cyclic and is generated by $\Om_{Q/\mbf 1}$ for any finite abelian $p$-group $Q$. Thus $\Inf_{P/S}^P D_k (P/S, \{ \mbf 1 \})$ is generated by $\Om_{P/S}$ and hence is contained in
$D_k^{\Om | \Cal S} (P)$. The result now follows from~\eqref{endo2}.
\end{proof}

Let $A$ be a ring with an identity element. 
Recall that a complex $X$ of $A$-modules is said to be \emph{split} if $X$ is a direct sum of a contractible complex and the complex with the $i$th term equal to $H_i (X)$ having zero differentials. (A complex is said to be \emph{contractible} if it is equivalent to the zero complex in the homotopy category.)

\begin{defi}[Rickard~\cite{Rickard1996}, Section 7.1]
 Let $G$ be a finite group and $M$ be an $RG$-module. An endo-split ($p$-permutation) resolution of $M$ is a bounded complex $X$ of
 ($p$-permutation) $RG$-modules having homology concentrated in degree zero together with an isomorphism between $H_0 (X)$ and $M$ such that $X^*\otimes_R X$ is split as a complex of $RG$-modules (with $G$ acting diagonally).
\end{defi}

The following result is a slight generalisation of~\cite{HarrisLinckelmann2000}, Lemma 1.5(i), and is proved in exactly the same way (cf.\ also~\cite{Rickard1996}, Lemma 7.5).

\begin{lem}\label{dirsummandcomp} Let $X$ be an endo-split resolution of an $\Cal OG$-module $M$. If $N$ is a direct summand of $M$ then $X$ has a direct summand 
$Y,$ unique up to homotopy equivalence, such that $Y$ is an endo-split resolution of $N$.
\end{lem}

The following two results refine an important theorem of Rickard (\cite{Rickard1996}, Theorem 7.2). The proofs below essentially mimic the one in~\cite{Rickard1996}. 

\begin{thm}\label{endothm} Let $P$ be a finite abelian $p$-group. Let $\Cal S$ be a downward closed set of subgroups of $P$. 
If $M$ is an indecomposable capped $\Cal S$-endo-permutation $kP$-module then there is an endo-split $p$-permutation resolution of $M$ that is $\Cal S$-tempered with pivot $k$. 
\end{thm}

\begin{proof} It is shown in the proof of~\cite{Rickard1996}, Theorem 7.2, that 
$$
\cdots \lra 0 \lra kP \stackrel{g \mapsto 1}{\lra} k \lra 0 \lra \cdots, 
$$
where $kP$ is the $0$-term, is an endo-split $p$-permutation resolution of $\Om_{P/\mbf 1}$. This complex is clearly $\{ \mbf 1 \}$-tempered with pivot $k$. Similarly, the complex
$$
\xymatrix@1{
\cdots \ar[r] &  0 \ar[r] & k \ar[rr]^{\!\!\!\!1 \mapsto \sum_{g\in P} g} & & kP \ar[r] & 0 \ar[r] & \cdots, 
}
$$ 
where $kP$ is the $0$-term, is an endo-split $p$-permutation resolution of the indecomposable module $\Om^{-1}_{P/\mbf 1} $ and is also $\{ \mbf 1 \}$-tempered with pivot $k$. Therefore, for every $Q\in \Cal S$, the modules $\Om_{P/Q}$ and $\Om^{-1}_{P/Q}$ have resolutions of the required form (these can be obtained by inflating resolutions as above from $P/Q$ to $P$).

By Lemma~\ref{endolem2}, each element of $D_k (P,\Cal S)$ can be represented as a sum of elements of the form $[\Om_{P/Q}]$ and 
$[\Om_{P/Q}^{-1}]$ with $Q\in \Cal S$. So the result will follow once we show that, if $M$ and $N$ are indecomposable endo-permutation modules that have resolutions of the required form then so does $\capm(M\otimes_k N)$. To see this, let $X$ and $Y$ be resolutions of $M$ and $N$ respectively that satisfy the conditions of the theorem. By~\cite{Rickard1996}, Lemma 7.4, $X\otimes_k Y$ is an endo-split $p$-permutation resolution of $M\otimes N$. It is easy to see that $X\otimes Y$ is $\Cal S$-tempered with pivot $k$. Finally, by Lemma~\ref{dirsummandcomp}, the module $V=\capm(M\otimes N)$ has an endo-split $p$-permutation resolution $T$ that is a direct summand of the complex $X\otimes Y$. Then $T$ is certainly $\Cal S$-tempered, with pivot either $k$ or $0$. However, if $T$ has pivot $0$ then $\dim V = \sum_{i\in \mZ} (-1)^i \dim T^i$ is divisible by $p$, which is impossible because $\dim \End_k (V)\equiv 1 \pmod p$. So $T$ has pivot $k$, and the result follows. 
\end{proof}

%\begin{remark}
% Theorem~\ref{endothm} may be generalised to the case where $M$ is not necessarily indecomposable in the same way as it is done in~\cite{Rickard1996}.
%In this situation the condition on the pivot of the resolution in the statement needs to be amended: the pivot is $k^{\oplus n}$ where $n$ is the multiplicity of $\capm(M)$ as a direct summand of $M$.
%\end{remark}

\begin{cor}\label{endocor} 
Let $P$ be a finite abelian $p$-group and $\Cal S$ be a downward closed set of subgroups of $P$.
 Let $M$ be an $\Cal S$-endo-permutation $RP$-module. Then there is a $1$-dimensional $RP$-module $J$ such that $J\otimes M$ has an
endo-split $p$-permutation resolution that is $\Cal S$-tempered with pivot $R^{\oplus n}$, where $n$ is the multiplicity of $\capm (M)$ as a summand of $M$ if $M$ is capped and $n=0$ if $M$ is uncapped.
\end{cor}

\begin{proof}
 By Theorem~\ref{endothm}, $k\otimes_R \capm(M)\simeq \capm(k\otimes M)$ has an 
$\Cal S$-tempered endo-split $p$-permutation resolution $X$ with pivot $R$. 
By~\cite{Rickard1996}, Proposition 7.1, $X$ can be lifted to an endo-split $p$-permutation resolution $\tilde{X}$ of an $RP$-module $N$ such that 
$k\otimes N\simeq k\otimes M$. By~\cite{Dade1978b}, Proposition 12.1, there is a well-defined homomorphism from $D_R (P)$ onto $D_k (P)$ given by $[U] \mapsto [k\otimes_R U]$. Moreover, the kernel of this homomorphism consists of the classes of $1$-dimensional $RP$-modules, and it follows that $N=J\otimes \capm(M)$ for some $1$-dimensional $RP$-module $J$. By~\cite{Dade1978a}, Theorem 6.10,
$$
J \otimes M \simeq N^{\oplus n} \oplus \bigoplus_{i\in I} \Ind_{Q_i}^G \Res^G_{Q_i} N
$$
where $I$ is some indexing set and $Q_i<P$ for each $i$. For each $i$ we have $R\mid \End_R (\Res^P_{Q_i} N)$, and hence
$$
\Ind_{Q_i}^P R \divi \Ind_{Q_i}^P \End_R (\Res^P_{Q_i} N) \divi \End_R (\Ind_{Q_i}^P \Res^P_{Q_i} N) \divi \End_R (J\otimes M).
$$
Since $J\otimes M$ is an $\Cal S$-endo-permutation module, we deduce that $Q_i\in \Cal S$ for all $i\in I$. 
By the proof of~\cite{Rickard1996}, Lemma 7.6, the complex
$$
Y= X^{\oplus n} \oplus \bigoplus_{i\in I} \Ind_{Q_i}^G \Res^G_{Q_i} X
$$
is an endo-split $p$-permutation resolution of $J\otimes N$. Since $Q_i \in \Cal S$ for each $i$, the complex $Y$ is $\Cal S$-tempered with pivot $R^{\oplus n}$.
\end{proof}

We will use the following known result (cf.\ \cite{Rickard1996}, proof of Theorem 7.2). The proof is left as an exercise.

\begin{lem}\label{splitcrit}
 Let $X$ be a bounded complex of $A$-modules such that $X^i$ is projective for each $i\ne 0$ and the homology of $X$ is concentrated 
in degree $0$ (i.e.\ $H_i (X)=0$ for all $i\ne 0$). Then $X$ is split. 
\end{lem}

Let $G=LP$ be a semidirect product of a normal $p'$-subgroup $L$ and a $p$-group $P$. Let $C=C_L (P)$, so that $H\coleq N_G(P)=C\times P$ (as $[N_L(P),P]$ must be contained in both $L$ and $P$). Let $b$ be a block of $G$ with defect group $P$, and let $e\in \Bl(H)$ be its Brauer correspondent. Then $b\in \Cal OL$ and $e\in \Cal OC$ (see e.g.~\cite{CRI}, Proposition 56.37). Since $|C|$ is coprime to $p$, there is a unique indecomposable $\Cal OCe$-lattice, which we will denote by $Z$. 
Then $\Cal OCe\simeq Z\otimes_\Cal O Z^*$, whence
$\Cal OHe\simeq (Z\otimes \Cal OP) \otimes Z^*$. We consider the $\Cal OHe$-$\Cal OP$-lattice $Z\otimes \Cal OP$ appearing in this tensor product. Its bimodule structure is given by
\begin{equation}\label{ZOP}
(cy_1)(z\otimes x)y_2 = (cz) \otimes (y_1 x y_2), \qquad c\in C, \; z\in Z, \; x, y_1, y_2 \in P. 
\end{equation}
Let $\widetilde{\fr{Gr}} (G,b)$ be the $\Cal OG$-$\Cal OP$-bimodule that is the Green correspondent of $Z\otimes \Cal OP$ (with respect to the triple $(G\times P, H\times P, P\times P)$). This notation is justified by the following observation.

\begin{lem}\label{Grprod}
 We have $\fr{Gr}(G,b) \simeq \widetilde{\fr{Gr}}(G,b) \otimes_{\Cal OP} (\Cal OP \otimes Z^*)$. 
\end{lem}

\begin{proof}
 We have isomorphisms of $\Cal OH$-$\Cal OH$-bimodules
$$
\Cal OHe \simeq \Cal OCe \otimes \Cal OP \simeq (Z \otimes Z^*) \otimes \Cal OP \simeq (Z\otimes \Cal OP) \otimes Z^*.
$$
Here $(Z\otimes \Cal OP)\otimes Z^*$ may be viewed as an $\Cal O(H\times H)=\Cal O((H\times P)\times C)$-module given as the outer tensor product of the $\Cal O(H\times P)$-module
$Z\otimes \Cal OP$ and the $\Cal OC$-module $Z^*$. By \cite{Kuelshammer1992}, Proposition 1.2, and~\cite{Harris2005}, Proposition 3.4, which essentially state that vertices, sources and Green correspondents behave well with respect to outer tensor products, we have
$$
\fr{Gr}(G,b) \simeq \widetilde{\fr{Gr}}(G,b) \otimes Z^*.
$$
Finally, the isomorphism $\widetilde{\fr{Gr}}(G,b)\otimes Z^*\simeq \widetilde{\fr{Gr}}(G,b) \otimes_{\Cal OP} (\Cal OP \otimes Z^*)$ is clear. 
\end{proof}

If $\Cal S$ is a set of subgroups of $G$, we will say that two $RG$-modules $M$ and $N$ are \emph{$\Cal S$-equivalent} if there exist $\Cal S$-projective $RG$-modules $M_0$ and $N_0$ such that $M\oplus M_0 \simeq N \oplus N_0$. The next result is a refinement of~\cite{Rickard1996}, Theorem 7.8, and most steps of the proof below are the same as in \emph{loc.\ cit.}

\begin{thm}\label{pnilpspl} With the notation as above, assume that $P$ is abelian. Let $\Cal S = \Cal S(G,P,H)$. 
Then there exists a splendid tilting complex of $\Cal OGb$-$\Cal OP$-bimodules that is $\D \Cal S$-tempered with pivot $\widetilde{\fr{Gr}} (G,b)$. 
\end{thm}

\begin{proof}
Observe that $\Cal S= \{ Q\le P \mid C_L (Q) > C \}$. Indeed, since $H=CP$, if 
$Q=P\cap \lsa{g}{P}$ with $g\in L-C$, then $g$ centralises $Q$ because $[Q,g]$ must be contained both in $P$ and in $L$ (as $L$ is normal).

 The algebra $\Cal OLb$ has a unique indecomposable lattice, which can be extended to an $\Cal OGb$-lattice $M$, say. (This extendibility property is well known and follows from extendibility of the corresponding character of $L$; see~\cite{IsaacsBook}, Corollary 6.28.) Since the $\Cal OP$-module $\End_{\Cal O} (\Res^G_P M) \simeq \Cal OLb$ is a direct summand of $\Cal OL$ (with $P$ acting by conjugation), 
we see that $\Res^G_P M$ is an $\Cal S$-endo-permutation module.
 Let $U=\capm (\Res^G_P M)$ (so that $U$ is the source of $M$). By Corollary~\ref{endocor}, for some $1$-dimensional $\Cal OP$-module $J$, 
there is an endo-split $p$-permutation resolution $X$ of $J\otimes U$ that is $\Cal S$-tempered with pivot $\Cal O$. Replacing $M$ with $J\otimes M$, we may assume that $X$ is, in fact, a resolution of $U$. 

Hence $Z\otimes X$ is an endo-split resolution of the $\Cal OHe$-module $Z\otimes U$ (where both tensor products are outer ones).
Moreover, $Z\otimes X$ is $\Cal S$-tempered with pivot $\tilde{Z}=\Inf_{C}^H Z$. Let $Y=\Ind_H^G (Z\otimes X)$. Then $Y$ is a complex of $p$-permutation $\Cal OG$-modules that is $\Cal S$-tempered with pivot $V$, where $V$ is the Green correspondent of $\tilde{Z}$. The homology of $Y$ is concentrated in degree $0$, and $H_0 (Y) \simeq \Ind_H^G (Z\otimes U)$. 

We claim that $\End_{\Cal O} (Y)$ is split as a complex of $\Cal OG$-modules with the diagonal action of $G$ (cf.~\cite{Rickard1996}, proof of Lemma 7.7). We have 
\[
 \End_{\Cal O} (Y) \simeq \Res^{G\times G}_{\D G} \Ind_{H\times H}^{G\times G} ((Z\otimes X)^* \otimes (Z\otimes X)).
\]
By the Mackey formula, the right-hand side is a direct sum of complexes of the form
\[
 \Ind_{\D (H\cap \lsa{g}{H})}^{\D G} \Res^{H\times \lsa{g}{H}}_{\D(H\cap \lsa{g}{H})} ((Z^* \otimes X^*) \otimes \lsa{g}{(Z\otimes X)})
\]
for certain $g\in G$. Since $|(H\cap \lsa{g} H): (P\cap\lsa{g}{P})|$ is coprime to $p$, a complex of $\Cal O(H\cap\lsa{g} H)$-modules splits if and only if its restriction to $\Cal O(P\cap\lsa{g}P)$ splits. Hence, to prove the claim it suffices to show that 
\begin{equation}\label{pnseq}
\Res^{H\times \lsa{g}{H}}_{\D (P\cap \ls{g}{P})} ((Z^* \otimes X^*) \otimes (\lsa{g}Z \otimes \lsa{g}X))
\end{equation}
splits for each $g\in G$. However, as above, all elements of $P\cap \lsa{g}P$ commute with $g$. Thus the module~\eqref{pnseq} is isomorphic to $\Res^P_{P\cap\lsa{g}{P}} ((Z^* \otimes Z) \otimes (X^* \otimes X))$. Since $P$ acts trivially on $Z$ and $X^*\otimes X$ is split as a complex of $\Cal OP$-modules, the claim follows. To summarise, $Y$ is an endo-split $p$-permutation resolution of $\Ind_H^G (Z\otimes U)$ that is $\Cal S$-tempered with pivot $V$.

Since $M$ is the Green correspondent of $Z\otimes U$, there is a direct summand $Y'$ of $Y$ that is an endo-split $p$-permutation resolution of $M$ (by Lemma~\ref{dirsummandcomp}). 
The $\Cal S$-pivot of $Y'$ must be either $0$ or $V$. However, if $Y'$ has pivot $0$ then $\sum_{i} (-1)^i\dim (Y')^i$ is divisible by $p$, which is impossible because the homology of $Y'$ is concentrated in degree zero and $H_0 (Y')\simeq M$, whereas $p$ does not divide $\dim M$. So the $\Cal S$-pivot of $Y'$ is $V$ and, in particular, is $\Cal S$-equivalent to $\Ind_H^G (\tilde{Z})$.

We identify $G$ with 
\[
\D_P G = \{ (lx, x)\in G\times P \mid l\in L, x\in P \} 
\]
 and $H$ with $\D_P H=\{ (cx,x)\in H\times P \mid c\in C, x\in P\}$. 
Consider the complex $T=\Ind_{\D_P G}^{G\times P} (Y')$ of $\Cal OG$-$\Cal OP$-bimodules. 
Certainly, $T$ is $\D \Cal S$-tempered with pivot $\D \Cal S$-equivalent to $\Ind_{\D_P H}^{G\times P} (\tilde{Z})$. However, 
$$\Ind_{\D_P H}^{H\times P} (\tilde Z) \simeq Z\otimes \Cal OP,$$ 
where the module structure on the right-hand side is as described by~\eqref{ZOP}; 
an isomorphism is given by
$$
(x,1) \otimes z \mapsfrom x\otimes z, \quad x\in P,\; z\in Z.
$$
So the $\Cal S$-pivot of $T$ is isomorphic to the Green correspondent of $Z\otimes \Cal OP$, i.e.\ to 
$\widetilde{\fr{Gr}}(G,b)$. Also, all terms of $T$ are $p$-permutation bimodules because all terms of $Y'$ are $p$-permutation $\Cal OG$-modules. 

We will now verify that $T$ is a Rickard tilting complex of $\Cal OGb$-$\Cal OP$-bimodules, thus completing the proof. The homology of $T$ is concentrated in degree $0$, and $H_0 (T) \simeq \Ind_{\D_P G}^{G\times P} (M) = N$, say. As is observed in~\cite{Rickard1996}, Section 7.4, the bimodule $N$ induces a Morita equivalence between the algebras $\Cal OGb$ and $\Cal OP$, that is, there are isomorphisms
\[
 N\otimes_{\Cal OP} N^* \simeq \Cal OGb \quad \text{ and } \quad N^* \otimes_{\Cal OG} N \simeq \Cal OP
\]
of $\Cal OGb$-$\Cal OGb$- and $\Cal OP$-$\Cal OP$-bimodules respectively. By the Mackey theorem, 
\[
\Res^{G\times P}_{G\times \mbf 1} T \simeq \Ind_{L\times \mbf 1}^{G\times P} \Res^{\D_P G}_{L\times \mbf 1} (Y'),
\]
 so the terms of 
$T$ are projective as left 
$\Cal OG$-modules. Since the homology of $T$ is concentrated in degree $0$, by Lemma~\ref{splitcrit}, $T$ splits as a complex of left $\Cal OG$-modules, whence $T^*\otimes_{\Cal OG} T$ has homology concentrated in degree $0$ with 
$H_0 (T^*\otimes_{\Cal OG} T) \simeq N^*\otimes_{\Cal OG} N \simeq \Cal OP$ as $\Cal OP$-$\Cal OP$-bimodules. Similarly, all the terms of $T$ are projective as right $\Cal OP$-modules and $T\otimes_{\Cal OP} T^*$ has homology concentrated in degree $0$ and isomorphic to 
$N\otimes_{\Cal OP} N^* \simeq \Cal OGb$ as an $\Cal OG$-$\Cal OG$-bimodule. 

It remains to show that $T\otimes_{\Cal OP} T^*$ and $T^*\otimes_{\Cal OG} T$ are split. Let $(\D_P G)^* =\{(x,lx) \mid x\in P, l\in L\}\subseteq P\times L$. We have
$$
T \otimes_{\Cal O} T^* \simeq \Ind_{\D_P G\times (\D_P G)^*}^{G\times P\times P \times G} (Y'\otimes_{\Cal O} (Y')^* ),
$$
and, by the Mackey theorem, the restriction of this to $G\times \D P \times G$ is isomorphic to
\[
 \Ind_{E}^{G\times \D P \times G} \Res^{G\times G}_E (Y' \otimes_{\Cal O} (Y')^* )
\]
where $E=\{ (lx,xl') \mid l,l'\in L, x\in P \} \subseteq G\times G$ and $E$ embeds into $G\times \D P \times G$ via
\[
 (lx,xl')\mapsto (lx,x,x,xl'), \quad l,l'\in L, \, x\in P.
\]
Since $Y'\otimes_{\Cal O} (Y')^*$ is split as a complex of $\Cal OG$-$\Cal OG$-bimodules, it follows that $T\otimes_{\Cal O} T^*$ is split when viewed as a complex of $\Cal O(G\times \D P\times G)$-modules. Therefore, $T\otimes_{\Cal OP} T^*$ is split as a complex of $\Cal OG$-$\Cal OG$-bimodules. The proof that $T^* \otimes_{\Cal OG}  T$ is split as a complex of $\Cal OP$-$\Cal OP$-bimodules is given in~\cite{Rickard1996}, proof of Theorem 7.8, and is similar.
\end{proof}

\begin{cor} Let $G=L\rtimes P$ where $L$ is a finite $p'$-group and $P$ is a finite abelian $p$-group. Then Conjecture~\emph{\ref{rBr}} holds for all $p$-blocks of $G$ with defect group $P$. 
\end{cor}

\begin{proof}
 We use the notation of the discussion preceding Theorem~\ref{pnilpspl} and of the proof above. Let $T$ be the complex of $\Cal OGb$-$\Cal OP$-bimodules given by Theorem~\ref{pnilpspl}. Clearly, $T'=T \otimes_{\Cal OP} (\Cal OP \otimes Z) \cong T\otimes_{\Cal O} \tilde{Z}$ 
is a Rickard tilting complex of $\Cal OGb$-$\Cal OHe$-bimodules, 
and all its terms are $p$-permutation bimodules. Since $H=CP$ controls the (non-existent) $G$-fusion of subgroups of $P$, it follows from Lemma~\ref{comp} that $T'$ is $\D \Cal S$-tempered with pivot $\widetilde{\fr{Gr}}(G,b) \otimes_{\Cal OP} (\Cal OP \otimes Z) \simeq \fr{Gr}(G,b)$, where the isomorphism is due to Lemma~\ref{Grprod}.
\end{proof}

\begin{remark}
Let $G=L\rtimes P$ be a $p$-nilpotent group.
 A character-theoretic ``shadow'' of the Brou{\'e} conjecture is the existence of a
so-called perfect isometry (see~\cite{Broue1990}) between $\Cal C(G,b)$ and $\Cal C(N_G (P),e)$, where $b$ and $e$ are the blocks in question. 
In the present case, it is not difficult to prove the existence of such an isometry purely by methods of character thory (i.e.\ via the Glauberman correspondence; see~\cite{IsaacsBook}, Chapter 13). By contrast, there appears to be no easy character-theoretic proof of the fact that (IRC-Syl) holds when $P$ is abelian. 
\end{remark}

\begin{remark}\label{pnilprem}
 Consider a semidirect product $G=L\rtimes P$ with $P$ not necessarily abelian (and $L$ a $p'$-group, as before). 
In this situation, the statement of Lemma~\ref{endolem2} is no longer true: for example, if $p>2$ and $P$ is an extraspecial group of order $p^3$ and exponent $p$, 
then $\mathbb Q\otimes_{\mathbb Z} D_k (P, \{ \mbf 1 \})$ is $(p+1)$-dimensional by~\cite{Alperin2001}, Theorem 4, but $D_k^{\Om|\Cal \{ \mbf 1 \} } (P)$ is cyclic, so 
$
 D_k (P, \{ \mbf 1 \}) \nsubseteq D_k^{\Om|\Cal \{ \mbf 1 \} } (P).
$
 By a result of Puig (\cite{Puig2009}, Theorem 7.8), with notation as above, the source of the unique simple $kGb$-module necessarily yields a torsion element of the  Dade group, so the proof of Theorem~\ref{pnilpspl} would still work as long as
\begin{equation}\label{nonab2}
D_k (P,\Cal S) \cap D_{k,t} (P) \subseteq D^{\Om |\Cal S}_k (P),
\end{equation}
where $D_{k,t}(P)$ denotes the torsion subgroup of $D_k (P)$. However, for $p=2$, there are torsion elements of the Dade group which do not even belong to the subgroup $D_k^{\Om} (P)$ spanned by all relative syzygies, and in such a case property (IRC-Syl) may fail, as is shown in Section~\ref{exampleQ8} below.
It is not clear whether~\eqref{nonab2} is true for odd $p$. 
\end{remark}

\section{The trivial intersection case}\label{TI}

\subsection{Properties (P+) and (G)}\label{Eaton}

Let $G$ be a finite group. Let $P$ be a $p$-subgroup of $G$ and $H$ be a subgroup of $G$ containing $N_G (P)$. In this section we mostly concentrate on the case where 
$\Cal S = \Cal S(G,P,H) =\{ \mbf 1 \}$. When this occurs for $H=N_G (P)$, it is said that $P$ is a \emph{trivial intersection} (or \emph{TI}) subgroup of $G$.

Denote by  $\Cal{P}(G)\le \Cal C(G)$ the abelian 
group spanned by the characters of projective indecomposable $\O G$-modules. In other words,
\begin{equation}\label{Pdesc}
\Cal P(G) = \{ \chi\in \Cal C(G) \mid \chi(g)=0 \; \text{for all} \; g\in G \text{ with } g_p\ne 1 \}
\end{equation}
(see e.g.\ \cite{NavarroBook}, Corollary 2.16), and also $\Cal P(G)=\Cal I(G,P,\{ \mbf 1 \}$) by Theorem~\ref{relpr}. (Note that $\Cal I(G,P,\{ \mbf 1\})$ does not actually depend on $P$.) If $b$ is a block of $G$, write $\Cal P(G,b)= \Cal P(G) \cap \Cal C(G,b)$.

Assume Hypothesis~\ref{hypbl}. Recall the virtual character $\om(G,b,H)$ of $G\times H$ defined by~\eqref{omega}. The following property of the quadruple $(G,b,P,H)$ has been defined by Eaton~\cite{Eaton2008} (in the case $H=N_G (P)$) and considered, in particular, in cases where $P$ is TI. (We retain the name of the property from~\cite{Eaton2008}.)

\bigskip
\noindent\begin{tabular}{lp{\width}}
\!\!\!(P+)\!\!\! & 
\emph{
\!\!\!\! Let $e\in \Bl(H)$ be the Brauer correspondent of $b$. There exists $\mu\in \Cal (G\times H, b\otimes \bar e)$ of the form 
$\mu=\om(G,b,H) + \sum_{i} \alpha_i \times \beta_i$, where $\alpha_i\in \Cal P(G,b)$ and $\beta_i \in \Cal P(H,\bar e)$ for each $i$, such that, for every 
$\phi\in\Irr_{0}(H,e)$ and $\chi\in \Irr_{0}(G,b)$, the virtual characters $I_{\mu}(\phi)$ and $R_{\bar \mu}(\chi)$ each have precisely one irreducible constituent of height zero, and this occurs with multiplicity $\pm 1$.}
\end{tabular}
\medskip

Using the following observation, we may replace the condition on $\mu$ in (P+) with the congruence 
$\mu\equiv \om(G,b,H) \,\bmod \Cal P(G\times H, b\otimes \bar e)$.

\begin{lem}\label{simlem}
 Let $G$ and $L$ be finite groups and $\nu\in \Cal C(G\times L)$. Then $\nu \in \Cal P(G\times L)$ if and only if 
$\nu$ can be expressed as $\sum_{i} \alpha_i\times \beta_i$ where $\alpha_i\in \Cal P(G)$ and $\beta_i\in \Cal P(L)$ for each $i$. 
\end{lem}

\begin{proof}
 Suppose $\nu\in \Cal P(G\times L)=\Cal I(G\times L, \mbf 1, \{ \mbf 1 \})$. Then $\mu$ is a $\mZ$-linear combination of characters of the form 
$\Ind_{E}^{G\times L} \phi$ where $E$ is a $p'$-subgroup and $\phi \in\Irr(E)$. If $E_1$ and $E_2$ are the projections of $E$ onto $G$ and $L$ then 
$\Ind_{E}^{G\times L} \phi=\Ind_{E_1\times E_2}^{G\times L} \Ind_{E}^{E_1\times E_2} \phi$ 
is a sum of characters of the form $(\Ind_{E_1}^G \psi_1)\times (\Ind_{E_2}^L \psi_2)$, so $\nu$ can be expressed as required. 
The converse is clear.
\end{proof}

Following the approach of Sections~\ref{intro} and~\ref{gensetup}, we may generalise (P+) to the case where $P$ is not necessarily TI, and thus consider, assuming Hypothesis~\ref{hypbl}, the following property of the quadruple $(G,b,P,H)$:

\bigskip
\noindent\begin{tabular}{lp{\wid1}}
\!\!\!(G)\!\!\! & 
\emph{ \!\!\!\!
There exists $\mu\in \Cal (G\times H, b\otimes \bar e)$ of the form 
$$
\mu \equiv \om(G,b,H) \dmod \Cal I(G\times H,\D P, \D \Cal S)
$$ 
such that for each 
$\phi\in\Irr_0(H,e)$ and $\chi\in \Irr_0(G,b)$, the virtual characters $I_{\mu}(\phi)$ and $R_{\bar\mu}(\chi)$ each have precisely one irreducible constituent of height zero, and this occurs with multiplicity $\pm 1$.}
\end{tabular}
\medskip

By Lemma~\ref{simlem}, if $\Cal S(G,P,H) =\{ \mbf 1 \}$ then (G) is equivalent to (P+). 

%We will show that in the TI case the property (P+) is closely related to Conjecture~\ref{conjBd}. We begin by stating a ``block version'' of 
% Conjecture~\ref{conjBd}, which holds whenever Conjectures~\ref{conjBd} and~\ref{conjCd} both hold (by Theorem~\ref{corrthmgen}).

%\begin{conj}\label{conjblocks}
 %Assume Hypothesis~\ref{hyp2}. 
%Then there exists a signed bijection $F:\pm \Irr_0 (G,b)\ra \pm\Irr_0 (H,e)$ such that
%$$
% F(\chi)  \equiv  \Proj_e \Res^G_H \chi \mod \; \; (\Cal C^p_P (H) + \Cal I(H,P,\Cal S)) \cap \Cal C(H,e) 
%$$
%for all $\chi\in \pm \Irr_{0}(G,b)$ and 
%$$
%F^{-1} (\phi)  \equiv  \Proj_b \Ind_H^G \phi \mod \;\; (\Cal C^p_P (G)+\Cal I(G,P,\Cal S))\cap \Cal C(G,b) 
%$$
%for all $\phi\in \pm \Irr_0 (H,e)$.
%\end{conj}

\begin{prop}\label{Eawirc}
Assume Hypothesis~\emph{\ref{hypbl}}. If \emph{(G)} holds for the quadruple $(G,b,P,H)$ then so do \emph{(WIRC-Bl)} and \emph{(WIRC*-Bl)}.
\end{prop}

\begin{proof}
Suppose $\mu$ witnesses (G) and define $F\colon \pm\Irr_0 (G,b) \ra \pm\Irr_0 (H,e)$ by setting $F(\chi)=\lan \phi,R_{\bar\mu} (\chi) \ran \phi$ where $\phi$ is the unique irreducible constituent of $R_{\bar \mu} (\chi)$ lying in $\Irr_0 (H,e)$.
Since $I_{\mu}$ and $R_{\bar{\mu}}$ are adjoint, it is easy to see that $F$ is a signed bijection.
It follows from Lemma~\ref{mu2} that this bijection satisfies the requirements of (WIRC-Bl) and (WIRC*-Bl).
\end{proof}

There appears to be no obvious reason why the converse to this observation might be true. 

Property (P+) is proved in~\cite{Eaton2008} for a number of pairs $(G,b)$, with respect to the normaliser of a defect group of $b$, 
including the following cases:
\begin{enumerate}[(i)]
\item $G$ is one of $\SU_3 (q)$, $\GU_3(q)$, $\SU_3(q).2$ and $\GU_3(q).2$, where the extensions are by a field automorphism of order $2$, $p$ is the defining characteristic (so $q$ is a power of $p$) and $b$ is any block of positive defect;
 \item $G=\lsa{2}{B}_2 (2^{2m+1})$, $p=2$ and $b$ is the principal block;
\item $G=\ls{2}{G}_2 (3)$, $p=3$ and $b$ is the principal block;
% \item $G=\Aut(\ls{2}B_2 (32))$, $p=5$ and $b$ is the principal block; 
\item $G=3.McL$ (the perfect triple cover of the sporadic simple group $McL$), $p=5$, and $b$ is any of the three blocks of positive defect.
\end{enumerate}

 We remark that the evidence of Section~\ref{calc} suggests that property (IRC) is false in many, if not all, of the cases (i)--(iv).

\subsection{An example}\label{exampleQ8}

In view of the results of Section~\ref{pnilpab}, it is instructive to consider an irreducible character $\chi$ of a $p$-nilpotent group $G=L\rtimes P$
such that an $\Cal OG$-lattice affording $\chi$ has vertex $P$ and a source $U$ with the property that the corresponding element 
$[k\otimes U]$ of the Dade group $D_k (G)$ does not belong to $D^{\Om|\Cal S}_k (G)$, where $\Cal S=\Cal S(G,P,N_G(P))$. 
Mazza~\cite{Mazza2003} showed that every indecomposable endo-permutation $\Cal OP$-module that gives rise to a torsion element in the 
Dade group occurs as a source of an $\Cal OG$-lattice for some $p$-nilpotent group $G=L\rtimes P$ (where $|L|$ is coprime to $p$). In what follows, we consider an example from~\cite{Mazza2003} where the source $U$ is a so-called ``exotic'' endo-permutation module, so that $[k \otimes U]$ does not belong to the subgroup $D^{\Om}_k (P)$ of $D_k (P)$ generated by all relative syzygies (cf.\ Remark~\ref{pnilprem}). 

Let $p=2$ and $P$ be the quaternion group of order $8$, so that
$$
P= \lan u,v \mid u^4 = 1, v^2 = u^2, \, \ls{u}{v}=v^{-1} \ran.
$$
Let $L$ be the extra-special group of order $125$ with exponent $5$, so that
$$
L= \lan x,y,z \mid x^5=y^5=z^5=1, \; [x,y]=z, [y,z]=[x,z]=1 \ran.
$$
Let $C=Z(L)=\lan z \ran$. 
Let $A$ be the subgroup of $\Aut(L)$ consisting of the maps $\tau$ such that $\tau(z)=z$. Then $A$ can be identified with $\SL_2 (5)$ where elements of $\SL_2 (5)$ are viewed as endomorphisms of the $\mF_5$-vector space $L/C$ with respect to the basis $\{ xC, yC \}$. Moreover, there is an isomorphism between $P$ and the Sylow $2$-subgroup of $A$ given by 
$$
u\mapsto \begin{pmatrix}
a & 0 \\
0 & a^{-1} 
\end{pmatrix}
\quad \text{and} \quad
v\mapsto
\begin{pmatrix}
0 & 1 \\
-1 & 0
\end{pmatrix}
$$
where $a$ is a generator the multiplicative group $\mF_5^{\times}$ (e.g.\ $a=2$). This isomorphism defines an action of $P$ on $L$, and we consider the corresponding semidirect product $G=L \rtimes P$. We have $C_{L}(x)=C$ for all $g\in P-\mbf 1$ (because some power of $g$ maps to the negation of the identity matrix in $\SL_2 (5)$), so $C=C_L (P)$ and $P$ is TI in $G$. As in Section \ref{pnilpab}, we have $N_G (P) =CP$.

Let $\phi$ be a non-trivial linear character of $C$, and let $\chi\in \Irr(L \di \phi)$. One can easily see that $\Ind_C^L \phi=5\chi$ using Clifford theory, and therefore $\chi$ is $G$-invariant. We have $\Res^G_L \chi=5\phi$, whence $\phi$ is the Glauberman correspondent of $\chi$ with respect to $P$ (see~\cite{IsaacsBook}, Chapter 13). Let $\tilde{\chi}$ be a character of $G$ extending $\chi$. Then $\Res^G_{CP} \tilde{\chi}=\phi\times \th$ for some $\th\in \Cal C(P)$ with $\th(1)=5$. Let $\rho=\rho_P$ be the regular character of $P$ and $\b$ be the unique irreducible character of $P$ of degree $2$.  By~\cite{IsaacsBook}, Theorem 13.6, we have $\th(g)=\pm 1$ for each $g\in P-\mbf 1$.
Using this information (or otherwise), one can easily show that $\th = \rho - \a - \b$ for some linear character $\a$ of $P$. Replacing $\tilde{\chi}$ with $\tilde{\chi}\bar{\a}$, we can (and do) ensure that 
$
\th = \rho - 1_P - \b.
$ 
Let $b$ be the block of $G$ containing $\chi$, so that $\Irr(G,b)=\{ \tilde{\chi}\g \mid \g \in \Irr(P) \}$, and the Brauer correspondent $e$ of $b$ is the block of $CP$ whose irreducible characters are the ones of the form $\phi\times \g$, $\g \in \Irr(P)$. First note that property (IRC-Bl) fails for the pair $(G,b)$
with respect to $CP$ because $\th$ cannot be expressed as $\a + t\rho$ with $\a$ a linear character of $P$ and $t\in \mZ$.

Observe that $\om(G,b,CP)=\sum_{\g\in \Irr(P)} (\tilde{\chi}\g) \times \ol{(\phi \times \th\g)}$ and consider
$\mu=\om(G,b,CP)-(\tilde{\chi}\rho)\times (\ol{\phi\times \rho})$. Then  
\[
R_{\bar\mu} (\tilde{\chi}\a)= \phi\times (- \a -\b) \quad \text{and} \quad I_{\mu} (\phi \times \a)=\tilde{\chi}(-\a-\b)
\]
for all linear characters $\a$ of $P$, so $\mu$ is a witness to property (G).
By Proposition~\ref{Eawirc}, property (WIRC-Bl) holds too. Also, $\Res^P_Q (\tilde{\chi}\b) = \phi\times (\b+\rho)$ and 
$\Ind_Q^P (\phi\times \b)=\tilde{\chi}(\b+\rho)$, so properties (pRes-Bl) and (pInd-Bl) are satisfied in this case. 

\begin{remark} Suppose $G=L\rtimes P$ where $L$ is a $p'$-group and $P$ is a TI $p$-subgroup of $G$. 
Suppose $\chi\in\Irr(L)$ is fixed by $P$. Let $M$ be an $\Cal OG$-lattice affording an extension of $\chi$ to $G$, and let $U$ be a source of $M$. Let $\th$ be the character afforded by $U$. Since $U$ is an endo-trivial module, $\th(g)\th(g^{-1})=1$ for each $g\in P-\mbf 1$. Using this, one can show that either $\th(1)=\pm \a + m\rho_P$ for some linear character $\a$ of $P$ and $m\in\mZ$ or $p=2$ and $P$ is dihedral, quaternion or semi-dihedral. In the latter case $\th$ must be one of a certain explicit list of characters. (We omit the details.) It then easily follows that $G$ satisfies (P+) for blocks of maximal defect, as well as properties (pRes-Syl) and (pInd-Syl) (where all three properties are considered with respect to  $N_G (P)$). An alternative way of proving this is to use \cite{Puig2009}, Theorem 7.8,
which implies that $[k\otimes U]$ must be a torsion element of the Dade group (cf.\  Remark~\ref{pnilprem}). It then follows 
that if $\dim U>1$ then $P$ must be cyclic, semi-dihedral or quaternion (see~\cite{Thevenaz2007}, Proposition 6.1).
\end{remark}

%\begin{lem} Let $P$ be a finite $p$-group, and denote by $\rho$ the regular character of $P$. Suppose $\chi\in \Cal C(P)$ satisfy $|\chi(g)|=1$ %for all $g\in G^{\#}$. If $p$ is odd then $\chi=\pm \lda + n \rho$ for some linear character $\lda$ of $P$ and some $n\in \mZ$. If $p=2$ then 
%\end{lem}

%However, 
%$R_{\mu} (\tilde{\chi}\b) = -\b - \a_1-\cdots - \a_4$, where $\a_1,\ldots,\a_4$ are the linear characters of $P$, so 
%$R_{\mu} (\tilde{\chi} \b) \notin \Cal C^p (P)$, and we can say that $\mu$ does not "witness" property (pRes-Bl). On the other hand, the character 
%$\mu'=\om(G,b)$ does witness both (pRes-Bl) and (pInd-Bl), in the sense that $R_{\mu'} (\g)\in \Cal C^p (CP)$ 
%for every $\g\in\Irr^p (G,b)$ and $I_{\mu'}(\zeta) \in \Cal C^p (G)$ for every $\zeta\in \Irr^p (CP,e)$ (in the present case, there is only one %possibility for $\g$ and only one for $\zeta$).

\section{Groups of Lie type in the defining characteristic}\label{lietype}

Let $q$ be a power of our prime $p$.
Let $\mbf G$ be a connected reductive algebraic group over the algebraic closure of $\mF_p$. We assume that $\mbf G$ is defined over $\mF_q$. Let $F\colon\mbf G\ra \mbf G$ be the corresponding Frobenius morphism. Note that this assumption excludes the possibility that $\mbf G^F$ is a Suzuki group $\lsa{2}B_2 (2^{2m+1})$ (with $p=2$) or a Ree group $\lsa{2}F_4 (2^{2m+1})$ or $\ls{2}G_2 (3^{2m+1})$ (with $p=2$ or $3$ respectively). For the theory of the groups $\mbf G^F$ and their characters we refer the reader to~\cite{CarterFGLT} or~\cite{DigneMichelBook}.

 Fix an $F$-stable Borel subgroup $\mbf B$ of $\mbf G$ and an $F$-stable maximal torus $\mbf T$ in $\mbf B$, and let $\mbf U$ be the unipotent radical of $\mbf B$. 
It is well known that $\mbf U^F$ is a Sylow $p$-subgroup of $\mbf G^F$ and that $\mbf B^F=N_{\mbf G^F} (\mbf U^F)$. 

Let $\Phi$ be the root system of $\mbf G$ and $\Phi^+$ be the set of positive roots corresponding to $\mbf B$.
The prime $p$ is said to be \emph{good} for $\Phi$ (or for $\mbf G$) if, for each $\a\in \Phi^+$, no coefficient of the linear combination expressing $\a$ in terms of simple roots is divisible by $p$. Specifically, $p$ is good for $\Phi$ if and only if none of the following holds (see e.g.~\cite{CarterFGLT}, Section 1.14):
\begin{itemize}
 \item $\Phi$ has a component of type $B_l$, $C_l$ or $D_l$ and $p=2$;
 \item $\Phi$ has a component of type $G_2$, $F_4$, $E_6$ or $E_7$ and $p\in \{2,3\}$;
\item $\Phi$ has a component of type $E_8$ and $p\in \{ 2,3,5\}$. 
\end{itemize}

Let $\Cal L\colon \mbf T\ra \mbf T$ be the Lang map, which is defined by $\Cal L (t)=t^{-1} F(t)$. Since $Z(\mbf G)$ is abelian, we may identify $H^1 (F, Z(\mbf G))$ with $Z (\mbf G)/\Cal L(Z (\mbf G))$. Note that $H^1 (F,Z(\mbf G))$ is a quotient of $Z(\mbf G)/Z(\mbf G)^{\mathrm o}$, where $Z(\mbf G)^{\mathrm o}$ denotes the connected component of $Z(\mbf G)$. Recall that $\mbf G$ is said to be \emph{split} if each root subgroup of $\mbf G$ with respect to $\mbf T$ is $F$-stable. 

The main aim of this section is to prove the following results.

\begin{thm}\label{liesplit}
 Suppose $\mbf G$ is split. Assume that $p$ is good for $\mbf G$ and $H^1 (F, Z(\mbf G))$ is cyclic. Then \emph{(IRC-Syl)} and \emph{(pInd-Syl)} hold 
for $\mbf G^F$ with respect to the prime $p$ and the normaliser $\mbf B^F$ of $\mbf U^F$. 
\end{thm}

\begin{thm}\label{lienonsplit} Assume that $H^1 (F,Z(\mbf G))$ is cyclic. The properties \emph{(pRes-Syl)} and \emph{(WIRC-Syl)} hold for the pair $(\mbf G^F, \mbf B^F)$ (with respect to the defining characteristic $p$)
unless one of the following holds:
\begin{enumerate}[(i)]
 \item $q=2$ and $\Phi$ has an irreducible component of type $B_l$, $C_l$, $F_4$ or $G_2$;
 \item $q=3$ and $\Phi$ has an irreducible component of type $G_2$.
\end{enumerate}
\end{thm}

Note that if $\mbf G$ is simple then $Z(\mbf G)/Z(\mbf G)^{\mathrm o}$ is cyclic except in the case when $\mbf G$ is the simply-connected group of type $D_{2l}$ for some integer $l\ge 2$ (see e.g. \cite{CarterFGLT}, Section 1.19). Thus, by Theorems~\ref{liesplit} and~\ref{lienonsplit}, properties (IRC-Syl), (pInd-Syl) and (pRes-Syl) hold for $\mbf G^F$ with respect to the prime $p$ and the subgroup $\mbf B^F$ whenever $\mbf G$ is simple, split and not of type $D_{2l}$ and $p$ is good for $\mbf G$. Also, by Proposition~\ref{normsub}, the same properties hold for the quotient $\mbf G^F/Z(\mbf G^F)$ under the same assumptions on $\mbf G$.

%Therefore, by Theorem~\ref{lienonsplit} and Proposition~\ref{normsub}, properties (pRes-Syl) and (WIRC-Syl), considered with respect to the normaliser of a Sylow $p$-subgroup, are true for all simple groups of Lie type defined in characteristic $p$, with the exception
% of the quotients of $\mbf G^F$ where $\mbf G$ is simply connected of type $D_{2l}$, cases listed under (i) and (ii) in Theorem~\ref{lienonsplit}, and Suzuki and Ree simple groups. (Recall that results for groups $\ls{2}B_2 (2^{2m+1})$ and $\ls{2}{G}_2 (3)$ have been proved by Eaton~\cite{Eaton2008}; see Section~\ref{Eaton}.)

Theorem~\ref{lienonsplit} is a generalisation of a result due to Brunat (\cite{Brunat2009}, Theorem 1.1). We use many of the same ingredients as the proof in~\cite{Brunat2009}, but avoid explicit computations.

% Denote by $U_{\reg}$ the set of regular elements of $\mbf G$ contained in $U$.

 In Section \ref{lietype1} we describe 
$\Cal I(\mbf B^G, \mbf U^G, \Cal S)$ in the case where $\mbf G$ is split. In Section \ref{lietype2} we use the theory of Gelfand--Graev characters to 
establish a necessary correspondence between certain characters of $\mbf G^F$ and some characters of $\mbf B^F$.

\subsection{Induced characters of the Borel subgroup}\label{lietype1}

Let $\Phi$ be an arbitrary root system. 
Fix a system $\Pi$ of fundamental roots in $\Phi$, and denote by $\Phi^+$ the corresponding set of positive roots.
Let $\mF$ be a field. If $\Phi$ is irreducible, denote by $U=U(\Phi^+,\mF)$ the unipotent subgroup of the Chevalley group of type $\Phi$ defined over the field $\mF$ (see~\cite{CarterSGLT}, Section 4.4). If $\Phi$ is reducible, we define $U(\Phi^+,\mF)$ as the direct product of the groups $U(\Psi^+,\mF)$ for the irreducible components $\Psi$ of $\Phi$ (counted with multiplicities). For each $\a\in \Phi^+$ let $X_{\a}$ be the corresponding root subgroup of $U$, so that there is an isomorphism $x_{\a}\colon \mF\ra X_{\a}$. 

Recall that the \emph{height} of a positive root $\beta=\sum_{\a\in \Pi} n_{\a} \a$ is defined by $\hght(\b)=\sum_{\a \in \Pi} n_{\a}$. We have a decomposition $U=\prod_{\a\in\Phi^+} X_{\a}$, where the product is taken in an order such that roots of smaller height always precede roots of bigger height. (All subsequent products of root subgroups will be assumed to be taken with respect to an ordering of this form.)

For each $h\in \mN$ let $\Phi^+_h$ be the set of (positive) roots of height $h$ in $\Phi$. Define
$$
U_h = \prod_{\hght(\a)\ge h} X_{\a}
$$
and $\bar{U}_h=U_h/U_{h+1}$. We will use the bar notation for the standard homomorphism $U_h \ra \bar U_h$.

Let $\fr g$ be the complex semisimple Lie algebra of type $\Phi$. Let $\fr g_{\mF}$ be the corresponding Lie algebra over $\mF$ (see~\cite{CarterSGLT}, Section 4.4) and denote by $\fr u = \fr u(\Phi^+,\mF)$ the Lie algebra spanned by the positive root vectors of $\fr g_{\mF}$. 
%(We will not mention $\fr g_{\mF}$ again.) 
Let $\{e_{\a} \mid \a \in \Phi^+\}$ be a Chevalley basis for $\fr u$. For $h\in \mN$ let 
$
 \fr u^h = \sum_{\hght(\a)=h} \mF e_{\a}.
$

If $\a,\b,\a+\b\in \Phi^+$ and $r,s\in \mF$, we have the commutator relation 
\begin{equation}\label{commrel}
[x_{\a} (r), x_{\b} (s)] \equiv  x_{\a+\b} (-N_{\a,\b} rs) \dmod U_{\hght (\a+\b)+1}
\end{equation} 
where $N_{\a,\b}$ is the element of the image of $\mZ$ in $\mF$ such that
$[e_{\a},e_{\b}] = N_{\a,\b} e_{\a+\b}$ (see~\cite{CarterSGLT}, Theorem 5.2.2). Also, if $\a,\b\in \Phi^+$ and $\a+\b\notin\Phi^+$, we have
$[x_{\a}(r),x_{\b}(s)]=1$ and $[e_{\a},e_{\b}]=0$.

Throughout Section~\ref{lietype} we assume that $q$ is a power of our fixed prime $p$.

\begin{prop}\label{redlin}  Let $\Phi^+$ be a set of positive roots of a root system $\Phi$, and let $\Pi$ be the corresponding set of simple roots.
Let $B=TU$ be a finite group which is a semidirect product of a normal $p$-subgroup $U=U(\Phi^+,\mF_q)$ and a $p'$-subgroup $T$.
Assume that the prime $p$ is good for $\Phi$.
% and that $[T,U_{\a}] \subseteq U_{\a} U_{\hght(\a)+1}$ for all $\a\in\Phi^+$. 
Let 
$$
\Cal S= \left\{ Q \le U \; \Big| \; Q\subseteq \prod_{\a\in\Phi^{+}-\{\d\}} U_{\a} \text{ for some } \d \in \Pi  \right\}.
$$ 
If $\chi\in\Irr(B)$ and $U_2 \not\subseteq \ker\chi$ then $\chi\in \Cal I(B,U,\Cal S)$. 
\end{prop}

We need a technical lemma. 
\begin{lem}\label{roots11}
Let $\Phi$ be a root system with a fixed set $\Pi$ of simple roots and corresponding set of positive roots $\Phi^+$. Suppose $\mF$ is a field of characteristic that is either $0$ or a good prime for $\Phi$, and let $\fr u=\fr u(\Phi^+, \mF)$.
% Let $\kappa:\fr u_1 \times \fr u_{h} \ra \fr u_{h+1}$ be the map induced by the Lie bracket. 
Let $h>0$ and suppose $V$ is a proper subspace of $\fr u^{h+1}$. Then there exists $\d$ in $\Pi$ such that, for every $a=\sum_{\a\in \Pi} a_{\a} e_{\a}$  ($a_{\a} \in \mF$) with $a_{\d}\ne 0$, there is $b\in u^{h}$ satisfying
$
[a, b] \notin V.
$
\end{lem}

\begin{proof}[Proof of Proposition~\ref{redlin} (assuming Lemma~\ref{roots11})]
 Let $h$ be the smallest integer such that $U_{h+2} \subseteq \ker\chi$, so $h\ge 1$ by the hypothesis. 
Let $\psi$ be an irreducible constituent of $\Res^G_{U_{h+1}}\chi$ (so $\psi$ is linear, as $\bar U_{h+1}$ is abelian). 
Let $\fr u=\fr u (\Phi^+,\mF_q)$ be the Lie algebra corresponding to $U$ with a Chevalley basis $\{ e_{\a} \}_{\a\in \Phi^+}$, as above. Consider the abelian group isomorphisms $f_i\colon U_i/U_{i+1} \ra \fr u^i$ defined by $f_i(\ol{x_{\a}(r)})=-re_{\a}$ for $i>1$ and 
$f_i(\ol{x_{\a}(r)})=re_{\a}$ for $i=1$, where $r\in \mF_q$. 
By~\eqref{commrel}, the map $\bar U_1 \times \bar U_h \ra \bar U_{h+1}$ induced by the commutator translates under these isomorphisms to the Lie bracket map $\fr u^1 \times \fr u^h \ra \fr u^{h+1}$.  
Let 
$$
V = \{ z\in \fr u^{h+1} \mid f_{h+1}^{-1} (\mF_q z) \subseteq \ker \phi \}.
$$ 
Then $V$ is a proper $\mF_q$-subspace of $\fr u^{h+1}$. Let $\d\in \Pi$ be the root given by Lemma~\ref{roots11} for this subspace.

If $h>1$ we set $L=U_{h}$. If $h=1$, let $\Gamma$ be the set of simple roots that are connected to $\d$ in the Dynkin diagram and set
$L=(\prod_{\a\in \Gamma} X_{\a})U_2$. In either case, $L/U_{h+2}$ is abelian (when $h=1$, this follows from the fact that no two roots in $\Gamma$ are connected to each other in the Dynkin diagram). Hence there exists a linear character $\phi\in\Irr(L)$ that extends $\psi$ and is a constituent of $\Res^B_L \chi$. Consider the inertia group $S=\Stab_{B}(\phi)$.  
We will show that $S\cap U\in \Cal S$, and the proposition will follow because $\chi$ is induced from a character of $S$ by Clifford theory. 

We claim that $S\cap U \subseteq \prod_{\a\in\Phi^{+}-\{\d\}} X_{\a}$. Let 
$a\in U - \prod_{\a\in\Phi^{+}-\{\d\}} X_{\a}$. By the conclusion of Lemma~\ref{roots11}, there exists $f_{h}(b)\in \fr u^{h}$ such that 
$
[f_h (b),f_1 (a)] \notin  V.
$
By definition of $V$, we may replace $f_h (b)$ with a scalar multiple $f_h (c)$ in such a way that $[c,a]\notin \ker\phi$. Moreover, if $h=1$, we may assume that $c\in L$ because $[e_{\a},e_{\d}]=0$ for all $\a\in \Pi-\Gamma$.
We have $\phi(c^a)= \phi(c[c,a])=\phi(c) \phi([c,a]) \ne \phi (c)$, so $a\notin S$. This proves the claim and with it the proposition.
\end{proof}

\begin{proof}[Proof of Lemma~\ref{roots11}]
 The proof is by a case-by-case analysis, which is similar to that of the proof of~\cite{Springer1966}, Theorem 2.6, but requires more detail.
(In fact, the lemma can be derived from \emph{loc.cit.} if we assume in addition that $q\ge |\Pi|$.) 
 The reasons for exclusion of bad primes are essentially the same as in~\cite{Springer1966}.  

Note that the lemma can be stated purely in terms of a bilinear map between vector spaces. During the proof we will call such a map $\kappa$ \emph{adequate} if it satisfies the condition of the lemma. More precisely, if $E$, $Y$ and $Z$ are vector spaces and a basis $\Cal B=\{e_1,\ldots,e_l\}$ of $E$ is fixed, we say that a bilinear map $\kappa\colon E\times Y \ra Z$ is adequate with respect to a proper subspace $V$ of $Z$ if there is $i\in [1,l]$ such that, for each 
$a=\sum_{j=1}^l a_j e_j$ with $a_i \ne 0$, there exists $b\in V$ satisfying $\kappa(a,b) \notin V$. We call $\kappa$ adequate if it is adequate with respect to all proper subspaces of $Z$. We will call a basis vector $e_i\in \Cal B$ \emph{irrelevant} with respect to $\kappa$ if $\kappa(\a,Y)=0$; otherwise $e_i$ will be called \emph{relevant}. Let $\Cal B'$ be the set of the relevant basis vectors, and $E'$ be the span of $\Cal B'$. Obviously, $\kappa$ is adequate if and only if its restriction to $E'\times Y$ is. 

It clearly suffices to prove the lemma for irreducible root systems, so we will assume that $\Phi$ is irreducible. 
By~\cite{CarterSGLT}, Theorem 4.2.1, whenever $\a,\b,\a+\b\in \Phi^+$, we have 
\begin{equation}\label{techlem1}
 [e_{\a}, e_{\b}] = \e_{\a,\b} N'_{\a,\b} e_{\a+\b}
\end{equation}
where $N'_{\a,\b}$ is the image in $\mF_q$ of the largest integer $i$ such that $\b-(i-1)\a\in \Phi$ and $\e_{\a,\b} \in \{ \pm 1\}$. Note that $N'_{\a,\b}\ne 0$ because $p$ is a good prime. If $\Phi$ is simply-laced and $h\ge 2$ then, by~\cite{Springer1966}, Lemma 1.9, 
we may (and do) choose the basis $\{ e_{\g} \}_{\g\in \Phi^+}$ in such a way that 
$\e_{\a,\b}=1$ whenever $\a\in \Phi^+_1$, $\b\in\Phi^+_h$ and $\a+\b\in \Phi^+_{h+1}$. For simply-laced root systems we also have $N'_{\a,\b}=1$. 

First assume $\Phi$ is of type $A_l$, $B_l$, $C_l$ or $G_2$ and consider $h>0$. In each of these cases we will choose
a total order $<$ on $\Pi$ and will consider the induced lexicographic order on $\Phi^{+}$. That is, $\sum_{\a \in \Pi} n_{\a} \a<\sum_{\a\in\Pi} m_{\a} \a$ if and only if there is $\a\in \Pi$ such that $n_{\a}<m_{\a}$ and $n_{\a'}=m_{\a'}$ for all $\a'>\a$ in $\Pi$. We will ensure that this order on $\Phi^+$ satisfies the following property: if $\g\in \Phi^+_{h+1}$ and $\d\in \Pi$ is the smallest simple root such that $\g-\d\in \Phi^+$ then, for each $\a\in \Pi-\{\d\}$, either $\g-\d+\a\notin \Phi^+$ or $\g-\d+\a>\g$. (Note that, if $\g\in \Phi^+$ and $\hght(\g)>1$ then there always exists $\a\in\Pi$ such that $\g-\a\in \Phi^+$; this is true for all root systems.)

Suppose such an order $<$ exists. Let $V$ be a proper subspace of $\fr u^{h+1}$.  Let $\g\in \Phi^+_h$ be the maximal element (with respect to $<$) such that $e_{\g} \notin V$. 
Let $\d\in \Pi$ be the minimal element such that $\b=\g-\d \in \Phi^+$. Then $[e_{\d}, e_{\b}]=N_{\d,\b} e_{\g}\notin V$ because 
$N_{\d,\b}$ is non-zero in $\mF_q$ as $p$ is assumed to be good for $\Phi$. On the other hand, for each $\a\in \Pi-\{\d\}$,  we have $[e_{\a}, e_{\b}]\in V$ by the condition imposed on the order $<$. Thus $\d$ clearly satisfies the requirement of the lemma.

\renewcommand{\arraystretch}{1.1}

\begin{figure}[htb]
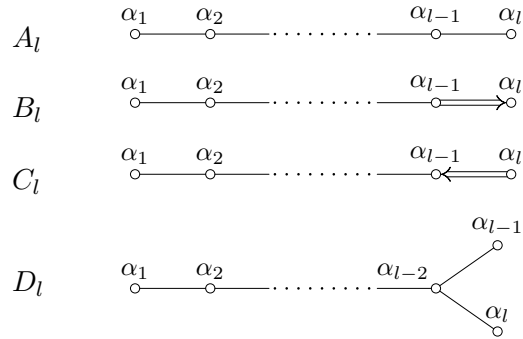

\caption{Labelling of simple roots \qquad\qquad\quad}
\label{labelroots}
\centering
\begin{tabular}{m{1cm}m{7cm}}
 %\hline 
$A_l$ & \includegraphics{al.mps} \\[0.3cm]
$B_l$ & \includegraphics{bl.mps} \\[0.3cm]
%\hline
$C_l$ & \includegraphics{cl.mps} \\[0.3cm]
%\hline
$D_l$ & \includegraphics{dl.mps} 
%\hline
\end{tabular}
\end{figure}

If $\Phi$ is of type $A_l$, $B_l$ or $C_l$ we choose the ordering $\a_1>\cdots>\a_l$, where the simple roots are labelled as in Fig.~\ref{labelroots}. 
Using the well-known descriptions of $\Phi^+$ given in Table~\ref{posroots}, 
it is straightforward to check that the condition is then satisfied. 
If $\Phi$ is of type $G_2$, both orderings of the simple roots have the required property.

\begin{table}[htb]
\caption{The positive roots}
\label{posroots}
$$
\begin{array}{|c|l|}
 \hline 
A_l, \; l\ge 1 & \begin{array}{l} \a_i+\a_{i+1}+ \cdots + \a_j, \; 1\le i \le j \le l \\ \end{array} \\
\hline
B_l, \; l\ge 2 & 
\begin{array}{l}
 \a_i + \a_{i+1} + \cdots + \a_j, \; 1\le i \le j\le l \\
\a_i+ \cdots + \a_{j-1} + 2\a_j + 2\a_{j+1} + \cdots + 2\a_l, \; 1\le i<j \le l \\
\end{array} \\
\hline
C_l, \; l\ge 3 &
\begin{array}{l}
 \a_i + \a_{i+1} + \cdots + \a_j, \; 1\le i \le j\le l \\
\a_i+ \cdots + \a_{j-1} + 2\a_j + 2\a_{j+1} + \cdots + 2\a_{l-1}+ \a_l, \; 1\le i\le j < l \\
\end{array} \\
\hline
D_l, \; l\ge 4 & 
\begin{array}{l}
 \lda_{ij}=\a_i + \a_{i+1} + \cdots + \a_j, \; 1\le i \le j\le l-1 \\
 \mu_{i}=\a_i+\a_{i+1} + \cdots + \a_{l-2} + \a_l, \; 1 \le i \le l-1 \\ 
\nu_{ij}=\a_i+ \cdots + \a_{j-1} + 2\a_j + \cdots + 2\a_{l-2}+ \a_{l-1}+\a_l, \; 1\le i< j\le l-1 \\
\end{array} \\
\hline
\end{array}
$$
\end{table}

The rest of the proof will proceed as follows. We suppose for contradiction that $V$ is a proper subspace of $\fr u^{h+1}$ such that $[\cdot, \cdot]$ is not adequate with respect to $V$. The argument proceeds in steps, and at each step we show that $V$ must contain a certain subset of 
the basis of $\fr u^{h+1}$ given by root vectors. 
Suppose we have proved that $V$ contains the subspace $Z_0=\{ e_{\g} \mid \g \in X_0 \}$ of $\fr u^{h+1}$ for some subset $X_0$ of $\Phi^+_{h+1}$ (initially $X_0 =\varnothing$). In this situation the problem is reduced to proving that the map $\fr u^1 \times \fr u^h \ra \fr u^{h+1} /Z_0$ induced by $[\cdot,\cdot]$ is adequate with respect to $V/Z_0$. We use the bar notation to denote the projection $\fr u^{h+1}\ra \fr u^{h+1}/Z_0=\bar{\fr u}^{h+1}$. 

We will then choose a subset $X_1$ of $\Phi^+_{h+1}$, disjoint from $X_0$, and a subset $S$ of $\Phi^+_h$, and will consider the subspaces 
\begin{equation}
\begin{array}{lcl}\label{YZ}
Y & = &\lan e_{\b} \mid \b \in S \ran \;\; \text{of } Y \qquad\qquad  \text{and} \\
Z & = &\lan \bar e_{\g} \mid \g \in X_1 \ran \;\; \text{of } \bar u^{h+1}.
\end{array}
\end{equation}
The choice will be made in such a way that (in particular) $|S|=|X_1|$ and $\ol{[\fr u^1, Y]} = Z$. 
Let $E$ be the span of basis vectors of $V$ that are relevant for the map $\fr u^1\times Y\ra Z$ induced by $[\cdot,\cdot]$, and let 
$\kappa\colon E\times Y \ra Z$ be the corresponding map. 
In each case it will turn out that, after reordering and relabelling of the given bases of $E$, $Y$ and $Z$ and after changing signs of some basis elements, the map $\kappa$ is represented by one of the matrices in the list below. We will check that each bilinear map in the list is adequate (if $\Char \mF$ is good for $\Phi$). Now if
$\bar V$ does not contain $Z$, then an element $\d\in \Pi \cap E$ witnessing the fact that $\kappa$ is adequate with respect to $\bar V \cap Z$ also witnesses the fact that the map $[\cdot,\cdot]\colon \fr u^1 \times \fr u^{h} \ra \fr u^{h+1}$ is adequate with respect to $V$. Hence $V$ must contain $X_1$. We now replace $X_0$ with $X_0\cup X_1$ and continue in the same manner until $X_0$ becomes equal to $\Phi^{+}_{h+1}$, which is obviously a contradiction. In fact, the proof above for types $A_l$, $B_l$, $C_l$ and $G_2$ can be framed in the same way, with $X_0$ and $S$ chosen to be appropriate singleton sets at each step.

Now we give the list of maps $\kappa\colon E\times Y \ra Z$ discussed in the previous paragraph. Each map will be given by a square 
$\dim(Y)\times \dim(Z)$-matrix with coefficients in $E^*$. If $\{ e_1,\ldots,e_m \}$ is the given basis of $E$ then the coefficients will be expressed as linear combinations of the dual basis $\{ e^1,\ldots, e^m \}$. In each case we will denote the given bases of $Y$ and $Z$ by $\{ y_1,\ldots, y_n\}$ and 
$\{z_1,\ldots,z_n\}$ respectively (so $n=\dim(Y)=\dim(Z)$). 
The correspondence between matrices of this form and bilinear maps $\kappa$ is given by the standard natural isomorphism $E^*\otimes \Hom(Y,Z) \ra \Hom(E\otimes Y,Z)$, $e^* \otimes f \mapsto (e\otimes y \mapsto e^* (e)f(y) )$. Whenever we consider an element $a\in E$, we will use the notation
$a=\sum_{i=1}^m a_i e_i$.

\ncase
\textbf{Case 1:} $\dim(Y')=\dim(Z')=1$ and the matrix of $\kappa$ is $(Ne^1)$, where $N$ is one of the structure constants $\pm N'_{\a,\b}$ (cf.~\eqref{techlem1}). Since $p$ is good, $N'_{\a,\b}\ne 0$ in $\mF_q$, and it follows that $\kappa$ is adequate. We remark that, in some sense this is the most common case. It occurs whenever we can find $\b\in \Phi^+_h$ with the property that there is exactly one $\g\notin X_0$ such that there exists $\a \in \Pi$ satisfying $\a+\b=\g$.

\ncase
\textbf{Case 2:} the matrix of $\kappa$ is 
$$
A=
\begin{pmatrix}
e^2 & e^1 & 0 \\
e^3 & 0 & e^1 \\
0 & e^3 & e^2 \\
\end{pmatrix}.
$$
We claim that $\kappa$ is adequate as long as $\Char \mF \ne 2$.
Suppose for contradiction that $\kappa$ is not adequate with respect to some proper subspace $W$ of $Z$. First assume that $z_1,z_2\notin W$. We claim that $e_1$ witnesses the fact that $\kappa$ is adequate for $W$. 
 Let 
$a\in E$ be an element with $a_1\ne 0$. 
We have 
$
\det(A(a))=-2a_1 a_2 a_3.
$
So if $a_2$ and $a_3$ are non-zero then $z_1 \in \kappa(a,Y)$. If $a_2=0$ then 
$\kappa(a,y_3)=a_1 z_2 \notin W$. If $a_3=0$ then $\kappa(a,y_2)=a_1 z_1 \notin W$. Thus $\kappa$ is adequate with respect to $W$ in this situation. 

Hence either $z_1 \in W$ or $z_2 \in W$. Without loss of generality, we may assume that the former holds. 
We see that $z_2 \in W$ by considering the first column of the matrix: otherwise $\kappa(a, y_1) \equiv a_2 z_2 \pmod W$, so $e_2$ witnesses 
the fact that $\kappa$ is adequate for $W$. Similarly, we deduce that $z_3\in W$ by considering the second column of $A$. Thus $W=Z$, which is a contradiction. 

\ncase
\textbf{Case 3:} the matrix of $\kappa$ is 
$$
A=
\begin{pmatrix}
e^2 & e^1 & 0 \\   
e^3 & 0 & e^1 \\
0 & 2e^3 & e^2 \\
\end{pmatrix}.
$$
We claim that $\kappa$ is adequate provided $\Char \mF \ne 2,3$. Indeed, if $a\in E$ and $a_i\ne 0$ for $i=1,2,3$ then $\det A(a)=-3a_1 a_2 a_3 \ne 0$. We omit the rest of the proof as it is the same as in the previous case. 

\ncase
\textbf{Case 4:} the matrix of $\kappa$ is
$$
A=
\begin{pmatrix}
e^5 & e^3 & 0 & 0 & 0 \\
e^2 & 0 & e^3 & 0 & 0 \\
0 & e^4 & 0 & e^1 & 0 \\
0 & e^2 & e^5 & 0 & e^1 \\
0 & 0 & 0 & e^2 & e^4 \\
\end{pmatrix}
$$
We claim that $\kappa$ is adequate provided $\Char \mF \ne 3$. Suppose $\kappa$ is not adequate with respect to some proper subspace $W<Z$. 
First suppose that, for all $i=1,\ldots,5$, we have $z_i\notin W$. Then there exists $a\in E$ with $a_1\ne 0$ such that $\kappa(a,Y)\subseteq W$. If $a_i\ne 0$ for $i=2,\ldots,5$ then $\kappa(a,Y) =Z \nsubseteq W$ because $\det A(a)=3a_1 a_2 a_3 a_4 a_5 \ne 0$ (by our assumption on $\Char \mF$). On the other hand, we have the following sequence of deductions (where each line may use conclusions of preceding lines):
\begin{align*}
a_2 \ne 0\colon \text{ otherwise } & \kappa(a,y_4)=a_1 z_3\notin W; \\
a_4 \ne 0\colon \text{ otherwise } & \kappa(a,y_5)=a_1 z_4\notin W; \\
a_5 \ne 0\colon \text{ otherwise } & \kappa(a,y_1)=a_5 z_1 \notin W; \\
a_3 \ne 0\colon \text{ otherwise } & \kappa(a,y_3)=a_5 z_4 \notin W.
\end{align*}
This is a contradiction as we have already established that $a_i=0$ for some $i\in \{2,3,4,5\}$. 

Thus at least one of the basis elements $z_i$, $i=1,\ldots,5$, must be in $W$. 
Note that if, for some $m$, the $m$th column of $A$ has nonzero entries only in rows $i$ and $j$, and these entries are $e^s$ and $e^t$ respectively, then $z_i\in W$ if and only if $z_j\in W$. Indeed, suppose $z_i \in W$. Then $\kappa(a,y_m)\equiv a_t z_j \pmod W$ for every $a\in E$, so if $z_j\notin W$ then $e_t$ witnesses the adequacy of $\kappa$ for $W$; thus $z_j\in W$. 

We have
$$
z_1 \in W \stackrel{1}{\Llra} z_2 \in W \stackrel{3}{\Llra} z_4 \in W \stackrel{5}{\Llra} z_5 \in W \stackrel{4}{\Llra} z_3 \in W,
$$
where the number above each arrow indicates a column of $A$ one may consider to establish the corresponding equivalence. Since at least one basis element $z_i$ lies in $W$, all of them must be in $W$, which is a contradiction.

\ncase
\textbf{Case 5:} the matrix of $\kappa$ is
$$
A=
\begin{pmatrix}
e^4 & 0 & e^1 & 0 & 0 & 0 & 0 \\
e^6 & e^2 & 0 & e^1 & 0 & 0 & 0 \\
0 & e^7 & 0 & 0 & 0 & e^1 & 0 \\
0 & 0 & e^6 & e^4 & 0 & 0 & 0 \\
0 & 0 & 0 & e^7 & e^3 & e^2 & 0 \\
0 & 0 & 0 & 0 & e^5 & 0 & e^2 \\
0 & 0 & 0 & 0 & 0 & e^5 & e^3 
\end{pmatrix}.
$$
We will show that $\kappa$ is adequate as long as $\Char \mF \ne 2,5$. Suppose $\kappa$ is not adequate for a subspace $W<Z$. Assume first that $W$ contains none of the basis vectors $z_i$, $i=1,\ldots,7$. Then there is $a\in E$ with $a_2\ne 0$ such that $\kappa(a,Y)\subseteq W$. Since 
$\det A(a)= 5a_1 a_2 a_3 a_4 a_5 a_6 a_7$, there is $i\in [1,7]$ such that $a_i=0$. However, 
\begin{align*}
a_3 \ne 0\colon \text{ otherwise } & \kappa(a,y_7)=a_2 z_6 \notin W; \\
a_5\ne 0\colon \text{ otherwise } & \kappa(a,y_5)=a_3 z_5 \notin W;\\
a_7\ne 0\colon \text{ otherwise } & \kappa(a,y_2)=a_2 z_2 \notin W. 
\end{align*}
Suppose $a_1=0$. Consider the subspaces $Z'=\lan z_5,z_6,z_7 \ran$ and $Y'=\lan y_5, y_6, y_7 \ran$. Since $a_1=0$, we have $\kappa(E,Y')\subseteq Z'$,
and the matrix of the restriction of $\kappa$ to $E\times Y'$ is given by the intersection of the last 3 rows and the last 3 columns of $A$. Since $a_2$, $a_3$ and $a_5$ are all non-zero, by the proof in Case 2 we have $\kappa(a,Y')=Z'$, so $W$ contains $z_5$, $z_6$ and $z_7$, contradicting our assumption. Hence $a_1\ne 0$, and we further deduce that
\begin{align*}
a_6\ne 0\colon \text{ otherwise } & \kappa(a,y_3)=a_1 z_1 \notin W; \\
a_4 \ne 0\colon \text{ otherwise } & \kappa(a,y_1)=a_6 z_2 \notin W.
\end{align*}
So $a_i \ne 0$ for $i\in [1,7]$, which we have shown to be impossible.

It remains to consider the case when $W$ contains $z_i$ for some $i\in [1,7]$. Using the convention of the previous case, we have equivalences
$$
z_3 \in W \stackrel{2}{\Llra} z_2 \in W \stackrel{1}{\Llra} z_1 \in W \stackrel{3}{\Llra} z_4 \in W.
$$
Also,
$$
z_5 \in W \stackrel{5}{\Llra} z_6 \in W \stackrel{7}{\Llra} z_7 \in W.
$$
Moreover, if both $z_2$ and $z_4$ are in $W$ then $z_5$ is in $W$ (consider column $4$ of $A$). And if $z_5,z_7\in W$ then $z_3\in W$ (consider column $7$). Combining all these implications, we see that, as at least one of the elements $z_i$, $i\in [1,7]$, is in $W$, all of them must belong to $W$, which is a contradiction.

\medskip
Now we apply the method described above for the remaining types $D_l$, $F_4$, $E_6$, $E_7$ and $E_8$.

\ncase
\textbf{Type $D_l$, $l\ge 4$.} If $h=1$ then the result is easy: e.g.\ the argument used above for types $A_l$, $B_l$, $C_l$ and $G_2$ works if we consider the ordering
$\a_{l-2}< \a_{l-1} <\a_l <\a_{l-3} < \cdots <\a_1$ of simple roots. So we may assume that $h>1$. 

Since $\Phi$ is simply-laced, by the discussion above (see~\eqref{techlem1}), we may assume that $N'_{\a,\b}=1=\e_{\a,\b}$ for each $\a\in \Pi$ and $\b\in\Phi^+_{h}$.
We use the case-by-case argument described above, labelling the positive roots as in Table~\ref{posroots}.
Suppose we have $\lda_{1,h}, \lda_{2,h+1}, \ldots, \lda_{i,h+1-i}\in X_0$ for some $i$, so $V$ must contain all these elements (initially $i=0$). Suppose $h+1-i<l-2$. Then we consider $X_1=\{ \lda_{i+1,h+2-i} \} \subseteq \Phi^{+}_{h+1}$ and $S=\{ \lda_{i+1,h+1-i} \}$. Let $Y$ and $Z$ be the $1$-dimensional subspaces defined by~\eqref{YZ}. Then $\ol{[\fr u^1, Y]}\subseteq Z$. Indeed, the only simple roots $\a$ satisfying $\a+\lda_{i+1,h+1-i}\in \Phi^+$ are $\a_i$ (if $i>0$) and $\a_{h+2-i}$, and if $i>0$ then $\a_i+\lda_{i+1,h+1-i}=\lda_{i,h+1-i}\in X_0$. The set of relevant basis vectors is therefore $E=\{ e_{\a_{h+2-i}} \}$. The map $\kappa\colon E \times Y \ra Z$ is as in Case 1 and hence is adequate. Therefore, we may add $\lda_{i,h+1-i}$ to $X_0$. Arguing by induction, we may now assume that 
$$
X_0=\{\lda_{i,h+1-i} \mid i\ge 1 \text{ and } h+1-i\le l-2 \}.
$$

Now let $m=l-h-1$ (so $m<l-2$ because $h>1$). Set $X_1=\{ \lda_{m,l-1}, \, \mu_m, \, \nu_{m+1,l-1} \} \subseteq \Phi^{+}_{h+1}$ and 
$S=\{ \lda_{m,l-2}, \, \lda_{m+1, l-1}, \, \mu_{m+1} \}\subseteq \Phi^+_h$, and let $Y$ and $Z$ be defined by~\eqref{YZ} as before. It is easy to see that then the set of the relevant basis elements of $\fr u^1$ is $E=\{ e_{\a_m}, e_{\a_{l-1}}, e_{\a_l} \}$ and that the resulting map 
$\kappa\colon E\times Y\ra Z$ is as in Case 2 (after reordering our bases). Since $\Char \mF\ne 2$, it follows that $\kappa$ is adequate, and we may replace $X_0$ with $X_0 \cup X_1$. 

In particular, now $\nu_{m+1,l-1} \in X_0$. Suppose that for some $j<(l-m-2)/2$  we have 
\[
\nu_{m+1,l-1}, \; \nu_{m+2,l-2}, \; \ldots, \; \nu_{m+j,l-j}\in X_0 
\]
 (initially $j=1$). Setting $X_1=\{ \nu_{m+j+1,l-j-1} \}$ and $S=\{ \nu_{m+j+1, l-j } \}$ and using Case 1, we deduce that $\nu_{m+j+1,l-j-1} \in V$ by the same method as before. We can thus add $\nu_{m+j+1,l-j-1}$ to $X$. Arguing by induction, we infer that all elements of $\Phi^+_{h+1}$ of the form $\nu_{i,t}$ must belong to $V$. We have already shown that all the other elements of $\Phi^{+}_{h+1}$ are in $V$, so we have the desired contradiction.

\ncase
\textbf{Types $E_6$, $E_7$, $E_8$, $F_4$.} 
In these cases, for each possible height $h$, one can list the pairs of roots $\a\in\Pi$ and $\b\in \Phi^h$ satisfying $\a+\b\in \Phi^{h+1}$ using a simple program in GAP~\cite{GAP4}. Note that, as explained above, the signs $\e_{\a,\b}$ may all be taken to be $1$ if $\Phi$ is simply-laced, whereas the signs for type $F_4$ may be found in~\cite{BurgoyneWilliamson1971}. This information suffices to construct the matrix describing the Lie bracket map $\fr u^1 \times \fr u^h \ra \fr u^{h+1}$ in each case (in practice, it is convenient to represent these data by a bipartite graph with vertex set $\Phi^+_h \cup \Phi^+_{h+1}$ and edges corresponding to nonzero entries of the matrix).

It is then routine to check, using the method we have described, that adequacy of maps in Cases 1 and 3 implies the result for type $F_4$ and that adequacy of maps in Cases 1--5 implies the result for type $E_8$. We have shown that maps in Case 3 are adequate if $\Char \mF \ne 2,3$ and the map in Case 5 is adequate if $\Char \mF\ne 2,5$. Since $\Char \mF$ is assumed to be $0$ or a good prime for $\Phi$, the result follows for types $F_4$ and $E_8$. 

In fact, for type $F_4$ the map of Case 3 occurs only when $h=3$. For type $E_8$, the map of Case $2$ occurs for $h=2,4,8,14$, the map of Case 4 occurs for $h=3,9$ and the map of Case 5 occurs for $h=5$. The heights of positive roots of $E_8$ vary from $1$ to $29$, and for the remaining values of $h\in [1,28]$ maps of Case 1 suffice. (These conclusions are in agreement with the statement of~\cite{Springer1966}, Theorem 2.6.) Now the root systems of types $E_6$ and $E_7$ can be embedded into the system of type $E_8$, and we deduce the result for $E_6$ and $E_7$ when $h\ne 5$ (as we have assumed that $\Char \mF\ne 2,3$ if $\Phi$ is of type $E_6$ or $E_7$). Using the same method as before, it is easy to check the lemma in the case of $E_7$, $h=5$ (in fact, maps of Case 1 suffice), and the result for type $E_6$ follows.
\end{proof}

\begin{lem}\label{dirprod}
Let $B=TP$ be a semidirect product where $P=P_1 \times \cdots \times P_l$ is an abelian normal $p$-subgroup of $B$ and $T$ is a $p'$-subgroup of $B$. Assume that $\lsa{t}{P}_i=P_i$ for all $t\in T$ and each $i$. Let 
$$
\Cal S= \{ Q\le P \mid Q \le P_1 \times \cdots \times P_{j-1} \times P_{j+1} \times \cdots \times P_{l} \; \text{ for some } j \}.
$$
Let $P_{\reg} = \{u_1\cdots u_l \in P \mid u_i\in P_i-\mbf 1 \text{ for each } i \}$. 
Let $\chi\in \Cal C(B)$ and suppose that $\chi(g)=0$ for all $g\in B$ such that $g_p\in P_{\reg}$. Then $\chi\in \Cal I(B,P,\Cal S)$.
\end{lem}

\begin{proof} We argue by induction on $l$. If $l=1$, we have $\chi(g)=0$ whenever $g\in B$ and $g_p\ne 1$, whence $\chi\in \Cal P(B)=\Cal I(B,P,\Cal S)$ (see e.g.\ Section~\ref{Eaton}).
Let $l\ge 2$ and consider $Q=P_2 \times \cdots \times P_l$. Define
\[
Q_{\reg}  =  \{u_2 \cdots u_l \in Q \mid u_i \in P_i - \mbf 1 \text{ for each } i \}.
\]
Let $\Omega$ be a set of representatives of the orbits of the action of $B$ on $\Irr(P)$. 
We may (and do) replace $\chi$ with
$$
\chi - \Ind_{TQ}^B \Res_{TQ}^B (\pi_{1_{P_1}}(\chi)),  
$$
so that 
$
\pi_{1_{P_1}} (\chi)=0.
$ 

Let $\phi\in \Irr(P_1)$ and let $X=\Stab_T (\phi)$. Let
\begin{equation}\label{thdef}
\th = \pi_{\phi} \Res^B_{XP} \chi.
\end{equation}
By Clifford theory, we have
\begin{equation}\label{chiClif}
\Ind_{XP}^B \th = \pi_{\phi} \chi.
\end{equation}
 Our aim is to apply the inductive hypothesis to $\Res^{XP}_{XQ} \th$ (note that $X$ normalises $Q$ by the hypothesis). Let $g=sv\in XQ$ where 
$v=g_p\in Q_{\reg}$ and $s=g_{p'}$. We claim that $\th(g)=0$. 

Let $\xi=\Res^B_{\lan g \ran P_1} \chi$. Using~\eqref{thdef} and the fact that $XP$ stabilises $\phi$, we obtain
$$
\pi_{\phi} \xi = \pi_{\phi}\Res^B_{\lan g \ran P_1} \chi = \Res^{XP}_{\lan g \ran P_1} \pi_{\phi} \Res^B_{XP} \chi= \Res^{XP}_{\lan g \ran P_1} \th.
$$
In particular, $\th(g)=(\pi_{\phi} \xi) (g)$. Let 
$$
\Om = \{ \psi \in \Irr(P_1) \mid \ls{s}{\psi} = \psi \} 
$$ 
and $C=C_{P_1} (s)$.
Let $\xi'=\sum_{\psi\in \Om} \pi_{\psi} (\xi)$. We may write 
$$
\Res^{\lan g \ran P_1}_{\lan g \ran C} \xi' = \sum_{\lda\in \Irr(C)} \a_{\lda} \times \lda,
$$
where the virtual characters $\a_{\lda}$ of $\lan g \ran$ are determined uniquely. By the Glauberman correspondence (\cite{IsaacsBook}, Theorem 13.1), there is a bijection between $\Om$ and $\Irr(C)$ given by restriction. Hence, for all $\psi\in\Om$, we have 
\begin{equation}\label{xilda}
\Res^{\lan g \ran P_1}_{\lan g \ran C} \pi_{\psi} \xi= \a_{\lda}\times \lda
\end{equation}
where $\lda=\Res^{P_1}_C \psi$. 
Since $\pi_{1_{P_1}}(\chi)=0$, we have $\a_{1_C} =0$. For any $u\in C$, if $\psi \in \Irr(P_1) - \Om$ and $\eta\in \Irr(\lan g \ran P_1 \di \psi)$ then $\eta(gu)=0$ by Clifford theory (because 
$gu\notin \Stab_{\lan g \ran P_1} (\psi)$), so $\xi(gu)=\xi'(gu)$. Thus, for all $u\in C- \mbf 1$,
$$
0=\chi(gu)=\xi(gu)=\xi' (gu) = \sum_{\lda\in \Irr(C)- \{1_C \}} \a_{\lda}(g) \lda(u).
$$  
So the class function $\g= \sum_{\lda \in \Irr(C)} \a_{\lda}(g) \lda$, defined on $C$, is zero on $C-\mbf 1$ and satisfies $\lan \g,1_C \ran=0$. It follows that $\g=0$, whence $\a_{\lda}(g)=0$ for all $\lda\in \Irr(C)$. By~\eqref{xilda} we have $\th(g)=(\pi_{\phi} \xi) (g)=0$, as claimed.  

Hence we may apply the inductive hypothesis to the group $XQ$ and the set 
$$
\Cal S'= \{ Y\le Q \mid Y \le P_2 \times \cdots \times P_{j-1} \times P_{j+1} \times \cdots \times P_{l} \; \text{ for some } j\in [2,l] \}
$$
in place of $\Cal S$, deducing that $\Res^{XP}_{XQ} \th\in \Cal I(XQ, Q, \Cal S')$. It follows from the definition of $X$ that the group $XP/\ker\phi$ decomposes as 
a direct product of $XQ$ and $P_1/\ker\phi$. Thus $\th$ is the inflation of $(\Res^B_{XQ} \th)\times \tilde\phi \in \Cal C(XP/\ker\phi)$, where $\tilde\phi$ is the deflation of $\phi$ to $P_1/\ker\phi$. Then it is clear that $\th\in \Cal I(B,P,\Cal S)$, and the result follows from~\eqref{chiClif}
 because $\phi\in \Irr(P_1)$ may be chosen arbitrarily. 
\end{proof}

Recall the notation introduced in the beginning of Section~\ref{lietype}. Denote by $W$ the Weyl group $N_{\mbf G} (\mbf T)/\mbf T$ of $\mbf G$.
 For the remainder of Section~\ref{lietype} we adopt the convention that $G=\mbf G^F$, $B=\mbf B^F$, $U=\mbf U^F$ and so on. Let $\Phi$ be the root system of $\mbf G$, let $\Phi^+$ and $\Phi^-$ be the sets of positive and negative roots with respect to $\mbf B$, and denote by $\Pi$ the set of simple roots in $\Phi^+$. 
Denote by $\mbf X_{\a}$ the root subgroup of $\mbf U$ corresponding to $\a \in \Phi^+$. 
As before, let $\mbf U_2= \prod_{\hght(\a) \ge 2} \mbf X_{\a}$ and $U_2=\mbf U_2^F$. We use the bar notation for the projection $\mbf U\ra \mbf U/\mbf U_2=\bar{\mbf U}$.
Due to the commutator relations~\eqref{commrel}, $\bar{\mbf U}$ is the direct product of the root spaces $\bar{\mbf X}_{\a}$ for $\a \in \Pi$.
 Recall that an element $u\in \mbf U$ is \emph{regular} in $\mbf G$ if and only if the 
$\bar{\mbf X}_{\a}$-component of $\bar u$ is not equal to $1$ for each $\a\in\Pi$ (see~\cite{DigneMichelBook}, Proposition 14.14).
Denote by $U_{\reg}$ the set of regular elements of $\mbf G$ contained in $U$. An element $u\in \mbf G$ is regular if and only if $u$ is $G$-conjugate to an element of $U_{\reg}$.

\begin{prop}\label{IforB} Suppose $\mbf G$ is split and $p$ is good for $\mbf G$. Then $\Cal I(B,U,\Cal S (G,U,B))$ consists precisely of those virtual characters $\phi\in \Cal C (B)$ that vanish on the elements $g\in G$ such that $g_p \in U_{\reg}$.
\end{prop}

\begin{proof} It is well known that $U$ may be identified with $U(\Phi^+,\mF_q)$.
Let $n\in N_{\mbf G} ({\mbf T})^F$, and let $w$ be the image of $n$ in the Weyl group $W$. Then $\ls{n}{U} \cap U$ is the product of the root subgroups $X_{\a}$ for  
$\a\in \Phi^+$ such that $w^{-1}(\a)\in \Phi^{+}$. Since only the trivial element of $W$ stabilises $\Phi^{+}$, there is $\d \in \Pi$ such that $w^{-1}(\d) \in \Phi^{-}$, so $\ls{n}{U} \cap U \subseteq \prod_{\a \in \Phi^{+} - \{\d \} } X_{\a}$. On the other hand, if $w$ is the reflection corresponding to a simple root $\d$ then $\d$ is the only positive root sent to $\Phi^-$ by $w$, so $\ls{n}{U}\cap U = \prod_{\a \in \Phi^{+} - \{\d \} } X_{\a}$ in this case. Since every double $B$-$B$-coset in $G$ contains an element of $N_{\mbf G} (\mbf T)^F$, it follows that
$$
\Cal S \coleq \Cal S(G,U,B) = \left\{ Q \le U \; \Big| \; Q\subseteq \prod_{\a\in\Phi^{+}-\{\d\}} U_{\a} \text{ for some } \d \in \Pi \right\}.
$$
Since $U_{\reg}$ is disjoint from each element of $\Cal S$ (and $\Cal S$ is closed under $B$-conjugacy), for every $\chi\in \Cal I(B,U,\Cal S)$ we have $\chi(g)=0$ whenever $g_p \in U_{\reg}$. Conversely, suppose that $\chi\in \Cal C(B)$ and $\chi(g)=0$ for all $g\in B$ with $g_p\in U_{\reg}$. By Proposition~\ref{redlin}, we have  $\chi-\pi_{1_{U_2}} (\chi)\in \Cal I(B,U,\Cal S)$, so we may assume that $\chi=\pi_{1_{U_2}} (\chi)$. However, $B/U_2$ decomposes as a semidirect product of $T$ and $U/U_2$, and $U/U_2=\prod_{\a \in \Pi} \bar X_{\a}$ is a direct product. Therefore, the result follows from
Lemma~\ref{dirprod}.
\end{proof}

\subsection{Semisimple characters}\label{lietype2}

%\begin{thm}[(?)]\label{LietypeF} Assume $p$ is good for $\mbf G$. Suppose $H^1 (F,\mbf Z)$ is cyclic. 
%Then there exists a signed bijection $F: \pm \Irr_{p'} (G) \ra \Irr_{p'} (B)$ such that 
%$F(\chi)(tu) = \chi(tu)$ for all $u\in U_{\reg}$, $t\in Z (G)$ and $\chi\in \Irr_{p'}(G)$.
%\end{thm}

We state certain facts concerning regular characters of $U$ and Gelfand--Graev characters of $G$, 
which may be found in~\cite{DigneMichelBook}, Chapter 14.
Recall that $\tau$ is the permutation of $\Pi$ induced by $F$, and let $\Pi/\tau$ be the set of its orbits. For each $\Om\in \Pi/\tau$ let 
$\bar{\mbf X}_{\Om}=\prod_{\d \in \Om} \bar{\mbf X}_{\d}$ and $\bar X_{\Om} = \bar{\mbf X}_{\Om}^F$. Then $\bar X_{\Om}$ is isomorphic to the additive group of the finite field $\mF_{q^{|\Om|}}$ (see e.g.\ the paragraph preceding Definition 14.27 in~\cite{DigneMichelBook}). Moreover, $U/U_2 = \prod_{\Om\in \Pi/\tau} \bar X_{\Om}$ is a direct product decomposition. 

If $Y$ is a subgroup of $B$ containing $U_2$ and $\zeta \in \Cal C(Y)$ is such that the kernel of every irreducible constituent of $\zeta$ contains $U_2$, we write $\tilde{\zeta}=\Def^Y_{Y/U_2} \zeta$. 
A linear character $\psi\in \Irr(U)$ is called \emph{regular} if $U_2 \subseteq \ker \psi$ and $\Res^{\bar U}_{\bar X_\Om}\tilde{\psi} \ne 1_{U_\Om}$ for every 
$\Om\in \Pi/\tau$.
The $T$-orbits of regular characters of $U$ are parametrised by the elements of 
$H^1 (F, Z(\mbf G) )$, the parametrisation being uniquely determined by the choice of $\psi_1$ (which can be taken to be any regular character of $U$). For each $z\in H^1 (F, Z(\mbf G))$ let $\psi_z$ be a regular character of $U$ contained in the orbit corresponding to $z$. More precisely, we may take $\psi_z=\ls{t}{\psi_1}$ where $t\in \Cal L^{-1}(z \Cal L(Z(\mbf G)))$. 

\begin{remark}
 If we assume $Z(\mbf G)$ to be connected then $H^1 (F,Z(\mbf G))$ is trivial and the proofs below become much simpler.
\end{remark}

The Gelfand--Graev character attached to 
$z\in H^1 (F, Z (\mbf G))$ is defined by 
$\Gamma_z = \Ind_U^G \psi_z$. Let $D_{\mbf G}$ be the isometry
$\Cal C(G)\ra \Cal C(G)$ given by the Alvis--Curtis--Kawanaka duality (see~\cite{DigneMichelBook}, Chapter 8). 
Let $\Xi_z = D_{\mbf G} \Gamma_z$ and $\xi_z = \Res^G_B \Xi_z$.

An explicit description of $X_{\Om}$ as a subset of $\bar{\mbf X}_{\Om}$ (see e.g.\ \cite{CarterSGLT}, Proposition 13.6.3) shows that  an element $u\in U$ is regular if and only if, for each $\Om\in\Pi/\tau$, 
the image of $u$ under the projection $U\ra \bar X_{\Om}$ is different from $1$.   
The virtual characters $\Xi_z$ vanish outside regular unipotent elements of $G$ 
(see~\cite{DigneMichelBook}, Proposition 14.33). Since each regular unipotent element of $G$ 
is contained in a unique $G$-conjugate of $B$ and $N_G (B)=B$, we have $\Ind_B^G \xi_z = \Xi_z$ for all $z\in H^1 (F, Z(\mbf G))$.

Thus, by the Frobenius reciprocity,
\begin{equation}\label{Frobrec}
\lan \Xi_z, \Xi_{z'} \ran = \lan \xi_z, \xi_z' \ran \quad \text{ for all } z,z' \in H^1 (F, Z(\mbf G)). 
\end{equation}
An irreducible character of $G$ is called \emph{semisimple} if it is a constituent of some $\Xi_z$. 
We denote by $\Irr_{s} (G)$ the set of semisimple characters of $G$ and, similarly, we write
$$
\Irr_{s} (B) =\{ \phi\in \Irr(B) \mid \lan \xi_z, \phi \ran \ne 0 \text{ for some } z\in H^1 (F, Z(\mbf G)) \}.
$$

Let $J\subseteq \Pi/\tau$. Denote by $\mbf P_J$ the standard parabolic subgroup containing $\mbf B$ associated with the union of $J$. Let $\mbf L_J$ be the (unique) Levi subgroup of $\mbf P_J$ containing $\mbf T$, and denote by $\mbf V_J$ the unipotent radical of $\mbf P_J$. 
The inclusion $Z(\mbf G) \hra Z(\mbf L_J)$ induces a surjective map $h_J \colon H^1 (F, Z(\mbf G)) \ra H^1 (F, Z(\mbf L_J))$  (\cite{DigneMichelBook}, Lemma 14.31). 
We set $\psi^J_1=\Res^{U}_{L_J \cap U} \psi_1$. This gives a parametrisation of $T$-orbits of regular unipotent characters of $L_J \cap U$ by elements of $H^1 (F, Z(\mbf L_J))$ as above. Write $\psi^J_y$ for a representative of the orbit corresponding to $y$.

By the proof of Proposition 14.33 in~\cite{DigneMichelBook}, we have
$
\Xi_z = \Ind^G_U ( \xi'_z)
$
where
\[
\xi'_z = \sum_{J \subseteq \Pi/\tau} (-1)^{|J|} \Inf_{L_J\cap U}^U \psi^J_{h_L (z)}.
\]
It is also shown in the quoted proof that $\xi'_z$ vanishes outside $U_{\reg}$, and we deduce that $\Xi_z (u)=\xi'_z(u)$ for each $u\in U_{\reg}$, whence
\[
\xi_z = \Ind_U^B (\xi'_z).
\]
Thus
 \begin{equation}\label{xi_z}
 \xi_z = \Ind_U^B \left(\sum_{J\subseteq \Pi/\tau} (-1)^{|J|} \Inf_{L_J\cap U}^U \psi^J_{h_J (z)} \right).
 \end{equation}

\begin{remark}\label{Bonrem}
The proof of Proposition 14.32 in~\cite{DigneMichelBook}, on which the quoted results rely, is incomplete. C{\'e}dric Bonnaf{\'e}
has provided a correct proof to the author (see~\cite{DigneMichelErrata}).
\end{remark}

As is well known, we have a semidirect product decomposition $U = V_J \rtimes (L_J \cap U)$, which induces a direct product decomposition 
$\bar U=\bar V_J \times\ol{L_J\cap U}$. Moreover, $\ol{L_J\cap U} = \prod_{\Om \in J} \bar X_{\Om}$ and 
 $\bar V_J=\prod_{\Om \in (\Pi/\tau)-J} \bar X_{\Om}$, where both products are direct. Let $z\in H^1(F,Z(\mbf G))$. 
Since $U_2\cap L_J \subseteq \ker\psi^{J}_{h_J (z)}$, we may view~\eqref{xi_z} as an equality of virtual characters of $B/U_2$. 
Hence we may rewrite~\eqref{xi_z} as 
 \begin{equation}\label{barxi_z}
 \tilde{\xi}_z = \Ind_{\bar U}^{\bar B} \left(  \sum_{J \subseteq \Pi/\tau} (-1)^{|J|} \th^J_{h_J (z)} \right)
 \end{equation}
 where $\th^J_y = \Inf_{\ol{L_J\cap U}}^{\bar U} \tilde{\psi}^J_y$.

The significance of semisimple characters for the McKay conjecture is clarified by the following lemma, which for the most part is a summary of known results.

\begin{lem}\label{pdeqss} 
We have $\Irr_{p'}(G)=\Irr_{s}(G)$ and $\Irr_{p'} (B)=\Irr_{s} (B)$ unless one of the following holds:
\begin{enumerate}[(i)]
 \item $q=2$ and $\Phi$ has a component of type $B_l$, $C_l$, $F_4$ or $G_2$;
\item $q=3$ and $\Phi$ has a component of type $G_2$.
\end{enumerate}
Moreover, we always have $\Irr_s (B)=\Irr(B\di 1_{U_2})$.
\end{lem}

\begin{proof}
The statement for $G$ is~\cite{BrunatHimstedt2010}, Lemma 3.3. (That result is stated only for the case when $\mbf G$ is split, but the proof does not require that assumption.) 

It is easy to see, using Clifford theory, that $\Irr_{p'}(B)=\Irr(B \di 1_{[U,U]})$. (This is true for any finite group $B$ with a normal Sylow $p$-subgroup $U$.) 
Moreover, if $\mbf G$ is not one of the exceptions (i) and (ii), then $[U,U]=U_2$. For split $\mbf G$, this was proved by Howlett (\cite{Howlett1974}, Lemma 7). This fact appears to be known for non-split $\mbf G$ as well (cf.\ \cite{DigneMichelBook}, discussion following Definition 14.27). In any case, one can prove it by a similar method to that used in~\cite{Howlett1974}, using the description of ``root subgroups'' of $U$ in~\cite{CarterSGLT}, Section 13.6. 

So we have $\Irr_{p'}(B)=\Irr(B\di 1_{U_2})$. It remains to show that $\Irr_s (B)=\Irr(B \di 1_{U_2})$ (without exceptions). Certainly, every constituent $\chi$ of any $\xi_z$ has a kernel containing $U_2$ (see the paragraph following~\eqref{xi_z}). 

Let $\chi\in \Irr(B\di 1_{U_2})$. Let $J$ be the set of orbits $\Om\in \Pi/\tau$ such that $\bar X_{\Om} \subseteq \ker\tilde{\chi}$.
Let $\lambda$ be an irreducible constituent of $\Res^{\bar B}_{\bar U} \chi$. Then $\bar V_J \subseteq \ker \lambda$, and 
 $\Def^{\bar U}_{\ol{U\cap L_J}} \lda$ is a deflation of some regular character of $U\cap L_J$. Thus $\lambda$ is $T$-conjugate to 
 $\th^J_y$ for some $y\in H^1 (F, Z(\mbf L_J))$, and we have $y=h_J (z)$ for some $z\in H^1 (F, Z(\mbf G))$ because $h_J$ is surjective. 
So $\tilde{\chi}$ is a constituent of $\Ind_{\bar U}^{\bar B} \th^J_y$. On the other hand, suppose that $\tilde{\chi}$ is a constituent of 
$\Ind_{\bar U}^{\bar B} \th^{J'}_{y'}$ for some $J'\subseteq \Pi/\tau$ and some $y'\in H^1 (F, Z(\mbf L_{J'}))$. Then, for each $\Om\in\Pi/\tau$, we have $\bar X_{\Om} \subseteq \ker\tilde{\chi}$ if and only if $\bar X_{\Om} \subseteq \ker \th^{J'}_{y'}$, that is, if and only if $\Om \in J'$. Therefore, $J'=J$. Thus, for the chosen $z$, the character $\tilde{\chi}$ is a constituent of precisely one summand on the right-hand side of~\eqref{barxi_z}, and the result follows. 
\end{proof}

%We will call a generalised character $\chi$ of a group $L$ \emph{multiplicity-free} if $\lan \chi,\phi \ran$ is $0$ or $\pm 1$ for each $\phi\in \Irr(L)$; we %will say that $\phi$ is a constituent of $\chi$ if $\lan \chi,\phi \ran \ne 0$. An irreducible character of $\mbf G$ has 

Now we state and prove two lemmas analogous to each other, one for the virtual characters $\Xi_z$, the other for $\xi_z$. Using these, we will construct a correspondence between $\pm \Irr_s (G)$ and $\pm \Irr_s (B)$ (see Theorem~\ref{liesemisimple}) which will subsequently lead to a proof of Theorems~\ref{liesplit} and~\ref{lienonsplit}.

\begin{lem}\label{lemG} For all $z_0,z_1\in H^1 (F, Z(\mbf G))$ we have:
\begin{enumerate}[(i)]
\item $\Xi_{z_0}$ is multiplicity-free;
\item for each irreducible constituent $\chi$ of $\Xi_{z_0}$ there is a subgroup $A$ of $H^1 (F,Z(\mbf G))$ such that $\chi$ is a constituent of $\Xi_{z_0 z}$ if and only if $z\in A$;
\item if $\chi$ is an irreducible constituent of both $\Xi_{z_0}$ and $\Xi_{z_1}$ then $\lan \chi, \Xi_{z_0} \ran = \lan \chi, \Xi_{z_1}\ran$.
\end{enumerate}
\end{lem}

\begin{proof} As $D_{\mbf G}$ is an isometry of $\Cal C(G)$ onto itself, parts (i), (ii) and (iii) are equivalent to the same statements with virtual characters $\Xi_z$ replaced by $\Gamma_z$. Thus (i) holds by~\cite{DigneMichelBook}, Theorem 14.30 (see also~\cite{CarterFGLT}, Theorem 8.1.3), and (iii) becomes trivial as $\Gamma_z$ is an actual character for each $z$. Part (ii) follows from results of~\cite{DigneLehrerMichel1992}, namely Proposition 3.12 (and its proof) and Corollary 3.14. (The subgroups $A$ are described explicitly in the proof of the quoted proposition.)
\end{proof}

\begin{lem}\label{lemB} For all $z_0,z_1\in H^1 (F, Z(\mbf G))$ we have:
\begin{enumerate}[(i)]
\item $\xi_{z_0}$ is multiplicity-free;
\item for each irreducible constituent $\chi$ of $\xi_{z_0}$ there is a subgroup $A$ of $H^1 (F,Z(\mbf G))$ such that $\chi$ is a constituent of $\xi_{z_0 z}$ if and only if $z\in A$;
\item if $\chi$ is an irreducible constituent of both $\xi_{z_0}$ and $\xi_{z_1}$ then $\lan \chi, \xi_{z_0} \ran = \lan \chi, \xi_{z_1}\ran$.	
\end{enumerate}
\end{lem}

\begin{proof} Since $\xi_z=\pi_{1_{U_2}} \xi_z$ for every $z$, 
we may replace $B$ with $\bar B$ and $\xi_z$ with $\tilde{\xi_z}$ (for each $z$) in the statement of the lemma, thus obtaining an equivalent formulation.
 Suppose $\chi\in \Irr(\bar B)$ is a constituent of $\tilde \xi_{z_0}$ for some $z_0\in H^1 (F, Z(\mbf G))$. 
 Define $J_{\chi}$ to be the set of orbits $\Om\in \Pi/\tau$ such that $\bar X_{\Om} \subseteq \ker\chi$. 
 Suppose $\th^J_{h_J (z_0)}$ is a constituent of $\Res^{\bar B}_{\bar U} \chi$, where $J\subseteq \Pi/\tau$.
As we observed in the proof of Lemma~\ref{pdeqss}, it follows that $J=J_{\chi}$. 
Let $E=\Stab_{\bar B} (\th^J_{h_J (z_0)})$, and let $\zeta\in\Irr(E)$ be the Clifford correspondent of $\chi$ with respect to $\th^J_{h_J (z_0)}$, so that
$\chi=\Ind_E^{\bar B} \zeta$ and $\zeta \in \Irr(E \di \th^J_{h_J (z_0)})$. Since $E/\bar U$ is abelian and has order coprime to that of $\bar U$, it follows from standard results of character theory (\cite{IsaacsBook}, Theorem 6.16 and Corollary 6.28) that 
$\Res^E_{\bar U} \zeta = \th^J_{h_J (z_0)}$, whence  
$\lan \chi, \Ind_{\bar U}^{\bar B} \th^J_{h_J (z_0)} \ran=1$. Thus, by~\eqref{barxi_z}, $\chi$ occurs in $\tilde \xi_z$ with multiplicity $(-1)^{|J_{\chi}|}$. Since $J_{\chi}$ does not depend on $z_0$, this proves (i) and (iii). 

For (ii), consider $z_1=z_0 z\in H^1 (F, Z(\mbf G))$. 
 As we have seen above, if $\th^{J'}_{h_{J'}(z_1)}$ is a constituent of $\Res^{\bar B}_{\bar U} \chi$ for some $J'\subseteq \Pi/\tau$ then 
 $J'=J_{\chi}=J$. Moreover, by Clifford theory, $\th^{J}_{h_{J}(z_1)}$ is a constituent of $\Res^{\bar B}_{\bar U} \chi$ if and only if $\th^{J}_{h_J (z_1)}$ is $T$-conjugate to $\th^{J}_{h_J (z_0)}$. As the characters $\th^J_y$, $y\in H^1 (F,Z(\mbf L_J))$,  lie in distinct $T$-orbits, this happens precisely when $h_J (z_1)=h_J(z_0)$, i.e.\ 
 $h_J (z)=1$. So, by~\eqref{barxi_z}, the subgroup $A=\ker h_{J_{\chi}}$ satisfies the requirement of (ii).
\end{proof}

It is well-known that $Z(G)=C_T (u) = Z(\mbf G)^F$ for every $u\in U_{\reg}$ (see~\cite{DigneMichelBook}, Lemma 14.15).
For each $\nu\in \Irr(Z(G))$ write $\Xi_{z,\nu}=\pi_{\nu} \Xi_z$ and $\xi_{z,\nu}=\pi_{\nu} \xi_{z}$. 
Then we have $\Xi_{z,\nu}= \Ind_B^G \xi_{z,\nu}$ and $\xi_{z,\nu} = \Res^G_B \Xi_{z,\nu}$ due to the corresponding identities for $\Xi_z$ and $\xi_z$.

\begin{thm}\label{liesemisimple} Suppose $H^1 (F, Z(\mbf G))$ is cyclic. 
Then there exists a signed bijection $F\colon \pm \Irr_{s} (G) \ra \pm \Irr_{s} (B)$ such that 
$\lan F(\chi), \xi_{z,\nu} \ran = \lan \Res^G_B \chi, \xi_{z,\nu} \ran$ 
for all $z\in H^1 (F, Z(\mbf G))$, $\nu\in \Irr(Z(G))$ and $\chi\in \Irr_{s}(G)$.
\end{thm}

\begin{proof}
Fix $\nu \in \Irr(Z(G))$.
Lemmas~\ref{lemG} and~\ref{lemB} still hold if we replace $\Xi_{z}$ with $\Xi_{z,\nu}$ and $\xi_z$ with $\xi_{z,\nu}$ for each $z$. This follows from the observation that if $\chi$ is a constituent of both $\Xi_{z,\nu}$ and $\Xi_{z'}$ for some $z,z'\in H^1 (F,Z (\mbf G))$ then $\chi$ must be a constituent of $\Xi_{z',\nu}$ and from a similar statement for characters of $B$. Let $X$ be a non-empty subset of $H^1 (F, Z (\mbf G))$ and set
\begin{eqnarray*}
 \Cal G (X) & = & \{ \chi \in \pm \Irr(G) \mid \lan \chi, \Xi_{z,\nu} \ran = 1 \} \quad \text{and}  \\
 \Cal B (X) & = & \{ \phi \in \pm \Irr(B) \mid \lan \phi, \xi_{z,\nu} \ran = 1 \}.
\end{eqnarray*}
 Let $z_0 \in X$ and define $Y$ to be the subgroup of $H^1 (F, Z(\mbf G))$ generated by the elements $z$ such that $z_0 z\in X$. 
Since $H^1 (F,Z(\mbf G))$ is cyclic, $Y$ must be cyclic too, and we pick a generator $z$ of $Y$. It follows from Lemma~\ref{lemG}(ii),(iii) and Lemma~\ref{lemB}(ii),(iii) that
$\Cal G (X) = \Cal G (\{ z_0, z_0z \})$ and $\Cal B (X) = \Cal B (\{z_0, z_0 z\})$. Therefore
$$
|\Cal G(X)| = \lan \Xi_{z_0,\nu}, \Xi_{z_0 z,\nu} \ran = \lan \xi_{z_0,\nu}, \xi_{z_0 z,\nu} \ran = |\Cal B (X)|.
$$
Here the first equality follows from Lemma~\ref{lemG}(i),(iii), the third from Lemma~\ref{lemB}(i)(iii), and the second from Frobenius reciprocity. 

Let $\bar{\Cal G} (X)=\Cal G(X) - \bigcup_{Y\supsetneq X} \Cal G(Y)$ and $\bar{\Cal B}(X) = \Cal B(X) - \bigcup_{Y\supsetneq X} \Cal B (Y)$. Using the inclusion-exclusion formula, we see that $|\bar{\Cal G}(X)| = |\bar{\Cal B} (X)|$ for each non-empty $X\subseteq H^1 (F,Z(\mbf G))$. 
Let $F\colon \bigsqcup_{X} \bar{\Cal G}(X) \ra \bigsqcup_{X} \bar{\Cal B}(X)$ 
be a bijection which maps $\bar {\Cal G}(X)$ onto $\bar {\Cal B}(X)$ for each $X$. 
We can extend $F$ to a signed bijection $F\colon \pm \Irr_{s} (G \di \nu) \ra \pm \Irr_{s} (B \di \nu)$ by 
setting  $F(-\chi)=-F(\chi)$ for each $\chi$ on which $F$ is not already defined. Then we have 
$$
\lan \xi_{z,\nu}, F(\chi) \ran = \lan \Xi_{z,\nu}, \chi \ran = \lan \xi_{z,\nu}, \Res^G_B \chi \ran 
$$ 
for each $\chi\in \pm \Irr_{s}(G \di \nu)$ and each $z\in H^1 (F,Z(\mbf G))$. 
\end{proof}

\begin{lem}\label{zerolem} Suppose $\zeta\in \Cal C(B)$ satisfies $\lan \zeta, \xi_{z,\nu} \ran=0$ for each $z\in H^1 (F,Z(\mbf G)$ and $\nu\in\Irr(Z(G))$. Then 
$\pi_{1_{U_2}} (\zeta)$ vanishes on all $g\in B$ such that $g_p\in U_{\reg}$. Moreover, if $p$ is good for $\mbf G$ then $\zeta(g)=0$ 
whenever $g_p\in U_{\reg}$.
\end{lem}

\begin{proof}
Let $\nu\in\Irr(Z(G))$. By the hypothesis, $\lan \pi_{\nu} \zeta,\xi_z \ran=0$ for all $z\in H^1 (F, Z(\mbf G))$. By~\cite{DigneMichelBook}, Proposition 14.26 (and the discussion preceding it), the $\bar B$-conjugacy classes of $U/U_2$ are parametrised by $H^1 (F, Z(\mbf G))$. Let 
$u_z\in U/U_2$ be a representative of the class corresponding to $z\in H^1 (F, Z(\mbf G))$. (The parametrisation depends on the choice of $u_1$.)  
By~\cite{DigneMichelBook}, Theorem 14.35, for every $\chi\in \Irr(B\di 1_{U_2})$ and each $z\in H^1 (F, Z(\mbf G))$, the number $\tilde{\chi}(u_z)$ can be expressed as a $K$-linear combination of the scalar products $\lan \chi, \Xi_{z'} \ran$, $z'\in H^1 (F,Z(\mbf G))$. Therefore, for all $u\in U_{\reg}$, we have
$(\pi_{\nu} \pi_{1_{U_2}} \zeta) (u)=0$, whence $(\pi_{\nu} \pi_{1_{U_2}} \zeta)(tu)=\nu(t)(\pi_\nu \pi_{1_{U_2}} \zeta (u))=0$ for all $t\in Z(G)$. Taking the sum over all $\nu\in\Irr(Z(G))$, we see that $\pi_{1_{U_2}}\zeta (tu)=0$ for all $t$, proving the first statement.

Now suppose $p$ is good for $\mbf G$. Then, by~\cite{DigneMichelBook}, Corollary 14.38, for every $u\in U_{\reg}$, $(\pi_{\nu}\zeta)(u)$ is a $K$-linear combination of the scalar products $\lan \pi_{\nu}\zeta, \Xi_{z'} \ran$ with $z'\in H^1(F,Z(\mbf G))$,
whence $(\pi_{\nu}\zeta)(u)=0$. The proof of the second statement of the lemma is concluded as in the previous paragraph.
\end{proof}

\begin{proof}[Proof of Theorem~\ref{liesplit}]
By Lemma~\ref{pdeqss}, $\Irr_{p'} (G)=\Irr_s (G)$ and $\Irr_{p'}(B)=\Irr_s (B)$. Let 
$F\colon \pm \Irr_{p'} (G) \ra \Irr_{p'} (B)$ be a signed bijection given by Theorem~\ref{liesemisimple}. Then, by Lemma~\ref{zerolem}, for each $\chi\in\pm \Irr_{p'}(G)$, the difference 
$F(\chi)-\Res^G_B \chi$ vanishes on all elements $g\in B$ such that $g_p\in U_{\reg}$. Property (IRC-Syl) now follows by Proposition~\ref{IforB}.

Let $\Cal S=\Cal S(G,U,B)$ and
suppose $\phi\in \Irr^p (B)$. Then $U_2 \nsubseteq \ker\phi$ by Lemma~\ref{pdeqss}, whence $\phi\in \Cal I(B,U,\Cal S)$ by Proposition~\ref{redlin}. Hence $\Ind_B^G \phi \in \Cal I(G,B,\Cal S)$, and (pInd-Syl) follows.
\end{proof}

\begin{proof}[Proof of Theorem~\ref{lienonsplit}]
We begin by proving (pRes-Syl). 
Let $\eta\in\Irr^p (B)$ and write $\zeta=\Res^G_B \eta$. By Lemma~\ref{pdeqss}, $\lan \eta, \Xi_{z,\nu}\ran =0$, 
whence $\lan \zeta, \xi_{z,\nu} \ran =0$ (by Frobenius reciprocity), for all $z\in H^1 (F,Z(\mbf G))$ and $\nu \in \Irr(Z(G))$. Thus 
$\zeta'=\pi_{1_{U_2}} \zeta$ vanishes on all $g\in B$ such that $g_p\in U_{\reg}$, by Lemma~\ref{zerolem}. Applying Lemma~\ref{dirprod} to $\Def_{B/U_2}^B \zeta'$, we deduce that $\zeta'\in \Cal I(G,B,\Cal S)$ (where $\Cal S=\Cal S(G,U,B)$, as usual). On the other hand, all irreducible constituents of $\zeta-\zeta'$ have degrees divisible by $p$ by Lemma~\ref{pdeqss}. So $\zeta \in \Cal I(G,U,\Cal S) + \Cal C^p (B)$, and (pRes-Syl) holds.

To prove (WIRC-Syl), consider a signed bijection $F\colon \pm \Irr_{p'}(G) \ra \pm \Irr_{p'} (B)$ given by Theorem~\ref{liesemisimple} (considered together with Lemma~\ref{pdeqss}). Let $\chi\in\pm \Irr_{p'}(G)$, so that $\g=F(\chi)-\Res^G_B \chi$ satisfies $\lan \g, \xi_{z,\nu} \ran=0$ for all $z$ and $\nu$. By the argument of the previous paragraph it follows that $\g\in \Cal C^p (B)+\Cal I(B,U,\Cal S)$. So $F$ satisfies the requirement of (WIRC-Syl). 
\end{proof}

%To conclude the section, we remark that the arguments above imply
%a generalisation of a result due to Green, Lehrer and Lusztig (\cite{GreenLehrerLusztig1975}, Lemma 1; see also~\cite{CarterFGLT}, Proposition 8.3.4) 
%in the case of split groups. (In the original lemma the congruence below is proved only modulo $p$, though for not necessarily split groups.)

%\begin{cor} Let $\mbf G$ be a connected reductive group with a connected centre. Assume that $\mbf G$ is defined over $\mF_q$ with corresponding Frobenius endomorphism $F$ and, moreover, that $\mbf G$ is split. Assume also that $q$ is a power of a prime $p$ that is good for $\mbf G$. Let $u$ be a regular unipotent element of $G=\mbf G^F$. Then 
%$$
%\chi(u) \equiv \chi(1) \mod q
%$$
%for every $\chi\in \Irr(G)$.
%\end{cor}

%\begin{proof}
% Let $\zeta=\Res^G_U \chi$. By~\cite{CarterFGLT}, Proposition 8.3.3, the value $\zeta(u)$ is the same for all $u\in U_{\reg}$ and is equal to $1$, $-1$ or $0$. Hence there is $m\in \mZ$ such that $\zeta-m1_U$ vanishes on $U_{\reg}$. By (the proof of) Proposition~\ref{redlin}, it follows that 
%$\zeta-m1_U \in \Cal I(U,U,\Cal S)$ where $\Cal S$ is as in Proposition~\ref{redlin}. Since every element of $\Cal S$ has index at least $q$ in $U$, it follows that $\zeta(1)\equiv m \pmod q$, and we already know that $\zeta(u)=m$. 
%\end{proof}

\section{Specific cases}\label{calc}

In this section we record the outcome of testing the properties of Section~\ref{blarb} for certain ``small'' groups 
using GAP~\cite{GAP4} and the data provided in~\cite{AtlasQMUL}. The GAP code used for these checks is available on the author's website~\cite{EvseevWebsite}.
For each triple $(G,p,P)$ listed in the tables below, properties (IRC), (pRes), (pInd) and (WIRC), with respect to the normaliser $H=N_G (P)$, have been tested. The latter three properties hold in all these
 cases (and hence (WIRC*) does too), and the tables state, in each case, whether or not (IRC) holds.  We do not list any cases where $P$ is cyclic (due to Theorem~\ref{cyclicthm}) or normal in $G$. We use the shorthand 
$\Cal Q_1=\Cal C(H,P)/\Cal I(H,P,\Cal S)$ and $\Cal Q_2 = \Cal C(H,P)/(\Cal C^p (H,P)+ \Cal I(H,P,\Cal S))$ where $H=N_G (P)$ and $\Cal S=\Cal S(G,P,N_G (P))$, and we give the structure of these quotients (as abelian groups) in the tables. In a sense, this structure gives an indication as to how much information is encoded in the properties we consider. 
The decomposition of $\Cal Q_1$ and $\Cal Q_2$ into components corresponding to different blocks is indicated by square brackets. (We omit the square brackets if $G$ has only one block with $P$ as a defect group.)
We note that in a few cases a component of $\Cal Q_1$ or $\Cal Q_2$ corresponding to a block is isomorphic to $\mZ/p$. In such a situation property (IRC-Bl) or (WIRC-Bl) (respectively) does not give any more information than Conjecture~\ref{conjINgen} of Isaacs and Navarro.

All possible non-cyclic defect groups of blocks $S_n$ and $A_n$ have been checked for $n\le 11$ (for all primes), and properties (IRC), (pRes) and (pInd) hold in all these cases. When $G$ is a symmetric or an alternating group, we only list cases where $P$ is a Sylow subgroup for the sake of brevity.

In the cases where $P$ is not Sylow in Tables~\ref{tablespor} and~\ref{tablelie}, each $P$ is uniquely determined up to $G$-conjugation by its isomorphism type (which is given) and the fact that $P$ is a defect group of a block. 
In Table~\ref{tablelie}, $D_8$ and $Q_8$ denote the dihedral and quaternion groups of order $8$, and $\lsa{2}F_4 (2)'$ is the Tits group (i.e.\ the derived subgroup of $\lsa{2}F_4 (2)$). 

Note that (IRC) holds in Table~\ref{tablelie} in all cases except those where $G$ is a twisted group and $p$ is the defining characteristic.
(Also, (IRC-Syl) holds when  $G=\PSU_4 (2)\simeq \PSp_4 (3)$ for $p=2$ and when $G=\PSU_4(3)$ and $p=3$.)

 Finally, we remark that the sporadic groups $HS$ and $McL$ contain subgroups isomorphic to $\PSU_3 (5)$, and $J_2$ has a subgroup isomorphic to $\PSU_3 (3)$ (see~\cite{AtlasQMUL}). Perhaps, these observations might ``explain'' the three failures of (IRC) in Table~\ref{tablespor}.

\renewcommand{\arraystretch}{1.1}

\begin{table}[hbtp]
 \caption{Symmetric and alternating groups  ($P$ is a Sylow $p$-subgroup)}
\label{tablesym}
\centering
%\begin{minipage}[t]{0.5\linewidth} 
\subfigure{
\centering
\begin{tabular}{|cIc|c|c|}
\hline
$G$, $p$ & $\Cal Q_1$ & (IRC) & $\Cal Q_2$\\
\whline
$S_4$, $2$ & $\mZ^2$ & Yes & $\mZ^2$ \\
\hline
$S_5$, $2$ & $\mZ$ & Yes & $\mZ$ \\
\hline
$S_6$, $2$ &  $\mZ^2$ & Yes & $\mZ^2$ \\
\hline
$S_6$, $3$ &  $\mZ^4$ & Yes & $\mZ^4$ \\
\hline
$S_7$, $2$ & $\mZ$ & Yes & $\mZ$ \\
\hline
$S_7$, $3$ & $\mZ^2$ & Yes & $\mZ^2$ \\
\hline
$S_8$, $2$ & $\mZ \oplus \mZ$ & Yes & $\mZ$ \\
\hline
$S_8$, $3$ & $[\mZ^2] \oplus [\mZ^2]$ & Yes & $[\mZ^2] \oplus [\mZ^2]$ \\
\hline
$S_9$, $2$ & $\mZ$ & Yes & $\mZ$ \\
\hline
$S_9$, $3$ &  $\mZ$ & Yes & $\mZ$ \\
\hline
$S_{10}$, $2$ &  $\mZ^2$ & Yes & $\mZ^2$ \\
\hline
$S_{10}$, $3$ & $\mZ$ & Yes & $\mZ$ \\
\hline
$S_{10}$, $5$ &  $\mZ^2$ & Yes & $\mZ^2$ \\
\hline
$S_{11}$, $2$ &  $\mZ$ & Yes & $\mZ$ \\
\hline
$S_{11}$, $3$ & $[\mZ] \oplus [\mZ]$ & Yes & $[\mZ] \oplus [\mZ]$ \\ 
\hline
$S_{11}$, $5$ & $\mZ^2$ & Yes & $\mZ^2$ \\
\hline
$S_{12}$, $2$ & $\mZ^2$ & Yes & $\mZ^2$ \\
\hline
$S_{12}$, $3$ & $\mZ^3$ & Yes & $\mZ^3$ \\
\hline
$S_{12}$, $5$ & $[\mZ^2] \oplus [\mZ^2]$ & Yes & $[\mZ^2] \oplus [\mZ^2]$ \\
\hline
\end{tabular}
}
%\end{minipage}
\hspace{0.5cm}
\subfigure{
%\begin{minipage}[t]{0.5\linewidth} 
\centering
\begin{tabular}{|cIc|c|c|}
\hline
$G$, $p$ & $\Cal Q_1$ & (IRC) & $\Cal Q_2$\\
\whline
$A_5$, $2$ & $\mZ$ & Yes & $\mZ$ \\
\hline
$A_6$, $2$ & $\mZ$ & Yes & $\mZ$ \\
\hline
$A_6$, $3$ & $\mZ^2$ & Yes & $\mZ^2$ \\
\hline
$A_7$, $2$ &  $\mZ$ & Yes & $\mZ$ \\
\hline
$A_7$, $3$ & $\mZ$ & Yes & $\mZ$ \\
\hline
$A_8$, $2$ & $\mZ$ & Yes & $\mZ$ \\
\hline
$A_8$, $3$ & $\mZ^2$ & Yes & $\mZ^2$ \\
\hline
$A_9$, $2$ &  $\mZ/2$ & Yes & $\mZ/2$ \\
\hline 
$A_9$, $3$ &  $\mZ^2$ & Yes & $\mZ^2$ \\
\hline
$A_{10}$, $2$ & $\mZ$ & Yes & $\mZ$ \\
\hline
$A_{10}$, $3$ & $\mZ^2$ & Yes & $\mZ^2$ \\ 
\hline
$A_{10}$, $5$ & $\mZ$ & Yes & $\mZ$ \\
\hline
$A_{11}$, $2$ &  $\mZ$ & Yes & $\mZ$ \\
\hline
$A_{11}$, $3$ &  $\mZ$ & Yes & $\mZ$ \\
\hline
$A_{11}$, $5$ &  $\mZ$ & Yes & $\mZ$ \\
\hline
$A_{12}$, $2$ & $\mZ^2$ & Yes & $\mZ^2$ \\
\hline
$A_{12}$, $3$ & $\mZ^3$ & Yes & $\mZ^3$ \\
\hline
$A_{12}$, $5$ & $\mZ^2$ & Yes & $\mZ^2$ \\
\hline
\end{tabular}
}
% \end{minipage}
\end{table}

%For each pair $G,p$ listed in the table, with the exception of $G=He$ and $p=2,3$, the table contains a row for every non-cyclic $P\le G$ that is a defect group of a $p$-block.

\begin{table}[hbtp]
\caption{Some sporadic simple groups} 
\label{tablespor}
\centering
\begin{tabular}{|c|cIc|c|c|}
\hline
$G$, $p$ & $P$ & $\Cal Q_1$ & (IRC) & $\Cal Q_2$\\
\whline
$M_{11}$, $2$ & Sylow & $\mZ^2$ & Yes & $\mZ$ \\ 
\hline
$M_{11}$, $3$ & Sylow & $\mZ^2$ & Yes & $\mZ^2$ \\
\hline
\multirow{2}{12mm}{$M_{12}$, $2$}
   & Sylow & $\mZ^2$ & Yes & $\mZ^2$ \\
\cline{2-5}
  & $C_2 \times C_2$ & $\mZ/2$ & Yes & $\mZ/2$ \\
\hline
$M_{12}$, $3$ & Sylow & $\mZ/3$ & Yes & $\mZ/3$ \\ 
\hline
$M_{22}$, $2$ & Sylow & $\mZ$ & Yes & $\mZ$ \\
\hline
$M_{22}$, $3$ & Sylow & $\mZ/3$ & Yes & $\mZ/3$ \\
\hline
$M_{23}$, $2$ & Sylow & $\mZ$ & Yes & $\mZ$ \\
\hline
$M_{23}$, $3$ & Sylow & $\mZ/3$ & Yes & $\mZ/3$ \\
\hline
$M_{24}$, $2$ & Sylow & $\mZ$ & Yes & $\mZ$ \\
\hline
$M_{24}$, $3$ & Sylow & $\mZ/3$ & Yes & $\mZ/3$ \\
\hline
\multirow{2}{12mm}{$HS$, $2$} 
 & Sylow & $\mZ^2$ & Yes & $\mZ$ \\
\cline{2-5}
 & $C_2 \times C_2$ & $\mZ/2$ & Yes & $\mZ/2$ \\
\hline
$HS$, $3$ & Sylow & $[\mZ/3] \oplus [\mZ/3]$ & Yes & $[\mZ/3] \oplus [\mZ/3]$ \\
\hline
$HS$, $5$ & Sylow & $\mZ^5$ & No & $\mZ$ \\
\hline
$J_1$, $2$ & Sylow & $\mZ/4$ & Yes & $\mZ/4$ \\
\hline
\multirow{2}{12mm}{$J_2$, $2$}
& Sylow & $\mZ$ & Yes & $\mZ$ \\
\cline{2-5}
 & $C_2 \times C_2$ & $\mZ$ & Yes & $\mZ$ \\
\hline
$J_2$, $3$ & Sylow & $\mZ/9$ & No & $\mZ/3$ \\ 
\hline
$J_2$, $5$ & Sylow & $\mZ^4$ & Yes & $\mZ^4$ \\
\hline
$McL$, $2$ & Sylow & $\mZ$ & Yes & $\mZ$ \\
\hline
$McL$, $3$ & Sylow & $\mZ^2$ & Yes & $\mZ^2$ \\
\hline
$McL$, $5$ & Sylow & $\mZ^7$ & No & $\mZ$ \\
\hline
$He$, $5$ & Sylow & $\mZ/5$ & Yes & $\mZ/5$ \\
\hline
$He$, $7$ & Sylow & $\mZ^4$ & Yes & $\mZ^4$ \\
\hline
\end{tabular}
\end{table}
\medskip

\begin{table}[hbtp]
\caption{Some groups of Lie type}
\label{tablelie} 
\centering
\begin{tabular}{|c|cIc|c|c|}
\hline
$G$, $p$ & $P$ & $\Cal Q_1$ & (IRC) & $\Cal Q_2$\\
\whline
%$\SL_2 (8)$, $3$ & Sylow & $\mZ^4$ & Yes & $\mZ^4$ \\ cyclic
%\hline
$\SL_2 (11)$, $2$ & Sylow & $\mZ/2$ & Yes & $\mZ/2$ \\
\hline
$\SL_2 (13)$, $2$ & Sylow & $\mZ/2$ & Yes & $\mZ/2$ \\
\hline
$\SL_2 (17)$, $2$ & Sylow & $\mZ^6$ & Yes & $\mZ$ \\
\hline
$\SL_2 (19)$, $2$ & Sylow & $\mZ/2$ & Yes & $\mZ/2$ \\
\hline
$\SL_3 (3)$, $2$ & Sylow & $\mZ^2$ & Yes & $\mZ$ \\
\hline
$\SL_3 (5)$, $2$ & Sylow & $\mZ/2 \oplus \mZ/2$ & Yes & $\mZ/2 \oplus \mZ/2$ \\
\hline
$\PSL_3 (8)$, $7$ & Sylow & $\mZ/7\oplus \mZ^5$ & Yes & $\mZ/7 \oplus \mZ^5$ \\
\hline
$\PSL_4 (2)$, $3$ & Sylow & $\mZ^2$ & Yes & $\mZ^2$ \\
\hline
$\PSL_4 (3)$, $2$ & Sylow & $\mZ$ & Yes & $\mZ$ \\ % done
\hline
$\PSL_5 (2)$, $3$ & Sylow & $\mZ/3$ & Yes & $\mZ/3$ \\ % done
\hline
$\PSp_4 (3)$, $2$ & Sylow & $\mZ$ & Yes & $\mZ$ \\
\hline
\multirow{3}{20mm}{$\PSp_4 (5)$, $2$}
   & Sylow & $\mZ/2$ & Yes & $\mZ/2$ \\ 
\cline{2-5}
   & $Q_8$ & $\mZ/2 \oplus \mZ$ & Yes & $\mZ/2 \oplus \mZ$ \\
\cline{2-5}
   & $D_8$ & $\mZ$ & Yes & $\mZ$ \\ % done
\hline
$\PSp_4 (5)$, $3$ & Sylow & $[\mZ/3]\oplus [\mZ/3]$ & Yes & $[\mZ/3]\oplus [\mZ/3]$ \\ % done
\hline
$\PSp_6 (2)$, $3$ & Sylow & $\mZ$ & Yes & $\mZ$ \\ % done
\hline
$\PSp_6 (3)$, $2$ & Sylow & $\mZ^4$ & Yes & $\mZ^4$ \\ % not done
\hline
$\PSU_3 (3)$, $2$ & Sylow & $\mZ^2$ & Yes & $\mZ^2$ \\ % done   $G=G_2 (2)'$
\hline
$\PSU_3 (3)$, $3$ & Sylow & $\mZ^5$ & No & $\mZ$ \\ % done
\hline
$\PSU_3 (4)$, $2$ & Sylow & $\mZ^6$ & No & $\mZ$ \\ % done
\hline
$\PSU_3 (4)$, $5$ & Sylow & $\mZ/5\oplus \mZ^2$ & Yes & $\mZ/5\oplus \mZ^2$ \\ % done
\hline
\multirow{2}{20mm}{$\PSU_3 (5)$, $2$} 
   & Sylow & $\mZ^2$ & Yes & $\mZ$ \\
\cline{2-5}
   & $C_2 \times C_2$ & $\mZ/2$ & Yes & $\mZ/2$ \\  % done
\hline
$\PSU_3 (5)$, $3$ & Sylow & $\mZ/3$ & Yes & $\mZ/3$ \\ % done
\hline
$\PSU_3 (5)$, $5$ & Sylow & $\mZ^5$ & No & $\mZ^3$ \\ % done
\hline
$\PSU_3 (7)$, $2$ & Sylow & $(\mZ/2)^4$ & Yes & $(\mZ/2)^4$ \\ % done
\hline
$\PSU_3 (7)$, $7$ & Sylow & $\mZ^9$ & No & $\mZ$ \\ % done
\hline
$\PSU_4 (3)$, $2$ & Sylow & $\mZ$ & Yes & $\mZ$ \\ % done
\hline
$\PSU_4 (3)$, $3$ & Sylow & $\mZ^4$ & Yes & $\mZ^4$ \\ % done
\hline
$\PSU_5 (2)$, $2$ & Sylow & $\mZ^4$ & No & $\mZ$ \\ % done
\hline
$\PSU_5 (2)$, $3$ & Sylow & $\mZ^6$ & Yes & $\mZ^6$ \\ % done
\hline
$\lsa{2}B_2 (8)$, $2$ & Sylow & $\mZ^3$ & No & $\mZ$ \\ % done
\hline
$\lsa{2}F_4 (2)'$, $2$ & Sylow & $\mZ^6$ & Yes & $\mZ^3$ \\ % done
\hline
$\lsa{2}F_4 (2)'$, $3$ & Sylow & $\mZ/3$ & Yes & $\mZ/3$ \\ % done
\hline
$\lsa{2}F_4 (2)'$, $5$ & Sylow & $\mZ^2$ & Yes & $\mZ^2$ \\ % done
\hline
\end{tabular}
\end{table}
\medskip

% ``Second'' block has cyclic defect for He, p=7.

% ``done'' comment means that all cases of non-cyclic non-normal defect groups have been included in the table for the given G and p. (Absence of such a comment does not imply anything as the comment has been introduced at a late stage.)

% Cases where not all blocks with non-cyclic and not normal defect groups have been accounted for:
% 1. S_{12}, p=2
% 2. He, p=2,3.

%\bibliographystyle{amsplain}
%\bibliography{fingroups}

\providecommand{\bysame}{\leavevmode\hbox to3em{\hrulefill}\thinspace}
\providecommand{\MR}{\relax\ifhmode\unskip\space\fi MR }
\providecommand{\MRhref}[2]{%
  \href{http://www.ams.org/mathscinet-getitem?mr=#1}{#2}
}
\providecommand{\href}[2]{#2}

\end{document}